\documentclass{article}
\usepackage{amsmath}
\usepackage{amssymb}
\usepackage{amsbsy}
\usepackage{amsgen}
\usepackage{amsopn}
\usepackage{amscd}
\usepackage{amstext}
\usepackage{amsthm}
\usepackage{graphicx}
\usepackage{tabularx}

\newtheorem{theorem}{Theorem}[section]
\newtheorem{proposition}[theorem]{Proposition}

\newtheorem{lemma}[theorem]{Lemma}

\theoremstyle{definition}
\newtheorem{definition}[theorem]{Definition}
\newtheorem{ex}[theorem]{Example}
\newtheorem{remark}[theorem]{Remark}
\newtheorem{problem}[theorem]{Problem}

\title{A colored Khovanov bicomplex}

\author{Noboru Ito\thanks{Address : Waseda Institute for Advanced Study, 1-6-1, Nishi Waseda, Shinjuku-ku, Tokyo 169-8050, JAPAN.  MSC: 57M27, 57M25.  Keywords: Khovanov homology, colored Jones polynomial, bicomplex.  }}

\begin{document}
\date{May 7, 2014}
%\date{September 10, 2010; \\ in revised form {\today}}
\maketitle

\begin{abstract}
In this note, we prove the existence of a tri-graded Khovanov-type bicomplex (Theorem \ref{bicomplex_theorem}).  The graded Euler characteristic of the total complex associated with this bicomplex is the colored Jones polynomial of a link.  The first grading of the bicomplex is a homological one derived from cabling of the link  (i.e., replacing a strand of the link with several parallel strands); the second grading is related to the homological grading of ordinary Khovanov homology; finally, the third grading is preserved by the differentials, and corresponds to the degree of the variable in the colored Jones polynomial.  In particular, we introduce a way to take a small cabling link diagram directly from a big cabling link diagram (Theorem \ref{small_cable_l}).  
\end{abstract}

\section{Introduction}\label{intro}
Throughout this paper we work in the smooth category.  A {\it{link}} is a closed one-dimensional submanifold of $\mathbb{R}^{3}$ and a {\it{knot}} is a one-component link.  The equivalence of links is given by an ambient isotopy.  A {\it{link diagram}} is a {\it{regular}} projection of the link to a plane, where each double point is specified by over-crossing and under-crossing branches.  The term {\it{regular projection}} here means a projection to a plane in which every singular point is a transversal double point.  

For every link diagram, we can naturally consider a {\it{framing}} that is a non-vanishing normal vector field on the link considered up to isotopy.  An isotopy class of framings contains those annihilated by the projection and those whose vectors are projected to nonzero vectors.  We choose the latter type and call them {\it{blackboard framings}} if their normal vectors are sufficiently short.  

Let $\textbf{m}$ $=$ $(m_{1}, m_{2}, \dots, m_{l})$ be a finite sequence of nonnegative integers.  The $(m_{1}, m_{2}, \dots, m_{l})$-cable of a link diagram $D$ of an oriented $l$-component link $L$ is the diagram $D^{(m_{1}, m_{2}, \dots, m_{l})}$, defined by replacing the $i$-th component of $D$ with $m_{i}$-oriented parallel strands, as shown in Figs. \ref{ori_cable} and \ref{cable_ex}.  This replacement procedure is performed by using the link diagram and its blackboard framing.  Every point is pushed in the direction of the normal vector $n$, where for every tangent vector $t$ that is tangential to $D$, the pair $(t, n)$ is positively oriented on the plane (this description is taken from \cite{viro, w}).  
\begin{figure}
\begin{center}
\begin{picture}(0,0)
\put(152,88){$n$}
\put(152,30){$2$}
\put(152,72){$\cdot$}
\put(152,57){$\cdot$}
\put(152,42){$\cdot$}
\put(152,11){$1$}
\end{picture}
\qquad\ 
\includegraphics[width=10cm, bb=0 0 400 152]{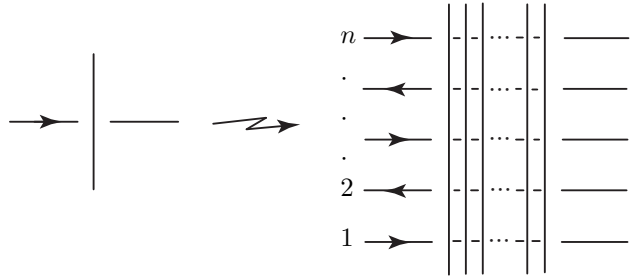}
\end{center}
\caption{Orientation of the cable of a diagram.  }\label{ori_cable}
\end{figure}
\begin{figure}
\begin{center}
\begin{picture}(0,0)
\put(15,45){I}
\put(63,90){I\!I}
\put(200,62){I}
\put(242,75){I\!I}
\put(196,35){$2$}
\put(241,37){$3$}
\put(262,37){$1$}
\put(249,83){$2$}
\put(179,45){$1$}
\qbezier(266,78)(261,81)(256,84)
\end{picture}
\quad
\includegraphics[width=10cm, bb=0 0 282.33 94]{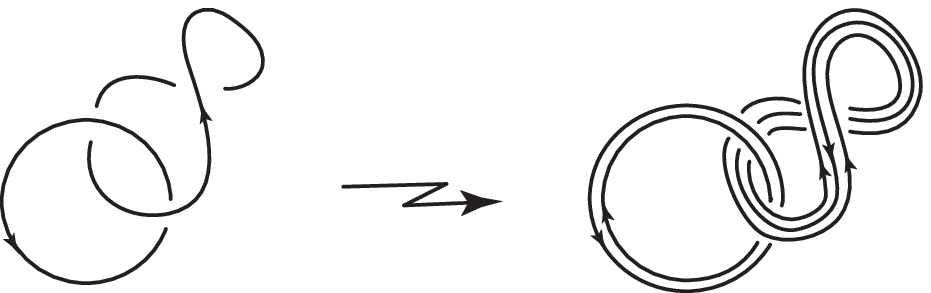}
\caption{A link diagram $D$ (left) and the $(2, 3)$-cable of $D$, denoted by $D^{(2, 3)}$ (right).}\label{cable_ex}
\end{center}
\end{figure}

For a $l$-component link $L$, the colored Jones polynomial $J_{\textbf{n}}(L)$ can be written as follows: 
\begin{equation}\label{colored_jones}
J_{\textbf{n}} (L) = \sum_{{\textbf{k}} = {\textbf{0}}}^{\lfloor \frac{{\textbf{n}}}{2} \rfloor}(-1)^{|{\textbf{k}}|}\begin{pmatrix} {\textbf{n}}-{\textbf{k}} \\ {\textbf{k}}\end{pmatrix}  J(D^{{\textbf{n}} - 2{\textbf{k}}})
\end{equation} where $\textbf{n}$ $=$ $(n_{1}, n_{2}, \dots, n_{l})$, $\textbf{k}$ $=$ $(k_{1}, k_{2}, \dots, k_{l})$ for arbitrary $n_{i}$, $k_{i} \in \mathbb{Z}_{\ge 0}$, $|{\textbf{k}}|$ $=$ $\sum_{i}k_{i}$, $\left( \begin{smallmatrix}{\textbf{n}} - {\textbf{k}} \\ {\textbf{k}}\end{smallmatrix} \right)$ $=$ $\prod_{i = 1}^{l} \left( \begin{smallmatrix} n_{i} - k_{i} \\ k_{i} \end{smallmatrix} \right)$, the sum $\sum_{\textbf{k} = \textbf{0}}^{\lfloor \textbf{n}/2 \rfloor}$ is the sum over all $0 \le k_i \le \lfloor {n_i}/2 \rfloor$ for all $i$, and $J(D)$ is the Jones polynomial of a link diagram $D$ that has $J(D^{\textbf{0}})$ $=$ $1$ (for more details of $J_{\textbf{n}}$, see \cite{kirbymelvin}, \cite{murakami}, and \cite{coloredkhovanov}).   

In \cite{khovanovjones}, Khovanov defined a bigraded chain complex whose graded Euler characteristic is the Jones polynomial and whose homology group, known as the Khovanov homology, is a link invariant.  In \cite{coloredkhovanov}, Khovanov made two proposals for an analogous homology theory for the colored Jones polynomial.  However, the first homology theory proposed in \cite{coloredkhovanov} was defined over only $\mathbb{Z}_2$, and the second one, for another normalization of the colored Jones polynomial, works only for knots.  Later, Beliakova and Wehrli developed the Khovanov theories of colored links over $\mathbb{Z}[1/2]$.  Furthermore, Mackaay and Turner \cite{mt} independently proposed another approach to constructing homology theories over $\mathbb{Z}_2$ for colored links, in which they calculated Bar-Natan's version of the Lee homology groups of knots and links.  

In \cite{bw, w}, Beliakova and Wehrli defined colored Khovanov brackets of colored links using {\textit{formal Khovanov brackets}}, which are universal objects introduced by Bar-Natan \cite{ba} that reconstruct the Khovanov and Lee homologies.  They pointed out that the colored Khovanov bracket is not a bicomplex and speculated as to whether it was indeed possible to construct a bicomplex.  

In these constructions, we require a direct definition of a coboundary operator between certain complexes of the link diagrams inducing the Khovanov-type homology whose graded Euler characteristic is the colored Jones polynomial of a link.  

With this background, the present note discusses the following problem.  
\begin{problem}\label{note_prob}
Is there a Khovanov-type bicomplex whose homological gradings are derived from cabling and from the homological grading of Khovanov homology, which can be viewed as a nontrivial categorification of the colored Jones polynomial?
\end{problem}
Here, the term {\textit{Khovanov-type bicomplex}} is used to mean that it possesses the properties of the Khovanov homology with respect to Euler characteristics, has a differential defined by the Frobenius calculus.  The following theorem provides a positive answer to Problem \ref{note_prob}.  
\begin{theorem}\label{bicomplex_theorem}
For each diagram $D$ of a link $L$, there exists a nontrivial tri-graded bicomplex $\{C^{k, i, j}_{\operatorname{\bf{n}}}(D), d', d''\}$ whose differentials preserve the grading $j$ such that
\begin{equation}\label{kh-eq1}
J_{\operatorname{\bf{n}}}(L) = \sum_{j} q^{j} \sum_{i, k} (-1)^{i+k}\, {\operatorname{rank}}\, H^{k}(H^{i}(C^{*, *, j}_{\operatorname{\bf{n}}}(D), d''), d').
\end{equation} 
\end{theorem}
The proof of Theorem \ref{bicomplex_theorem} is presented in Sec. \ref{proof_section} using results that we obtain in Sec. \ref{preliminary}--Sec. \ref{def_diff_cable}.  On the basis of Theorem \ref{bicomplex_theorem}, we define a colored Khovanov bicomplex and its homology.  
\begin{definition}\label{tri_homology}
From Theorem. \ref{bicomplex_theorem}, the coboundary operator $d'$ causes the map ${d'}^{*} : H^{i}(C_{\textbf{n}}^{k, *, j}(D))$ $\to$ $H^{i}(C_{\textbf{n}}^{k+1, *, j}(D))$ to imply $H^{k}(H^{i}(C_{\textbf{n}}^{*, *, j}(D)))$.  The complex $\{C_{\textbf{n}}^{k, i, j}(D), d', d''\}$ is called the {\textit{colored Khovanov bicomplex}}, and its cohomology $H^{k}(H^{i}(C_{\textbf{n}}^{*, *, j}(D)))$ is called the {\textit{colored Khovanov homology}}.  
\end{definition}

\section{Preliminaries}\label{preliminary}
\subsection{Jones polynomial and colored Jones polynomial}\label{JonesPolynomial}
In this paper, the Jones polynomial $J(L)$ of variable $q$ of an oriented link $L$ in $\mathbb{R}^{3}$ is defined by the skein relation
\begin{equation}
q^{-2} J(L_{+}) - q^{2} J(L_{-}) = ( q^{-1} - q ) J(L_{0})
\end{equation}
for three arbitrary links $L_{+}$, $L_{-}$, and $L_{0}$ that differ as shown in Fig. \ref{jones_skein}, and its value on the unknot is $q+q^{-1}$.  
\begin{figure}
\begin{center}
\begin{picture}(0,0)
\put(50,-10){$L_{+}$}
\put(150,-10){$L_{-}$}
\put(248,-10){$L_{0}$}
\end{picture}
\qquad \includegraphics[width=10cm, bb=0 0 289.15 90.62]{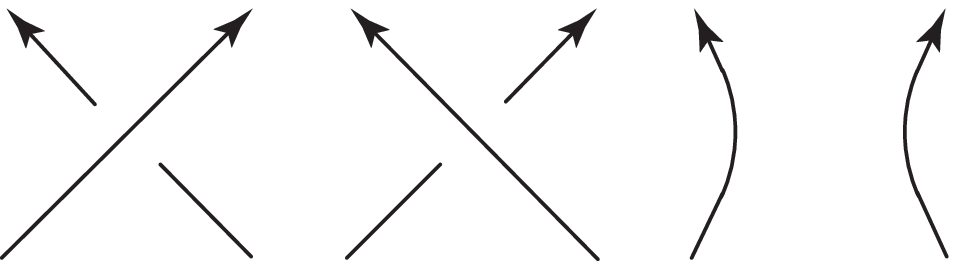}
\end{center}
\caption{Neighborhoods of the same part of an oriented link diagram.  The three exteriors of these neighborhoods are the same.}\label{jones_skein}
\end{figure}
Another definition of the Jones polynomial of an oriented link that has a diagram $D$ is given by
\begin{equation}
J(L) |_{q = - A^{-2}} = (-A)^{- 3 w(D)} \langle D \rangle
\end{equation}
where $w(D)$ is the number of crossings of $L_{+}$ minus that of $L_-$.  $\langle \cdot \rangle$ is the Kauffman bracket of the link diagram $D$ neglecting its orientation, and is defined by
\begin{equation}
\langle D_{\times} \rangle = A \langle D_{0} \rangle + A^{-1} \langle D_{\infty} \rangle
\end{equation}
for three arbitrary link diagrams $D_{\times}$, $D_{0}$, and $D_{\infty}$ that differ as shown in Fig. \ref{kauffman_b}, and its value on the Jordan curve in $\mathbb{R}^{2}$ is $- A^{-2} - A^{2}$.  
\begin{figure}
\begin{center}
\begin{picture}(0,0)
\put(50,-10){$D_{\times}$}
\put(150,-10){$D_{0}$}
\put(248,-10){$D_{\infty}$}
\end{picture}
\qquad \includegraphics[width=10cm, bb=0 0 289.15 90.62]{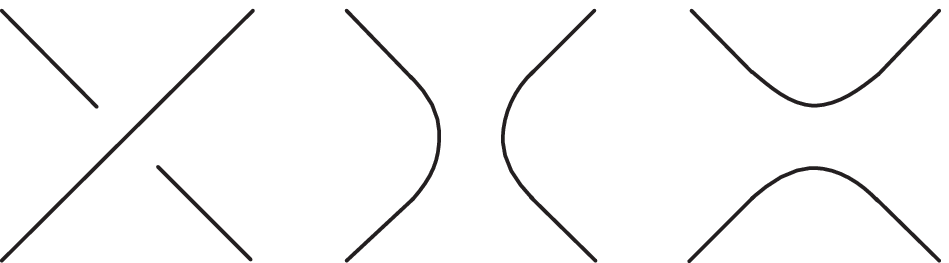}
\caption{Neighborhoods of the same part of an unoriented link diagram.  The three exteriors of these neighborhoods are the same.}\label{kauffman_b}
\end{center}
\end{figure}

In a similar manner to Khovanov \cite{khovanovjones}, we define the {\it{colored Jones polynomial}} as follows.  
\begin{definition}
The colored Jones polynomial is defined by formula (\ref{colored_jones}) using the Jones polynomial $J(L)$ defined in Sec. \ref{JonesPolynomial}.  
\end{definition}
For further details of formula (\ref{colored_jones}), see \cite{kirbymelvin, coloredkhovanov}.

\subsection{Khovanov homology of the Jones polynomial}
In this section, we recall the definition of the Khovanov homology of the Jones polynomial in the style of Viro \cite{viro}.  
\subsubsection{The $\mathbb{Z}_2$ Khovanov homology}\label{z_2homology}
Two cases are available for understanding the constructions of the desired bicomplex.  One entails using the coefficient $\mathbb{Z}$, and the other case, which is comparatively simpler, entails using the coefficient $\mathbb{Z}_2$.  Then, first, we recall the Khovanov homology with coefficients in $\mathbb{Z}_2$.  

Let us consider a link diagram and place a small edge (Fig. \ref{marker_b} (b) or (c)), called a {\it{marker}}, for every crossing (Fig \ref{marker_b} (a)) on the link diagram.  
\begin{figure}[htb]
\begin{center}
\begin{picture}(0,0)
\put(22,-10){(a)}
\put(84,-10){(b)}
\put(145,-10){(c)}
\put(203,-10){(d)}
\put(265,-10){(e)}
\end{picture}
\quad \includegraphics[width=10cm,bb=0 0 592 91.5]{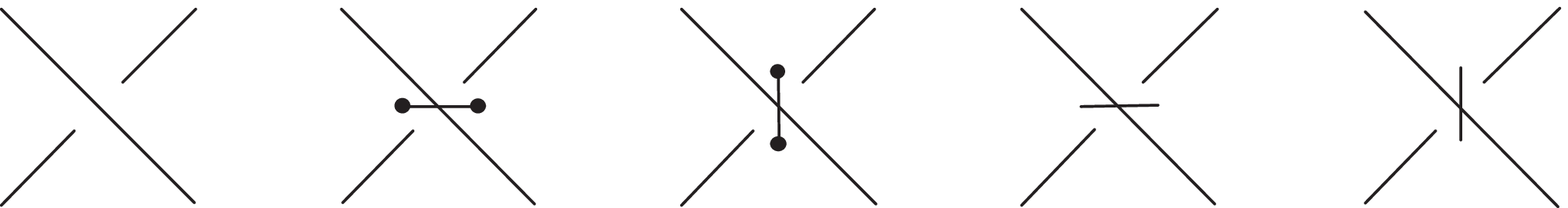}
\caption{A crossing (a) of a link diagram showing a positive marker (b), a negative marker (c), and a simple notation (d) ((e)) corresponding to (b) ((c)).  }\label{marker_b}
\end{center}
\end{figure}
Every marker and its sign are defined by the direction of smoothing for every crossing of the link diagram, as in Fig. \ref{marker_c}.  In the rest of this paper, we use the simple notation of Fig. \ref{marker_b} (d) ((e)) corresponding to the marker of Fig. \ref{marker_c} (b) ((c)).  
\begin{figure}
\begin{center}
\quad \includegraphics[width=10cm, bb=0 0 450 100]{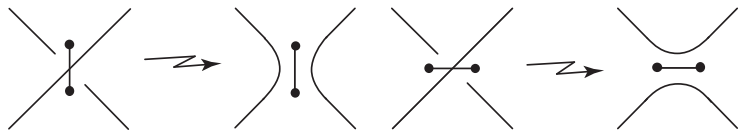}
\caption{Smoothing along markers.}\label{marker_c}
\end{center}
\end{figure}
The smoothed link diagram consists of Jordan curves, and is called the {\it{Kauffman state}} of the link diagram, or simply the {\it{state}} in this paper.  It is well known that Kauffman states determine the {\it{Kauffman bracket}} as follows.  For an arbitrary link diagram $D$, we denote the result of the smoothing by $s$ and the number of circles by $|s|$.  The number of positive markers minus the number of negative markers for an arbitrary $s$ is denoted by $\sigma(s)$.  The Kauffman bracket is written as
\begin{equation}
\langle D \rangle = \sum_{{\text{states}}~s} A^{\sigma (s)} (-A^{2} - A^{-2})^{|s|}.  
\end{equation}
As the next step of defining the Khovanov homology, we assign label $x$ or $1$ for every Jordan curve of the state.  We define the degrees of $x$ and $1$ by the map ``$\text{deg}$'' from $\{x, 1\}$ to $\{-1, 1\}$ such that ${\text{deg}}(x)$ $=$ $-1$ and ${\text{deg}}(1)$ $=$ $1$.  The state whose Jordan curves have the labels $x$ or $1$ is called an {\it{enhanced state}}, and is denoted by $S$.  Clearly, we can extend the definition of $\sigma$ for every enhanced state $S$ corresponding to $s$, therefore do so.  For $y$ $=$ $x$ or $y$ $=$ $1$, set $\tau(S)$ $=$ $\sum_{y~{\text{in}}~S}$ ${\text{deg}}(y)$.  For an oriented link diagram $D$ of a link $L$, the Jones polynomial $J(L)$ defined in Sec. \ref{JonesPolynomial} is obtained as 
\begin{equation}
J(L) = \sum_{{\text{enhanced states}}~S} (-1)^{i(S)} q^{j(S)}
\end{equation}
where $i(S)$ $=$ $(w(D) - \sigma(S))$ and $j(S)$ $=$ $w(D)$ $+$ $i(S)$ $+$ $\tau(S)$.  
Here, we would like to remark that
\begin{equation}
\begin{split}
J(L) &= (-A)^{-3w(D)} \langle D \rangle \\
&= \sum_{{\text{states}}~s} (-A)^{-3w(D)} A^{\sigma(s)} (-A^{2} - A^{-2})^{|s|}\\
&= \sum_{{\text{enhanced states}}~S} (-1)^{-3w(D) + |S|} A^{-3w(D) + \sigma(S) - 2\tau(S)}\\
&= \sum_{S} (-1)^{w(D) + \tau(S)} (A^{-2})^{w(D) + (w(D) - \sigma(S))/2 + \tau(S)}\\
&= \sum_{S} (-1)^{(w(D) - \sigma(S))/2} (- A^{-2})^{w(D) + (w(D) - \sigma(S))/2 + \tau(S)}\\
&= \sum_{S} (-1)^{(w(D) - \sigma(S))/2} q^{w(D) + (w(D) - \sigma(S))/2 + \tau(S)}\\
&= \sum_{S} (-1)^{i(S)} q^{j(S)}
\end{split}
\end{equation}
where we use the formula $|S|$ $\equiv$ $\tau(S)$ (mod $2$) for the number $|S|$ of circles in $S$.  
We consider the abelian group $C^{i, j}(D; \mathbb{Z}_2)$ with the coefficient $\mathbb{Z}_{2}$ generated by the enhanced states $S$ of a fixed link diagram $D$ satisfying $i(S)$ $=$ $i$ and $j(S)$ $=$ $j$.  For $D$ $=$ $\emptyset$, we consider that $C^{0, 0}(\emptyset; \mathbb{Z}_2)$ is generated by only one generator; and then, an enhanced state $S$ $=$ $\emptyset$ and $C^{0, 0}(\emptyset; \mathbb{Z}_2)$ is equal to $\mathbb{Z}_2$.  

Next, we define the coboundary operator $d_2$, usually called the differential in the case of the {\it{Khovanov homology}}.  We consider every enhanced state, denoted by $T$, obtained when the neighborhood of a single crossing with a positive marker is replaced with that of a negative marker in each of the cases listed in Fig. \ref{differential2}.  
\begin{figure}
\begin{center}
\begin{picture}(0,0)
\put(-10,135){(a)}
\put(-10,79){(b)}
\put(-10,18){(c)}
\put(170,135){(d)}
\put(170,79){(e)}
\put(170,18){(f)}
\put(37,160){$S$}
\put(51,135){$1$}
\put(22,134){$1$}
\put(51,79){$1$}
\put(22,79){$x$}
\put(51,18){$x$}
\put(22,17){$1$}
\put(134,160){$T$}
\put(148,134){$1$}
\put(147,78){$x$}
\put(147,17){$x$}
\put(197,170){$S$}
\put(277,170){$T$}
\put(197,148){$1$}
\put(275,148){$1$}
\put(275,120){$x$}
\put(197,89){$1$}
\put(275,90){$x$}
\put(275,63){$1$}
\put(275,33){$x$}
\put(275,5){$x$}
\put(197,32){$x$}
\end{picture}
\quad\ \includegraphics[width=10cm, bb=0 0 400 250]{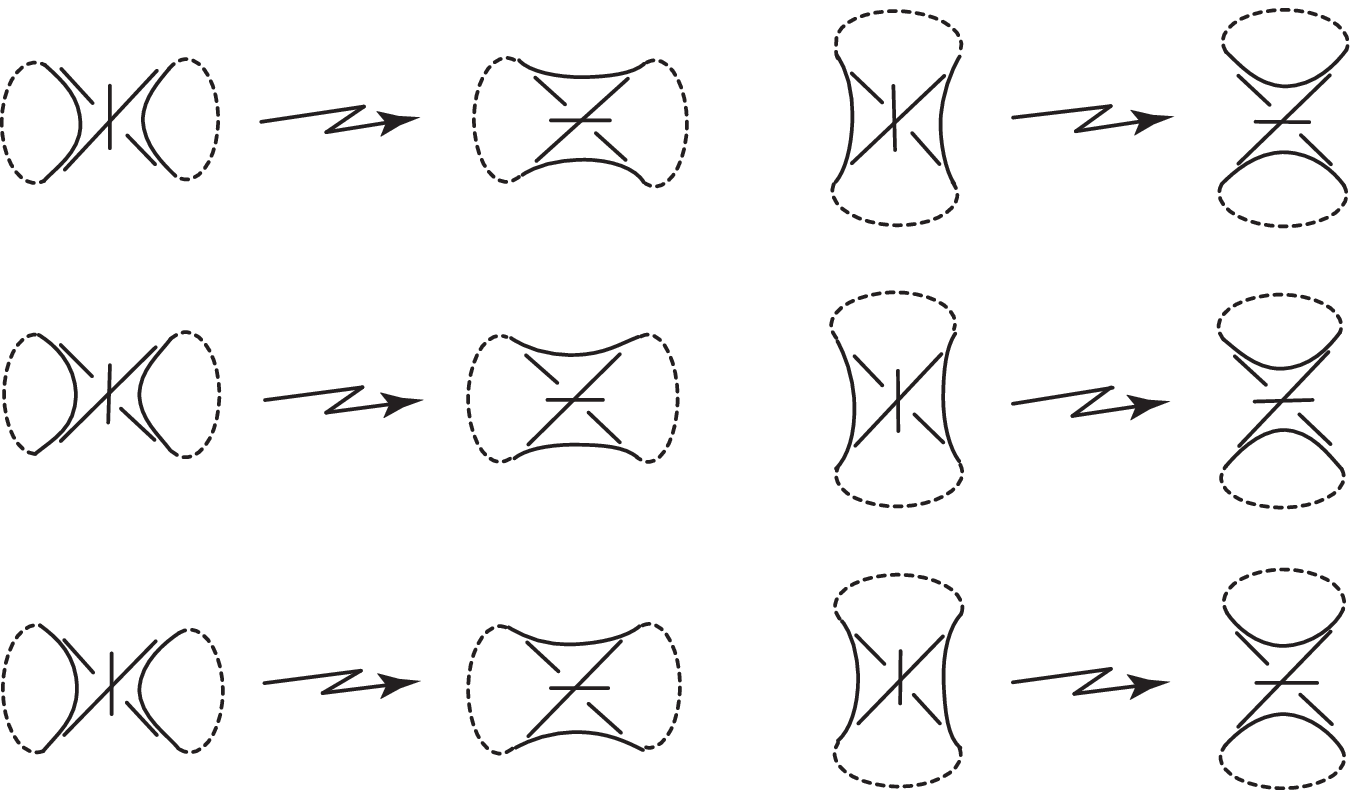}
\caption{Each figure to the left of an arbitrary arrow is $S$ and that to the right is $T$ for formula (\ref{differential_def}).  Each tuple of the enhanced states $S$ and $T$ defines the incidence number $(S : T)_2$ $=$ $1$ in formula (\ref{differential_def}).  The neighborhood of a positive marker of $S$ is replaced with that of a negative marker of $T$.  The dotted arcs are the common fragments of $S$ and $T$.  According to convention, we can represent the above figures using homomorphisms $m$ and $\Delta$ such as (a): $m(1 \otimes 1)$ $=$ $1$, (b): $m(x \otimes 1)$ $=$ $x$, (c): $m(1 \otimes x)$ $=$ $x$, (d) and (e): $\Delta(1)$ $=$ $1 \otimes x$ $+$ $x \otimes 1$, (f): $\Delta(x)$ $=$ $x \otimes x$.  In these formulae, each circle corresponds to $\mathbb{Z}_2 1$ $\oplus$ $\mathbb{Z}_2 x$ over $\mathbb{Z}_2$.}\label{differential2}
\end{center}
\end{figure}
For an arbitrary enhanced state $S$, the map $d_2$ is defined by
\begin{equation}\label{differential_def}
d_2(S) = \sum_{{\text{enhanced states}}~T}~(S:T)_2~T  
\end{equation}
where the incidence number $(S : T)_2$ is unity in each of the cases listed in Fig. \ref{differential2} and $(S : T)$ is $0$ otherwise.  The map $d_2$ is extended to the homomorphism from $C^{i, j}(D; \mathbb{Z}_2)$ to $C^{i+1, j}(D; \mathbb{Z}_2)$, since $j(S)$ $=$ $j(T)$ for the tuple of enhanced states listed in Fig. \ref{differential2} corresponding to $(S : T)_2$ $=$ $1$.  The homomorphism is denoted by the same symbol $d$ and becomes the coboundary operator of the Khovanov homology: that is, ${d_2}^{2}$ $=$ $0$ (for the proof of this formula, see \cite{viro}).  
\begin{theorem}[Khovanov]
Let $L$ be an arbitrary link with a diagram $D$.  For arbitrary $i$ and $j$, the cohomology group $H^{i}(C^{*, j}(D; \mathbb{Z}_2), d_2)$ is invariant under an ambient isotopy for $L$, so this cohomology group can be denoted by $H^{i, j}(L; \mathbb{Z}_2)$ and satisfies
\begin{equation}
J(L) = \sum_{j} q^{j} \sum_{i} (-1)^{i} {\operatorname{rank}}\,H^{i, j}(L; \mathbb{Z}_2).  
\end{equation}
\end{theorem}

From this definition, $H^{0, 0}(\emptyset; \mathbb{Z}_2)$ $=$ $\mathbb{Z}_2$ and $H^{0, 1}(\text{unknot}; \mathbb{Z}_2)$ $=$ $H^{0, -1}(\text{unknot};$ $\mathbb{Z}_2)$ $=$ $\mathbb{Z}_2$.  
\subsubsection{Extension to the coefficient $\mathbb{Z}$ case}\label{z_homology}
We consider the order of all the negative markers that belong to an enhanced state up to every permutation, and call this order {\it{the orientation}}.  The orientation is the {\it{opposite}} (resp. same) if two orders of negative markers differ by an odd (resp. even) permutation.  We consider the relation among enhanced states with the orders such that one enhanced state is another enhanced state multiplied by $-1$ ($1$) if the orders of these enhanced states have the opposite (same) orientations.  We call an enhanced state with this relation an {\it{oriented enhanced state}}.  We extend the relation to that of the abelian group generated by the oriented enhanced states $S$ satisfying $i(S)$ $=$ $i$ and $j(S)$ $=$ $j$ in a fixed link diagram $D$.  We denote the abelian group over the coefficient $\mathbb{Z}$ by $C^{i, j}(D)$.  

We now define the coboundary operator $d$ that is analogous to the $d_2$ given in Sec. \ref{z_2homology}.  For oriented enhanced states $S$ and $T$, we set the incidence number $(S : T)$ $=$ $1$ if $S$ and $T$ satisfy $(S : T)_2$ $=$ $1$ and are oriented by the orders of their negative markers such that the orders coincide on the common markers followed by the changing marker in the order of $T$. We define the map on oriented enhanced states as follows: 
\begin{equation}
d(S) = \sum_{{\text{oriented enhanced states}}~S} (S : T)~T.  
\end{equation}
We extend the map $d$ to that of $C^{i, j}(D)$ $\to$ $C^{i+1, j}(D)$, and denote this extended map by the same symbol $d$.  The extension of the coefficient from $\mathbb{Z}_2$ to $\mathbb{Z}$ is assured by the result of Viro \cite[Sec. 5.4]{viro} (originally given by Khovanov \cite{khovanovjones}), which outline below.  
\begin{theorem}[Viro]
The homomorphism $d$ satisfies $d^{2}$ $=$ $0$.  
\end{theorem}
Then, we have one of Khovanov's results \cite{khovanovjones}.  
\begin{theorem}[Khovanov]\label{khovanov_homology}
Let $L$ be an arbitrary link diagram.  For arbitrary $i$ and $j$, the cohomology group $H^{i}(C^{*, j}(D), d)$ is invariant under an ambient isotopy for $L$, so this cohomology group can be denoted by $H^{i, j}(L)$ and satisfies
\begin{equation}
J(L) = \sum_{j} q^{j} \sum_{i} (-1)^{i} {\operatorname{rank}}\,H^{i, j}(L).  
\end{equation}
\end{theorem}

\section{Technique for taking up small cabling diagrams from a big cabling diagram}
In this section, we introduce a technique for directly taking up a small cabling diagram from a big one.  Recall the definition of a $\textbf{k}$-cable provided in Sec. \ref{intro}; we denote ``$(k)$-cable'' in the case of a knot diagram as simply a ``$k$-cable''.  
\begin{theorem}\label{small_cable_k}
Let $D$ be a knot diagram and $D^{k}$ be the $k$-cable of $D$.  The knot diagram $D^{k}$ can be taken from $D^{k+2}$ by smoothing the crossings of two neighboring strands of $D^{k+2}$.  
\end{theorem}
\begin{proof}
Take any two neighboring strands from $D^{k+2}$, henceforth referred to as {\it{contracted strands}}, and choose any direction for them.  Strands that do not belong to contracted strands are called {\it{non-contracted strands}}.  Along the chosen direction, we smooth the crossings of the contracted strands according to the rule shown in Fig. \ref{type1_smooth}.  If contracted strands meet another pair of contracted strands, we smooth four crossings, as shown in Fig.\ref{type1_smooth} (a).  If contracted strands meet a non-contracted strand, we smooth two crossings, as in Fig. \ref{type1_smooth} (b).  After smoothing, we have the knot diagram $D^{k}$ and a finite number of Jordan circles, called {\it{contracted circles}}.  This is because any arcs in the right-hand images of Figs. \ref{type1_smooth} (a) and (b) will connect simple arcs or other non-contracted strands as in Fig. \ref{type1_closed}.  Finally, the obtained diagram $D^{k}$ does not depend on the choice of direction and contracted strands, since the smoothing rule shown in Fig.\ref{type1_smooth} (a) does not depend on these choices and the rule of Fig.\ref{type1_smooth} (b) does not change the diagram $D^{k}$ using Fig.\ref{type1_closed}.  This concludes the proof.  
\end{proof}
\begin{figure}
\begin{center}
\includegraphics[width=11cm,bb=0 0 432.5 77.34]{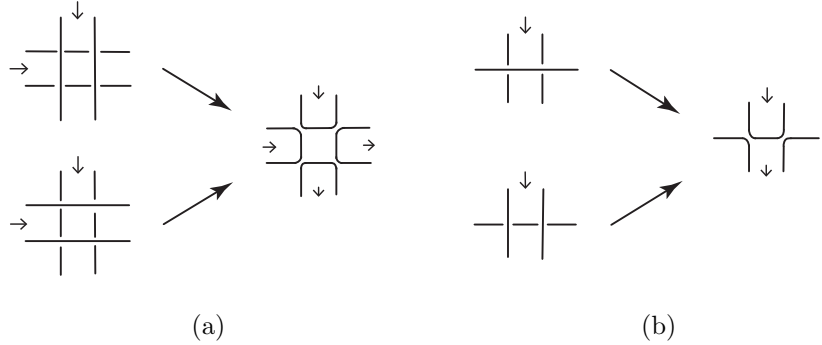}
\end{center}
\begin{picture}(0,0)
\put(85,0){(a)}
\put(255,0){(b)}
\end{picture}
\caption{Rule for smoothing the crossings of two neighboring strands, called contracted strands.  The small arrows show the directions of two neighboring strands.  (a) Four crossings consisting of contracted strands.  (b) Two crossings consisting of contracted strands and non-contracted strands.}\label{type1_smooth}
\end{figure}

\begin{figure}
\begin{center}
\includegraphics[width=5cm,bb=0 0 350.5 224.33]{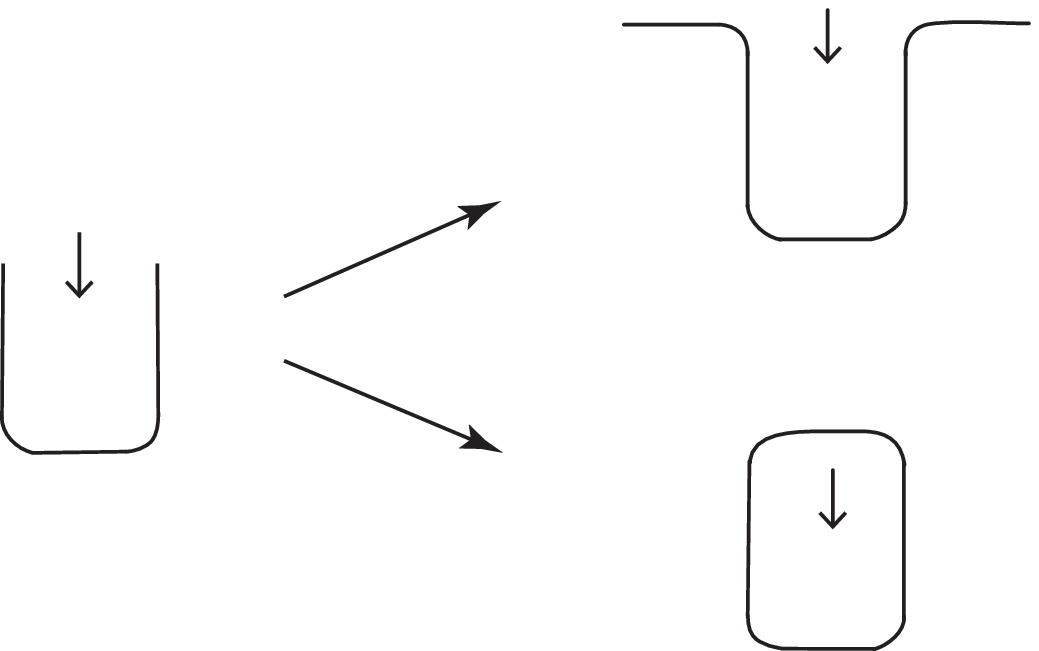}
\includegraphics[width=5cm,bb=0 0 300 200]{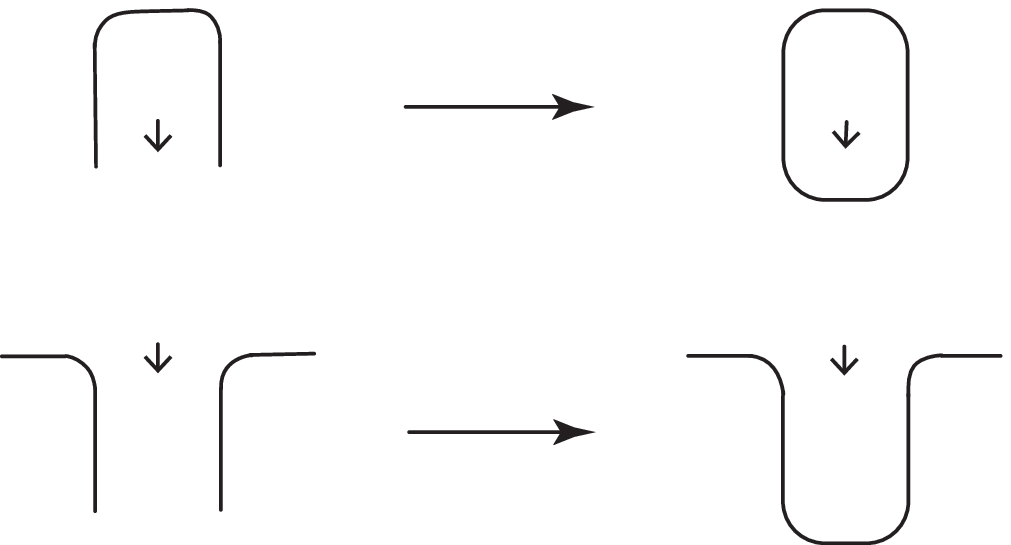}
\end{center}
\caption{Possible arcs after smoothing the crossings in Fig. \ref{type1_smooth}.}\label{type1_closed}
\end{figure}

Extending the above discussion to that of the link case, we have the following.  
\begin{theorem}\label{small_cable_l}
The $(k_1, k_2, \dots, k_i, \dots, k_l)$-cable of a $l$-component link diagram $D$ can be taken from the $(k_1, k_2, \dots, k_{i} + 2, \dots, k_l)$-cable of $D$.  
\end{theorem}
\begin{proof}
Let us choose an arbitrary $i$-th component of two neighboring strands, also called {\it{contracted strands}}, and their direction.  

For the link case, contracted strands meet other components.  All strands of the other components are non-contracted strands.  In this case, we apply smoothing of the kind shown in Fig. \ref{type1_smooth} (b).  After smoothing, considering Fig. \ref{type1_closed}, other components do not change the link diagram.  Let us call these Jordan circles {\it{contracted circles}}, as in the proof of Theorem \ref{small_cable_k}.  At their crossings, no new contracted circles arise, and the choice of contracted strands and their direction does not change the link diagram $D^{(k_1, k_2, \dots, k_l)}$ up to plane isotopy.  In the case of smoothing crossings of contracted strands as illustrated in Fig. \ref{type1_smooth} (a) in the proof of Theorem \ref{small_cable_k}, the choices of contracted strands and their direction does not change the link diagram $D^{(k_1, k_2, \dots, k_l)}$ up to plane isotopy and contracted circles, which are a finite number of Jordan circles.  
\end{proof}

\begin{definition}\label{def_type1}
For the $(k_1, k_2, \dots, {k_i} + 2,$ $\dots,$ $k_l)$-cable of a $l$-component link diagram $D$, the $(k_1, k_2, \dots, {k_i},$ $\dots,$ $k_l)$-cable of $D$ with contracted circles defined in the above construction is called a {\it{Type $1$ diagram}}, or simply {\it{Type $1$}}.  
\end{definition}

\begin{proposition}\label{contracted_even}
The number of contracted circles is even.  
\end{proposition}
\begin{proof}
\begin{itemize}

\item (Knot case) First, we will show that ``the number of contracted circles is even ($\ast$)" in the case of knot diagrams.  Every set of crossings of the cable diagram of a knot can be represented as in Fig. \ref{cable-fig} (a) or as its mirror image (see Fig. \ref{type1_smooth}).  Below, we explain the details of considering Type $1$ with the help of Fig. \ref{cable-fig} (bottom left of Fig. \ref{type1_smooth} (a)); it should be note that we must consider not only those cases in Fig. \ref{cable-fig}, but also their mirror images (top left of Fig. \ref{type1_smooth} (a)), for which the discussion is similar.  After considering the Type $1$ of Fig. \ref{cable-fig} (a), we obtain one of Fig. \ref{cable-fig} (b).  Consider the four pairs of neighboring crossings connecting as shown in Fig. \ref{cable-fig} (a).  Then, this part (i.e., (a)) of Type $1$ has to be connected to two simple arcs located to the right and the bottom of Fig. \ref{cable-fig} (b).  Similarly, each of the other two simple arcs, located to the left and top of Fig. \ref{cable-fig} (b) (or its mirror image), connects to other sets of crossings in other panels.  The dashed lines in every (b)-type panel of Fig. \ref{cable-fig} represent simple arcs connecting with other panels.  By shortening the dashed lines using a plane isotopy, we see that Type $1$ consists of the set of crossings shown in Fig. \ref{cable-fig} (c) and their mirror images.  This proves ($\ast$) in the case of knot diagrams.  In summary, noting that the terms of ``(b)-type'' (``(c)-type'') contain not only (b)-types ((c)) of Fig. \ref{cable-fig} but also their mirror images, 
\begin{equation}
\begin{split}
&\quad\ {\text{Type 1 of a diagram consisting of (a)-type crossings}}\\ &= {\text{a diagram consisting of (b)-type crossings and contracted circles}} \\
&\sim {\text{a diagram consisting of (c)-type crossings and contracted circles}}
\end{split}
\end{equation}
where $\sim$ is a plane isotopy.  

\item (Link case) Second, we consider the case of link diagrams.  In this case, the sets of crossings arise as shown in Figs. \ref{cable-fig2} (a) and (b), but they do not produce any contracted circles, because these sets of crossings change to Fig. \ref{cable-fig2} (c) when we consider Type 1 for panels (a) and (b), as in the proof of Theorem \ref{small_cable_l}.  In addition, the discussion concerning Fig. \ref{cable-fig} in the previous paragraph also applies to the case of a link.  Thus, following a similar argument as above for the dashed lines, ($\ast$) still holds in the case of link diagrams.  
\end{itemize}
\end{proof}

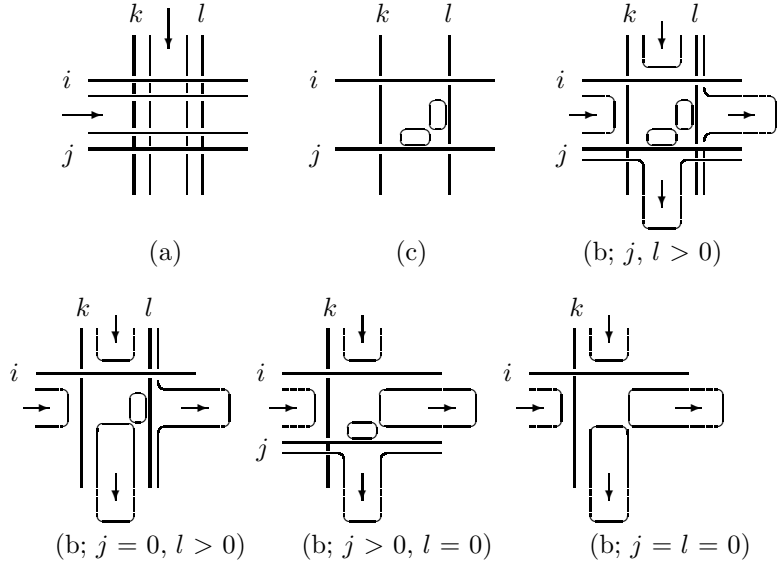
\begin{figure}[htbp]
\centering
\qquad
\qquad
\begin{minipage}{90pt}
\begin{picture}(70,90)
\put(22.5,-5){(a)}
\put(15,85){$k$}
\put(41,85){$l$}
\put(-10,32){$j$}
\put(-10,60){$i$}
\put(-10,50){\vector(1,0){15}}
\put(30,90){\vector(0,-1){15}}
\put(0,43){\line(1,0){60}}
{\thicklines
\put(0,37){\line(1,0){60}}}
{\thicklines
\put(0,63){\line(1,0){60}}}
\put(0,57){\line(1,0){60}}
{\thicklines
\put(17,20){\line(0,1){15}}}
\put(23,20){\line(0,1){15}}
{\thicklines
\qbezier(17,38.5)(17,40)(17,41.5)}
\qbezier(23,38.5)(23,40)(23,41.5)
{\thicklines
\put(43,20){\line(0,1){15}}}
\put(37,20){\line(0,1){15}}
\qbezier(37,38.5)(37,40)(37,41.5)
{\thicklines
\qbezier(43,38.5)(43,40)(43,41.5)}
{\thicklines
\put(17,45){\line(0,1){10}}}
\put(23,45){\line(0,1){10}}
{\thicklines
\qbezier(17,58.5)(17,60)(17,61.5)}
\qbezier(23,58.5)(23,60)(23,61.5)
\put(37,45){\line(0,1){10}}
{\thicklines
\put(43,45){\line(0,1){10}}}
\qbezier(37,58.5)(37,60)(37,61.5)
{\thicklines
\qbezier(43,58.5)(43,60)(43,61.5)}
{\thicklines
\put(17,65){\line(0,1){15}}}
\put(23,65){\line(0,1){15}}
\put(37,65){\line(0,1){15}}
{\thicklines
\put(43,65){\line(0,1){15}}}
\end{picture}
\end{minipage}
\begin{minipage}{90pt}
\begin{picture}(70,90)
\put(22.5,-5){(c)}
\put(15,85){$k$}
\put(41,85){$l$}
\put(-10,32){$j$}
\put(-10,60){$i$}
{\thicklines
\put(0,37){\line(1,0){60}}}
{\thicklines
\put(0,63){\line(1,0){60}}}
{\thicklines
\put(17,20){\line(0,1){15}}}
{\thicklines
\put(43,20){\line(0,1){15}}}
{\thicklines
\put(17,39){\line(0,1){22}}}
{\thicklines
\put(43,39){\line(0,1){22}}}
{\thicklines
\put(17,65){\line(0,1){15}}}
{\thicklines
\put(43,65){\line(0,1){15}}}
\put(30,41.5){\oval(11,6)}
\put(38.5,50){\oval(6,11)}
\end{picture}
\end{minipage}
\begin{minipage}{90pt}
\begin{picture}(70,90)
\put(0,-5){(b; $j$, $l$ $>$ $0$)}
\put(-5,50){\vector(1,0){10}}
\put(30,85){\vector(0,-1){10}}
\put(55,50){\vector(1,0){10}}
\put(30,25){\vector(0,-1){10}}
\put(15,85){$k$}
\put(41,85){$l$}
\put(-10,32){$j$}
\put(-10,60){$i$}
{\thicklines
\put(0,37){\line(1,0){60}}}
\put(49,43){\line(1,0){11}}
\put(49,57){\line(1,0){11}}
\qbezier(46,40)(46,43)(49,43)
\qbezier(46,60)(46,57)(49,57)
{\thicklines
\put(0,63){\line(1,0){60}}}
\qbezier(0,57)(1,57)(2,57)
\qbezier(3,57)(4,57)(5,57)
\qbezier(6,57)(7,57)(8,57)
\qbezier(0,43)(1,43)(2,43)
\qbezier(3,43)(4,43)(5,43)
\qbezier(6,43)(7,43)(8,43)
{\thicklines
\put(17,20){\line(0,1){12}}
\put(17,34){\line(0,1){2}}
\put(43,34){\line(0,1){2}}
}
\qbezier(23,19)(23,18)(23,17)
\qbezier(23,16)(23,15)(23,14)
\qbezier(23,13)(23,12)(23,11)
\put(23,20){\line(0,1){10}}
\qbezier(23,30)(23,33)(20,33)
\put(20,33){\line(-1,0){20}}
{\thicklines
\put(43,20){\line(0,1){12}}
}
\put(46,20){\line(0,1){12}}
\put(46,34){\line(0,1){2}}
\qbezier(37,19)(37,18)(37,17)
\qbezier(37,16)(37,15)(37,14)
\qbezier(37,13)(37,12)(37,11)
\put(37,20){\line(0,1){10}}
\qbezier(37,30)(37,33)(40,33)
\put(40,33){\line(1,0){20}}
{\thicklines
\put(17,65){\line(0,1){15}}}
\qbezier(23,80)(23,79)(23,78)
\qbezier(23,77)(23,76)(23,75)
\qbezier(23,74)(23,73)(23,72)
\qbezier(37,80)(37,79)(37,78)
\qbezier(37,77)(37,76)(37,75)
\qbezier(37,74)(37,73)(37,72)
{\thicklines
\put(43,65){\line(0,1){15}}}
\put(46,65){\line(0,1){15}}
\put(30,71){\oval(14,6)[b]}
\put(30,10){\oval(14,6)[b]}
\put(9,50){\oval(6,14)[r]}
\put(70,50){\oval(6,14)[r]}
\qbezier(61,43)(62,43)(63,43)
\qbezier(64,43)(65,43)(66,43)
\qbezier(67,43)(68,43)(69,43)
\qbezier(61,57)(62,57)(63,57)
\qbezier(64,57)(65,57)(66,57)
\qbezier(67,57)(68,57)(69,57)
\put(30,41.5){\oval(11,6)}
\put(38.5,50){\oval(6,11)}
{\thicklines
\put(17,39){\line(0,1){22}}}
{\thicklines
\put(43,39){\line(0,1){22}}}
\end{picture}
\end{minipage}\\
\begin{minipage}{90pt}
\begin{picture}(70,110)
\put(7,-5){(b; $j$ $=$ $0$, $l$ $>$ $0$)}
\put(-5,50){\vector(1,0){10}}
\put(30,85){\vector(0,-1){10}}
\put(55,50){\vector(1,0){10}}
\put(30,25){\vector(0,-1){10}}
\put(15,85){$k$}
\put(41,85){$l$}
\put(-10,60){$i$}
\put(49,43){\line(1,0){11}}
\put(49,57){\line(1,0){11}}
\qbezier(46,40)(46,43)(49,43)
\qbezier(46,60)(46,57)(49,57)
\put(46,40){\line(0,-1){20}}
{\thicklines
\put(0,63){\line(1,0){60}}}
\qbezier(0,57)(1,57)(2,57)
\qbezier(3,57)(4,57)(5,57)
\qbezier(6,57)(7,57)(8,57)
\qbezier(0,43)(1,43)(2,43)
\qbezier(3,43)(4,43)(5,43)
\qbezier(6,43)(7,43)(8,43)
\qbezier(23,19)(23,18)(23,17)
\qbezier(23,16)(23,15)(23,14)
\qbezier(23,13)(23,12)(23,11)
\put(23,20){\line(0,1){22}}
\qbezier(37,19)(37,18)(37,17)
\qbezier(37,16)(37,15)(37,14)
\qbezier(37,13)(37,12)(37,11)
\put(37,20){\line(0,1){22}}
{\thicklines
\put(17,65){\line(0,1){15}}}
\qbezier(23,80)(23,79)(23,78)
\qbezier(23,77)(23,76)(23,75)
\qbezier(23,74)(23,73)(23,72)
\qbezier(37,80)(37,79)(37,78)
\qbezier(37,77)(37,76)(37,75)
\qbezier(37,74)(37,73)(37,72)
{\thicklines
\put(43,65){\line(0,1){15}}}
\put(46,65){\line(0,1){15}}
\put(30,71){\oval(14,6)[b]}
\put(30,10){\oval(14,6)[b]}
\put(9,50){\oval(6,14)[r]}
\put(70,50){\oval(6,14)[r]}
\qbezier(61,43)(62,43)(63,43)
\qbezier(64,43)(65,43)(66,43)
\qbezier(67,43)(68,43)(69,43)
\qbezier(61,57)(62,57)(63,57)
\qbezier(64,57)(65,57)(66,57)
\qbezier(67,57)(68,57)(69,57)
\put(30,41){\oval(13.7,6)[t]}
\put(38.5,50){\oval(6,11)}
{\thicklines
\put(17,20){\line(0,1){41}}}
{\thicklines
\put(43,20){\line(0,1){41}}}
\end{picture}
\end{minipage}
\begin{minipage}{90pt}
\begin{picture}(70,110)
\put(7,-5){(b; $j$ $>$ $0$, $l$ $=$ $0$)}
\put(-5,50){\vector(1,0){10}}
\put(30,85){\vector(0,-1){10}}
\put(55,50){\vector(1,0){10}}
\put(30,25){\vector(0,-1){10}}
\put(15,85){$k$}
\put(-10,32){$j$}
\put(-10,60){$i$}
{\thicklines
\put(0,37){\line(1,0){60}}}
\put(60,43){\line(-1,0){21}}
\put(60,57){\line(-1,0){21}}
{\thicklines
\put(0,63){\line(1,0){60}}}
\qbezier(0,57)(1,57)(2,57)
\qbezier(3,57)(4,57)(5,57)
\qbezier(6,57)(7,57)(8,57)
\qbezier(0,43)(1,43)(2,43)
\qbezier(3,43)(4,43)(5,43)
\qbezier(6,43)(7,43)(8,43)
{\thicklines
\put(17,20){\line(0,1){12}}
\put(17,34){\line(0,1){2}}
}
\qbezier(23,19)(23,18)(23,17)
\qbezier(23,16)(23,15)(23,14)
\qbezier(23,13)(23,12)(23,11)
\put(23,20){\line(0,1){10}}
\qbezier(23,30)(23,33)(20,33)
\put(20,33){\line(-1,0){20}}
\qbezier(37,19)(37,18)(37,17)
\qbezier(37,16)(37,15)(37,14)
\qbezier(37,13)(37,12)(37,11)
\put(37,20){\line(0,1){10}}
\qbezier(37,30)(37,33)(40,33)
\put(40,33){\line(1,0){20}}
{\thicklines
\put(17,65){\line(0,1){15}}}
\qbezier(23,80)(23,79)(23,78)
\qbezier(23,77)(23,76)(23,75)
\qbezier(23,74)(23,73)(23,72)
\qbezier(37,80)(37,79)(37,78)
\qbezier(37,77)(37,76)(37,75)
\qbezier(37,74)(37,73)(37,72)
\put(30,71){\oval(14,6)[b]}
\put(30,10){\oval(14,6)[b]}
\put(9,50){\oval(6,14)[r]}
\put(70,50){\oval(6,14)[r]}
\qbezier(61,43)(62,43)(63,43)
\qbezier(64,43)(65,43)(66,43)
\qbezier(67,43)(68,43)(69,43)
\qbezier(61,57)(62,57)(63,57)
\qbezier(64,57)(65,57)(66,57)
\qbezier(67,57)(68,57)(69,57)
\put(30,41.5){\oval(11,6)}
\put(39.5,50){\oval(6,13.7)[l]}
{\thicklines
\put(17,39){\line(0,1){22}}}
\end{picture}
\end{minipage}
\begin{minipage}{90pt}
\begin{picture}(70,110)
\put(60,43){\line(-1,0){21}}
\put(60,57){\line(-1,0){21}}
\put(23,20){\line(0,1){22}}
\put(37,20){\line(0,1){22}}
\put(22.5,-5){(b; $j$ $=$ $l$ $=$ $0$)}
\put(-5,50){\vector(1,0){10}}
\put(30,85){\vector(0,-1){10}}
\put(55,50){\vector(1,0){10}}
\put(30,25){\vector(0,-1){10}}
\put(15,85){$k$}
\put(-10,60){$i$}
{\thicklines
\put(0,63){\line(1,0){60}}}
\qbezier(0,57)(1,57)(2,57)
\qbezier(3,57)(4,57)(5,57)
\qbezier(6,57)(7,57)(8,57)
\qbezier(0,43)(1,43)(2,43)
\qbezier(3,43)(4,43)(5,43)
\qbezier(6,43)(7,43)(8,43)
{\thicklines
\put(17,20){\line(0,1){42}}
}
\qbezier(23,19)(23,18)(23,17)
\qbezier(23,16)(23,15)(23,14)
\qbezier(23,13)(23,12)(23,11)
\qbezier(37,19)(37,18)(37,17)
\qbezier(37,16)(37,15)(37,14)
\qbezier(37,13)(37,12)(37,11)
{\thicklines
\put(17,65){\line(0,1){15}}}
\qbezier(23,80)(23,79)(23,78)
\qbezier(23,77)(23,76)(23,75)
\qbezier(23,74)(23,73)(23,72)
\qbezier(37,80)(37,79)(37,78)
\qbezier(37,77)(37,76)(37,75)
\qbezier(37,74)(37,73)(37,72)
\put(30,71){\oval(14,6)[b]}
\put(30,10){\oval(14,6)[b]}
\put(9,50){\oval(6,14)[r]}
\put(70,50){\oval(6,14)[r]}
\qbezier(61,43)(62,43)(63,43)
\qbezier(64,43)(65,43)(66,43)
\qbezier(67,43)(68,43)(69,43)
\qbezier(61,57)(62,57)(63,57)
\qbezier(64,57)(65,57)(66,57)
\qbezier(67,57)(68,57)(69,57)
\put(30,40){\oval(13.7,6)[t]}
\put(40.5,50){\oval(6,13.7)[l]}
\end{picture}
\end{minipage}
\caption{The thin lines denote one-cable strands.  The numbers $i$, $j$, $k$, and $l$ ($\ge 0$) with thick lines indicate the number of parallel strands.  (a) has $i+1$, $j+1$, $k+1$, and $l+1$-cable strands.  Circles in (b) and (c) represent contracted circles.  In the case of knot diagrams, there are always has two cases: 1) $i$ $=$ $l$ and 2) $j$ $=$ $k$, $i$ $=$ $k$, and $j$ $=$ $l$.  }\label{cable-fig}
\end{figure}
\begin{figure}[htbp]
\centering
\qquad
\qquad
\begin{minipage}{80pt}
\begin{picture}(70,80)
\put(-10,47){$l$}
{\thicklines
\put(0,50){\line(1,0){18}}
\put(22,50){\line(1,0){16}}
\put(42,50){\line(1,0){18}}
}
\put(20,30){\line(0,1){40}}
\put(40,30){\line(0,1){40}}
\put(23,0){(a)}
\put(30,80){\vector(0,-1){15}}
\qbezier(20,21)(20,22)(20,23)
\qbezier(20,24)(20,25)(20,26)
\qbezier(20,27)(20,28)(20,29)
\qbezier(40,21)(40,22)(40,23)
\qbezier(40,24)(40,25)(40,26)
\qbezier(40,27)(40,28)(40,29)
\qbezier(20,71)(20,72)(20,73)
\qbezier(20,74)(20,75)(20,76)
\qbezier(20,77)(20,78)(20,79)
\qbezier(40,71)(40,72)(40,73)
\qbezier(40,74)(40,75)(40,76)
\qbezier(40,77)(40,78)(40,79)
\end{picture}
\end{minipage}
\begin{minipage}{80pt}
\begin{picture}(70,80)
\put(-10,47){$l$}
{\thicklines
\put(0,50){\line(1,0){60}}}
\put(20,30){\line(0,1){16}}
\put(40,30){\line(0,1){16}}
\qbezier(20,21)(20,22)(20,23)
\qbezier(20,24)(20,25)(20,26)
\qbezier(20,27)(20,28)(20,29)
\qbezier(40,21)(40,22)(40,23)
\qbezier(40,24)(40,25)(40,26)
\qbezier(40,27)(40,28)(40,29)
\put(20,54){\line(0,1){16}}
\put(40,54){\line(0,1){16}}
\qbezier(20,71)(20,72)(20,73)
\qbezier(20,74)(20,75)(20,76)
\qbezier(20,77)(20,78)(20,79)
\qbezier(40,71)(40,72)(40,73)
\qbezier(40,74)(40,75)(40,76)
\qbezier(40,77)(40,78)(40,79)
\put(23,0){(b)}
\put(30,80){\vector(0,-1){15}}
\end{picture}
\end{minipage}
\begin{minipage}{80pt}
\begin{picture}(70,80)
\put(-10,47){$l$}
{\thicklines
\put(0,50){\line(1,0){60}}}
\put(23,0){(c)}
\end{picture}
\end{minipage}
\caption{(a), (b): Contracted strands encountered by $l$ parallel non-contracted strands.  (c): Type 1 corresponding to both (a) and (b).}\label{cable-fig2}
\end{figure}
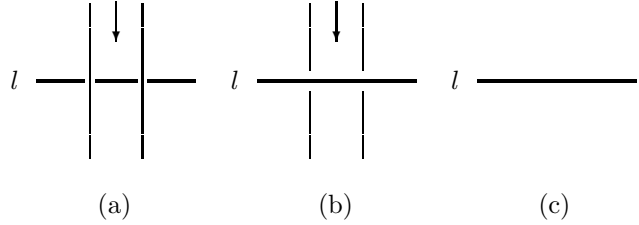

\subsection{Enhanced Kauffman states of Type 1 preserving grading $j$}
\subsubsection{Definition of Type $1_j$}\label{sec_def_type1}
On the basis of Theorem \ref{small_cable_l} and Proposition \ref{contracted_even}, we can define a map preserving grading $j$ between complexes of the Khovanov homology from a complex of a big cabling diagram to that of a small cabling diagram.  To do so, we must choose ``good'' enhanced states for Type 1, as defined below.  

Let us fix the tuple of nonnegative integers $\textbf{k}$ $=$ $(k_1, k_2, \dots, {k_i} + 2, \dots k_l)$ and an $l$-component link $D$.  By choosing markers that give a Type 1 smoothing of $D^{\textbf{k}}$, we have a $(k_1, k_2, \dots, k_i, \dots, k_l)$-cable of $D$.  Here, we have an even number of contracted circles (by Prop. \ref{contracted_even}).  We can define $x$ and $1$ of these contracted circles in Type 1 that preserves $j$.  For every panel of the kind in Fig. \ref{cable-fig} (c) of Type 1 or its mirror image, we associate $x$ ($1$) with the bottom (right) circle in the rectangle of Fig. \ref{cable-fig} (c) or its mirror image as one of the cases shown in Fig. \ref{cable-fig3}.  In other words, proceeding along contracted strands, every time we take the markers of crossings of contracted strands of $D^{\textbf{k}}$ giving Type $1$ smoothing, we choose either $x$ or $1$ for the contracted circles from Fig. \ref{cable-fig3}.  We then assign arbitrary markers for all crossings of non-contracted strands, and choose $x$ or $1$ for non-contracted strands that meet contracted strands as in (a-5) or (a-7) of Fig. \ref{type12}.  Next, we arbitrarily choose $x$ or $1$ for the other circles of the enhanced state.  The enhanced state given by the above process is called {\it{Type $1_j$}} (Fig. \ref{cable-fig3}).  For these enhanced states of Type $1_j$, local neighborhoods of crossings of contracted strands are characterized by (a-1) and (a-2) in Fig. \ref{type-figu1}, (a-3) and (a-4) in Fig. \ref{type-figu2}, and (a-5)--(a-8) in Fig. \ref{type12}.  
\begin{figure}[htbp]
\begin{center}
\quad\ 
\begin{minipage}{80pt}
\begin{picture}(70,85)
\put(-10,50){\vector(1,0){15}}
\put(30,90){\vector(0,-1){15}}
\put(15,85){$k$}
\put(41,85){$l$}
\put(-10,32){$j$}
\put(-10,60){$i$}
{\thicklines
\put(17,20){\line(0,1){60}}}
{\thicklines
\put(43,20){\line(0,1){60}}}
{\thicklines
\put(0,37){\line(1,0){15}}}
{\thicklines
\put(0,63){\line(1,0){15}}}
{\thicklines
\put(19,37){\line(1,0){22}}}
{\thicklines
\put(19,63){\line(1,0){22}}}
{\thicklines
\put(45,37){\line(1,0){15}}}
{\thicklines
\put(45,63){\line(1,0){15}}}
\put(30,41.5){\oval(11,6)}
\put(28,40.1){\tiny $x$}
\put(38.5,50){\oval(6,11)}
\put(36.7,48.2){\tiny $1$}
\end{picture}
\end{minipage}
\begin{minipage}{80pt}
\begin{picture}(70,85)
\put(-10,50){\vector(1,0){15}}
\put(30,90){\vector(0,-1){15}}
\put(15,85){$k$}
\put(41,85){$l$}
\put(-10,32){$j$}
\put(-10,60){$i$}
{\thicklines
\put(17,20){\line(0,1){60}}}
{\thicklines
\put(43,20){\line(0,1){60}}}
{\thicklines
\put(0,37){\line(1,0){15}}}
{\thicklines
\put(0,63){\line(1,0){15}}}
{\thicklines
\put(19,37){\line(1,0){22}}}
{\thicklines
\put(19,63){\line(1,0){22}}}
{\thicklines
\put(45,37){\line(1,0){15}}}
{\thicklines
\put(45,63){\line(1,0){15}}}
\put(30,41.5){\oval(11,6)}
\put(28,40.1){\tiny $1$}
\put(38.5,50){\oval(6,11)}
\put(36.7,48.2){\tiny $x$}
\end{picture}
\end{minipage}
\begin{minipage}{80pt}
\begin{picture}(70,85)
\put(-10,50){\vector(1,0){15}}
\put(30,90){\vector(0,-1){15}}
\put(15,85){$k$}
\put(41,85){$l$}
\put(-10,32){$j$}
\put(-10,60){$i$}
{\thicklines
\put(0,37){\line(1,0){60}}}
{\thicklines
\put(0,63){\line(1,0){60}}}
{\thicklines
\put(17,20){\line(0,1){15}}}
{\thicklines
\put(43,20){\line(0,1){15}}}
{\thicklines
\put(17,39){\line(0,1){22}}}
{\thicklines
\put(43,39){\line(0,1){22}}}
{\thicklines
\put(17,65){\line(0,1){15}}}
{\thicklines
\put(43,65){\line(0,1){15}}}
\put(30,41.5){\oval(11,6)}
\put(28,40.1){\tiny $1$}
\put(38.5,50){\oval(6,11)}
\put(36.7,48.2){\tiny $x$}
\end{picture}
\end{minipage}
\begin{minipage}{80pt}
\begin{picture}(70,85)
\put(-10,50){\vector(1,0){15}}
\put(30,90){\vector(0,-1){15}}
\put(15,85){$k$}
\put(41,85){$l$}
\put(-10,32){$j$}
\put(-10,60){$i$}
{\thicklines
\put(0,37){\line(1,0){60}}
\put(0,63){\line(1,0){60}}
\put(17,20){\line(0,1){15}}
\put(43,20){\line(0,1){15}}
\put(17,39){\line(0,1){22}}
\put(43,39){\line(0,1){22}}
\put(17,65){\line(0,1){15}}
\put(43,65){\line(0,1){15}}
}
\put(30,41.5){\oval(11,6)}
\put(28,40.1){\tiny $x$}
\put(38.5,50){\oval(6,11)}
\put(36.9,48.2){\tiny $1$}
\end{picture}
\end{minipage}
\caption{Type $1_j$.  }\label{cable-fig3}
\end{center}
\end{figure}
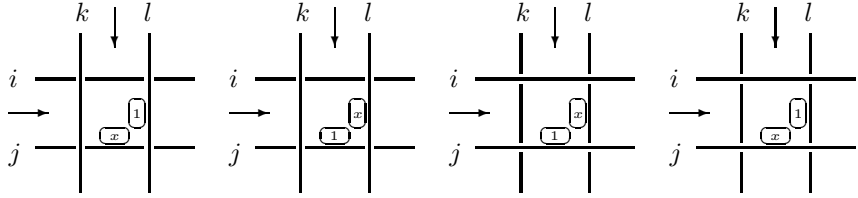

\subsubsection{Definition of Type $2_{j}$ for Type $1_{j}$}
To define the chain map $d' : C^{i, j}(D^{k+2})$ $\to$ $C^{i, j}(D^{k})$, consider an enhanced state $S$ such that $d(S)$ $=$ $\pm$ Type $1_j$ + other terms.  We have $d'(\pm {\text{Type}} 1_j )$ $+$ $d'({\text{other terms}})$ $=$ $d'(d(S))$ $=$ $d(d'(S))$ $=$ $0$ if $d'(S)$ $=$ $0$.  Then, ``other terms'' must then contain a term such that $d'({\text{the term}})$ and $d'(\pm {\text{Type}} 1_j)$ cancel out.  

Considering the above, let us define Type $2_{j}$ for Type $1_{j}$ using Figs. \ref{type-figu1}, \ref{type-figu2}, and \ref{type12}.  The correspondences of the above cancelations is assured by Lemmas \ref{lemma_type1_lem1}--\ref{lemma_type1_lem3} given below.  For readers who need a categorification with the coefficient $\mathbb{Z}$, the definition of the orientation of the order of negative markers provided in Sec. \ref{z_homology} is required (see Def. \ref{alternative_marker}).  
\begin{lemma}\label{lemma_type1_lem1}
Let $d$ be the coboundary operator of the Khovanov homology of Sec. \ref{preliminary} and let $S$ be one of the enhanced states of the panels shown in Fig. \ref{pre_type1}.  If $d(S)$ contains either {\rm{(a-1)}} or {\rm{(a-2)}} of Fig. \ref{type-figu1}, one of the following is established.  
\end{lemma}
\begin{equation}\label{correspond1}
d(S) = {\text{(a-$i$)}} + {\text{(b-$j$)}} + {\text{other terms}}~ {\text{\rm for}}~i, j \in \{{\rm 1}, {\rm 2}\}
\end{equation}
where {\rm{``other terms''}} do not contain either {\rm{(a-$i$)}} or {\rm{(b-$j$)}} for arbitrary $i$, $j \in \{1, 2\}$.  
\begin{proof}
Every case in Fig. \ref{pre_type1} is verified using Tables \ref{pre_type1_table1}--\ref{pre_type1_table8}.  The latter part of the claim is proved using the types of markers of $S$ in Fig. \ref{pre_type1}.  

For example, when we consider an enhanced state $S$ of the left of Fig. \ref{pre_type1} that will result in Fig. \ref{table_type} (a), i.e., (a-1) or (a-2), we refer to Table \ref{pre_type1_table1}.  If the labels of $S$ are distributed as $a$ $=$ $x$, $b$ $=$ $1$, and $c$ $=$ $x$, we see the third row of Table \ref{pre_type1_table1}, which shows (\ref{correspond1}) in the case, since all possible types --- (a-1), (a-2), or (b-1) --- appear in the third row by the calculus of Fig. \ref{differential2}.  Other cases for Fig. \ref{pre_type1} and Figs. \ref{table_type} (a)--(d) can be seen by following the order of the tables.  
\end{proof}
\begin{lemma}\label{lemma_type1_lem2}
Let $d$ be the coboundary operator of the Khovanov homology of Sec. \ref{preliminary}, and let $S$ be an enhanced state of the panel shown in Fig. \ref{pre_type1b}.  If $d(S)$ contains either {\rm{(a-3)}} or {\rm{(a-4)}} of Fig. \ref{type-figu2}, one of the following is established.  
\begin{equation}\label{correspond2}
\begin{split}
d(S) &= 2 {\text{\rm (a-$i$)}} + {\text{\rm (b-3)}} + {\text{\rm(b-4)}} + {\text{\rm other terms}}~{\rm{for}}~i = 3~{\rm{or}}~4, \\
d(S) &= {\text{\rm (b-3)}} + {\text{\rm (b-4)}} + {\text{\rm other terms}},~{\rm{and}}~\\
d(S) &= {\text{\rm (a-$i$)}} + {\text{\rm (b-$j$)}} + {\text{\rm other terms}}~{\rm{for}}~{i, j \in \{3, 4\}}.\\  
\end{split}
\end{equation}  
\end{lemma}
where the ``other terms'' in (\ref{correspond2}) do not contain the terms (a-3), (a-4), (b-3), or (b-4).
\begin{proof}
Every case of Fig. \ref{pre_type1b} is checked by Tables \ref{pre_type1_table9}--\ref{pre_type1_table12}.  
\end{proof}  
\begin{lemma}\label{lemma_type1_lem3}
Let $d$ be the coboundary operator of the Khovanov homology of Sec. \ref{preliminary}, and let $S$ be an enhanced state of the panel shown of Fig. \ref{pre_type1a}.  If $d(S)$ contains {\rm{(a-3)}} in Fig. \ref{type-figu2}, the following formula is established.  
\begin{equation}\label{corresponde3}
d(S) = {\text{\rm (a-3)}} + {\text{\rm (a-4)}} + {\text{\rm other terms}} 
\end{equation}
where the {\rm{``other terms''}} do not contain {\rm{(b-3)}} or {\rm{(b-4)}}.  
\end{lemma}
\begin{proof}
Every case in Fig. \ref{pre_type1a} is verified by Tables \ref{pre_type1_table13}--\ref{pre_type1_table16}.  The latter part of the claim is given by the types of markers of $S$ in Fig. \ref{pre_type1a}.  
\end{proof}
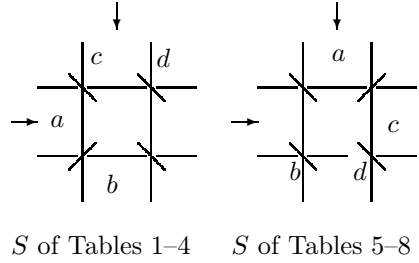
\begin{figure}
\begin{center}
\begin{minipage}{60pt}
\begin{picture}(60,100)
\qbezier(12,32)(17,37)(22,42)
\qbezier(12,68)(17,63)(22,58)
\qbezier(48,32)(43,37)(38,42)
\qbezier(38,68)(43,63)(48,58)
\put(5,48){$a$}
\put(20,72){$c$}
\put(26,23){$b$}
\put(45,71){$d$}
\put(17,20){\line(0,1){60}}
\put(43,20){\line(0,1){60}}
\put(0,37){\line(1,0){15}}
\put(0,63){\line(1,0){15}}
\put(19,37){\line(1,0){22}}
\put(19,63){\line(1,0){22}}
\put(45,37){\line(1,0){15}}
\put(45,63){\line(1,0){15}}
\put(-10,0){$S$ of Tables \ref{pre_type1_table1}--\ref{pre_type1_table4}}
{\linethickness{0.4pt}
\put(-10,50){\vector(1,0){10}}
\put(30,95){\vector(0,-1){10}}
}
\end{picture}
\end{minipage}
\qquad
\begin{minipage}{60pt}
\begin{picture}(60,100)
\qbezier(12,42)(17,37)(22,32)
\qbezier(12,68)(17,63)(22,58)
\qbezier(48,32)(43,37)(38,42)
\qbezier(38,58)(43,63)(48,68)
\put(12,28){$b$}
\put(28,73){$a$}
\put(49,46){$c$}
\put(36,28){$d$}
\put(17,20){\line(0,1){60}}
\put(43,20){\line(0,1){60}}
\put(0,37){\line(1,0){15}}
\put(0,63){\line(1,0){15}}
\put(19,37){\line(1,0){15}}
\put(19,63){\line(1,0){22}}
\put(19,63){\line(1,0){22}}
\put(45,37){\line(1,0){15}}
\put(45,63){\line(1,0){15}}
\put(-10,0){$S$ of Tables \ref{pre_type1_table5}--\ref{pre_type1_table8}}
{\linethickness{0.4pt}
\put(-10,50){\vector(1,0){10}}
\put(30,95){\vector(0,-1){10}}
}
\end{picture}
\end{minipage}
\end{center}
\caption{Enhanced states $S$ (characterized by the local part), in order from the left, $S$ of Tables \ref{pre_type1_table1}--\ref{pre_type1_table4} and $S$ of Tables \ref{pre_type1_table5}--\ref{pre_type1_table8}.}\label{pre_type1}
\end{figure}
\begin{figure}[htbp]
\centering
\begin{minipage}{60pt}
\begin{picture}(60,100)
\qbezier(12,32)(17,37)(22,42)
\qbezier(12,68)(17,63)(22,58)
\qbezier(48,32)(43,37)(38,42)
\qbezier(38,58)(43,63)(48,68)
\put(5,48){$p$}
\put(28,73){$q$}
\put(49,46){$x$}
\put(26,23){$1$}
\put(17,20){\line(0,1){60}}
\put(43,20){\line(0,1){60}}
\put(0,37){\line(1,0){15}}
\put(0,63){\line(1,0){15}}
\put(19,37){\line(1,0){22}}
\put(19,63){\line(1,0){22}}
\put(45,37){\line(1,0){15}}
\put(45,63){\line(1,0){15}}
\put(30,20){\oval(26,6)[b]}
\put(60,50){\oval(6,26)[r]}
\put(23,0){(a-1)}
{\linethickness{0.4pt}
\put(-10,50){\vector(1,0){10}}
\put(30,95){\vector(0,-1){10}}}
\end{picture}
\end{minipage}
\qquad
\begin{minipage}{60pt}
\begin{picture}(60,100)
\qbezier(12,32)(17,37)(22,42)
\qbezier(12,68)(17,63)(22,58)
\qbezier(48,32)(43,37)(38,42)
\qbezier(38,58)(43,63)(48,68)
\put(5,48){$p$}
\put(28,73){$q$}
\put(49,46){$1$}
\put(26,23){$x$}
\put(17,20){\line(0,1){60}}
\put(43,20){\line(0,1){60}}
\put(0,37){\line(1,0){15}}
\put(0,63){\line(1,0){15}}
\put(19,37){\line(1,0){22}}
\put(19,63){\line(1,0){22}}
\put(45,37){\line(1,0){15}}
\put(45,63){\line(1,0){15}}
\put(30,20){\oval(26,6)[b]}
\put(60,50){\oval(6,26)[r]}
\put(23,0){(a-2)}
\end{picture}
\end{minipage}
\qquad
\begin{minipage}{60pt}
\begin{picture}(60,100)
\qbezier(12,32)(17,37)(22,42)
\qbezier(12,68)(17,63)(22,58)
\qbezier(38,32)(43,37)(48,42)
\qbezier(38,68)(43,63)(48,58)
\put(5,48){$p$}
\put(45,73){$q$}
\put(19,73){$r$}
\put(17,20){\line(0,1){60}}
\put(43,20){\line(0,1){60}}
\put(0,37){\line(1,0){15}}
\put(0,63){\line(1,0){15}}
\put(19,37){\line(1,0){22}}
\put(19,63){\line(1,0){22}}
\put(45,37){\line(1,0){15}}
\put(45,63){\line(1,0){15}}
\put(23,0){(b-1)}
\end{picture}
\end{minipage}
\qquad
\begin{minipage}{60pt}
\begin{picture}(60,100)
\qbezier(12,42)(17,37)(22,32)
\qbezier(12,68)(17,63)(22,58)
\qbezier(38,32)(43,37)(48,42)
\qbezier(38,58)(43,63)(48,68)
\put(0,48){$p$}
\put(28,73){$q$}
\put(17,20){\line(0,1){60}}
\put(43,20){\line(0,1){60}}
\put(0,37){\line(1,0){15}}
\put(0,63){\line(1,0){15}}
\put(19,37){\line(1,0){22}}
\put(19,63){\line(1,0){22}}
\put(45,37){\line(1,0){15}}
\put(45,63){\line(1,0){15}}
\put(23,0){(b-2)}
\end{picture}
\end{minipage}
\caption{(a-1), (a-2): Type $1_{j}$.  (b-1), (b-2): Type $2_{j}$.  }\label{type-figu1}
\end{figure}
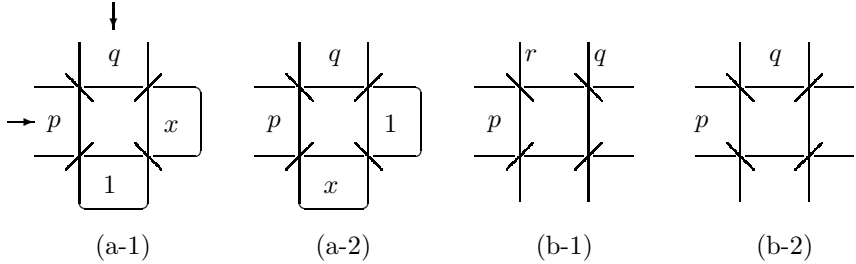
\begin{figure}[htbp]
\begin{center}
\begin{minipage}{60pt}
\begin{picture}(60,100)
\qbezier(12,32)(17,37)(22,42)
\qbezier(12,58)(17,63)(22,68)
\qbezier(48,32)(43,37)(38,42)
\qbezier(38,58)(43,63)(48,68)
\put(10,68){$c$}
\put(35,68){$b$}
\put(49,46){$a$}
\put(26,23){$d$}
\put(0,37){\line(1,0){60}}
\put(0,63){\line(1,0){60}}
\put(17,20){\line(0,1){15}}
\put(43,20){\line(0,1){15}}
\put(17,39){\line(0,1){22}}
\put(43,39){\line(0,1){22}}
\put(17,65){\line(0,1){15}}
\put(43,65){\line(0,1){15}}
\put(-10,0){$S$ of Tables \ref{pre_type1_table9}--\ref{pre_type1_table12}}
{\linethickness{0.4pt}
\put(-10,50){\vector(1,0){10}}
\put(30,95){\vector(0,-1){10}}}
\end{picture}
\end{minipage}
\end{center}
\caption{Enhanced states $S$ (characterized by the local part) of Tables \ref{pre_type1_table9}--\ref{pre_type1_table12}.}\label{pre_type1b}
\end{figure}
\begin{figure}[hbtp]
\begin{center}
\begin{minipage}{60pt}
\begin{picture}(60,100)
\qbezier(12,32)(17,37)(22,42)
\qbezier(12,68)(17,63)(22,58)
\qbezier(38,32)(43,37)(48,42)
\qbezier(38,58)(43,63)(48,68)
\put(5,48){$a$}
\put(28,73){$c$}
\put(26,23){$b$}
\put(0,37){\line(1,0){60}}
\put(0,63){\line(1,0){60}}
\put(17,20){\line(0,1){15}}
\put(43,20){\line(0,1){15}}
\put(17,39){\line(0,1){22}}
\put(43,39){\line(0,1){22}}
\put(17,65){\line(0,1){15}}
\put(43,65){\line(0,1){15}}
\put(-10,0){$S$ of Tables \ref{pre_type1_table13}--\ref{pre_type1_table16}}
{\linethickness{0.4pt}
\put(-10,50){\vector(1,0){10}}
\put(30,95){\vector(0,-1){10}}
}
\end{picture}
\end{minipage}
\end{center}
\caption{Enhanced states $S$ (characterized by the local part) of Tables \ref{pre_type1_table13}--\ref{pre_type1_table16}.}\label{pre_type1a}
\end{figure}

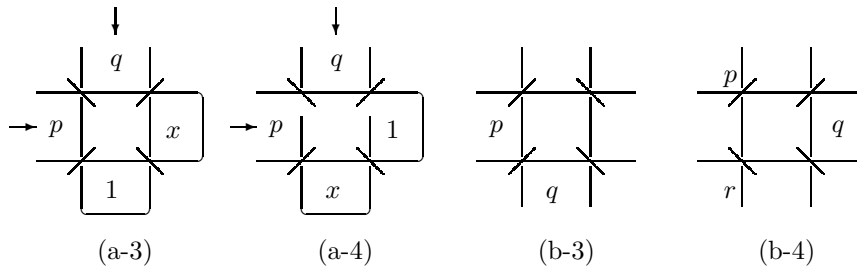
\begin{figure}[htbp]
\centering
\begin{minipage}{60pt}
\begin{picture}(60,100)
\qbezier(12,32)(17,37)(22,42)
\qbezier(12,68)(17,63)(22,58)
\qbezier(48,32)(43,37)(38,42)
\qbezier(38,58)(43,63)(48,68)
\put(5,48){$p$}
\put(28,73){$q$}
\put(49,46){$x$}
\put(26,23){$1$}
\put(0,37){\line(1,0){60}}
\put(0,63){\line(1,0){60}}
\put(17,20){\line(0,1){15}}
\put(43,20){\line(0,1){15}}
\put(17,39){\line(0,1){22}}
\put(43,39){\line(0,1){22}}
\put(17,65){\line(0,1){15}}
\put(43,65){\line(0,1){15}}
\put(30,20){\oval(26,6)[b]}
\put(60,50){\oval(6,26)[r]}
\put(23,0){(a-3)}
{\linethickness{0.4pt}
\put(-10,50){\vector(1,0){10}}
\put(30,95){\vector(0,-1){10}}}
\end{picture}
\end{minipage}
\qquad
\begin{minipage}{60pt}
\begin{picture}(60,100)
\qbezier(12,32)(17,37)(22,42)
\qbezier(12,68)(17,63)(22,58)
\qbezier(48,32)(43,37)(38,42)
\qbezier(38,58)(43,63)(48,68)
\put(5,48){$p$}
\put(28,73){$q$}
\put(49,46){$1$}
\put(26,23){$x$}
\put(0,37){\line(1,0){60}}
\put(0,63){\line(1,0){60}}
\put(17,20){\line(0,1){15}}
\put(43,20){\line(0,1){15}}
\put(17,39){\line(0,1){15}}
\put(43,39){\line(0,1){15}}
\put(17,65){\line(0,1){15}}
\put(43,65){\line(0,1){15}}
\put(30,20){\oval(26,6)[b]}
\put(60,50){\oval(6,26)[r]}
\put(23,0){(a-4)}
{\linethickness{0.4pt}
\put(-10,50){\vector(1,0){10}}
\put(30,95){\vector(0,-1){10}}
}
\end{picture}
\end{minipage}
\qquad
\begin{minipage}{60pt}
\begin{picture}(60,100)
\qbezier(12,32)(17,37)(22,42)
\qbezier(12,58)(17,63)(22,68)
\qbezier(48,32)(43,37)(38,42)
\qbezier(38,68)(43,63)(48,58)
\put(5,48){$p$}
\put(26,23){$q$}
\put(0,37){\line(1,0){60}}
\put(0,63){\line(1,0){60}}
\put(17,20){\line(0,1){15}}
\put(43,20){\line(0,1){15}}
\put(17,39){\line(0,1){22}}
\put(43,39){\line(0,1){22}}
\put(17,65){\line(0,1){15}}
\put(43,65){\line(0,1){15}}
\put(23,0){(b-3)}
\end{picture}
\end{minipage}
\qquad
\begin{minipage}{60pt}
\begin{picture}(60,100)
\qbezier(12,42)(17,37)(22,32)
\qbezier(12,58)(17,63)(22,68)
\qbezier(48,32)(43,37)(38,42)
\qbezier(48,68)(43,63)(38,58)
\put(10,67){$p$}
\put(51,48){$q$}
\put(10,23){$r$}
\put(0,37){\line(1,0){60}}
\put(0,63){\line(1,0){60}}
\put(17,20){\line(0,1){15}}
\put(43,20){\line(0,1){15}}
\put(17,39){\line(0,1){22}}
\put(43,39){\line(0,1){22}}
\put(17,65){\line(0,1){15}}
\put(43,65){\line(0,1){15}}
\put(23,0){(b-4)}
\end{picture}
\end{minipage}
\caption{(a-3), (a-4): Type $1_{j}$.  (b-3), (b-4): Type $2_{j}$.  }\label{type-figu2}
\end{figure}

\begin{figure}[htbp]
\centering
\begin{minipage}{60pt}
\begin{picture}(60,80)
\put(0,40){\line(1,0){14}}
\put(18,40){\line(1,0){24}}
\put(46,40){\line(1,0){14}}
\put(44,20){\line(0,1){40}}
\put(16,20){\line(0,1){40}}
\qbezier(36,32)(44,40)(52,48)
\qbezier(8,48)(16,40)(24,32)
\put(29,50){$p$}
\put(49,30){$q$}
\qbezier(16,19)(16,18)(16,17)
\qbezier(16,16)(16,15)(16,14)
\qbezier(16,13)(16,12)(16,11)
\qbezier(44,19)(44,18)(44,17)
\qbezier(44,16)(44,15)(44,14)
\qbezier(44,13)(44,12)(44,11)
\put(30,10){\oval(28,6)[b]}
\qbezier(16,60.8)(16,61.3)(16,62)
\qbezier(16,62)(16,62.5)(16.5,63.2)
\qbezier(17,64)(18,65)(19,65)
\qbezier(20.3,65)(21.3,65)(22.3,65)
\qbezier(23.8,65)(24.8,65)(25.8,65)
\qbezier(27.6,65)(28.6,65)(29.6,65)
\qbezier(31,65)(32,65)(33,65)
\qbezier(34.3,65)(35.3,65)(36.3,65)
\qbezier(38,65)(39,65)(40,65)
\qbezier(41.5,64.8)(42,64.8)(43,64)
\qbezier(43.6,63)(43.8,63.2)(44,61)
\put(23,-5){(a-5)}
\qbezier(0,40)(-0.5,40)(-1,43)
\qbezier(-1,45)(-1,46)(-1,47)
\qbezier(-1,49)(-1,50)(-1,51)
\qbezier(-1,53)(-1,54)(-1,55)
\qbezier(-0.7,57)(0,59)(0.8,61)
\qbezier(1.8,63)(2.3,64)(3.3,65)
\qbezier(5,67)(6,68)(7.5,69)
\qbezier(9.5,70.2)(10.5,70.8)(11.5,71.1)
\qbezier(13.5,71.8)(14.5,72)(15.5,72.2)
\qbezier(17.5,72.5)(18.5,72.5)(19.5,72.5)
\qbezier(21.5,72.5)(22.5,72.5)(23.5,72.5)
\qbezier(25.5,72.5)(26.5,72.5)(27.5,72.5)
\qbezier(29.5,72.5)(30.5,72.5)(31.5,72.5)
\qbezier(33.5,72.5)(34.5,72.5)(35.5,72.5)
\qbezier(37.5,72.5)(38.5,72.5)(39.5,72.5)
\qbezier(41.5,72.5)(42.5,72.5)(43.5,72.5)
\qbezier(45.5,72)(46.5,71.8)(47.5,71.4)
\qbezier(48.5,71)(49.5,70.6)(50.5,70.1)
\qbezier(52.5,68.8)(53.5,68.3)(54.5,67.2)
\qbezier(56.5,65.2)(57.5,64)(58.5,62.3)
\qbezier(59.4,60)(60,59)(60.2,58)
\qbezier(60.6,56)(60.9,55)(60.8,54)
\qbezier(61,52)(61,51)(61,50)
\qbezier(61,48)(61,46.5)(61,45)
\qbezier(61,43)(61,40)(60,40)
{\linethickness{0.4pt}
\put(30,20){\vector(0,-1){10}}}
\end{picture}
\end{minipage}
\quad
\begin{minipage}{60pt}
\begin{picture}(60,80)
\put(0,40){\line(1,0){14}}
\put(18,40){\line(1,0){24}}
\put(46,40){\line(1,0){14}}
\put(44,20){\line(0,1){40}}
\put(16,20){\line(0,1){40}}
\qbezier(8,32)(16,40)(24,48)
\qbezier(36,48)(44,40)(52,32)
\put(27,25){$q'$}
\qbezier(16,19)(16,18)(16,17)
\qbezier(16,16)(16,15)(16,14)
\qbezier(16,13)(16,12)(16,11)
\qbezier(44,19)(44,18)(44,17)
\qbezier(44,16)(44,15)(44,14)
\qbezier(44,13)(44,12)(44,11)
\put(30,10){\oval(28,6)[b]}
\qbezier(0,40)(-0.5,40)(-1,43)
\qbezier(-1,45)(-1,46)(-1,47)
\qbezier(-1,49)(-1,50)(-1,51)
\qbezier(-1,53)(-1,54)(-1,55)
\qbezier(-0.7,57)(0,59)(0.8,61)
\qbezier(1.8,63)(2.3,64)(3.3,65)
\qbezier(5,67)(6,68)(7.5,69)
\qbezier(9.5,70.2)(10.5,70.8)(11.5,71.1)
\qbezier(13.5,71.8)(14.5,72)(15.5,72.2)
\qbezier(17.5,72.5)(18.5,72.5)(19.5,72.5)
\qbezier(21.5,72.5)(22.5,72.5)(23.5,72.5)
\qbezier(25.5,72.5)(26.5,72.5)(27.5,72.5)
\qbezier(29.5,72.5)(30.5,72.5)(31.5,72.5)
\qbezier(33.5,72.5)(34.5,72.5)(35.5,72.5)
\qbezier(37.5,72.5)(38.5,72.5)(39.5,72.5)
\qbezier(41.5,72.5)(42.5,72.5)(43.5,72.5)
\qbezier(45.5,72)(46.5,71.8)(47.5,71.4)
\qbezier(48.5,71)(49.5,70.6)(50.5,70.1)
\qbezier(52.5,68.8)(53.5,68.3)(54.5,67.2)
\qbezier(56.5,65.2)(57.5,64)(58.5,62.3)
\qbezier(59.4,60)(60,59)(60.2,58)
\qbezier(60.6,56)(60.9,55)(60.8,54)
\qbezier(61,52)(61,51)(61,50)
\qbezier(61,48)(61,46.5)(61,45)
\qbezier(61,43)(61,40)(60,40)
\put(23,-5){(b-5)}
\qbezier(16,60.8)(16,61.3)(16,62)
\qbezier(16,62)(16,62.5)(16.5,63.2)
\qbezier(17,64)(18,65)(19,65)
\qbezier(20.3,65)(21.3,65)(22.3,65)
\qbezier(23.8,65)(24.8,65)(25.8,65)
\qbezier(27.6,65)(28.6,65)(29.6,65)
\qbezier(31,65)(32,65)(33,65)
\qbezier(34.3,65)(35.3,65)(36.3,65)
\qbezier(38,65)(39,65)(40,65)
\qbezier(41.5,64.8)(42,64.8)(43,64)
\qbezier(43.6,63)(43.8,63.2)(44,61)
\put(49,50){$p'$}
{\linethickness{0.4pt}
\put(30,20){\vector(0,-1){10}}}
\end{picture}
\end{minipage}
\qquad
\begin{minipage}{60pt}
\begin{picture}(60,80)
\put(0,40){\line(1,0){14}}
\put(18,40){\line(1,0){24}}
\put(46,40){\line(1,0){14}}
\put(44,20){\line(0,1){40}}
\put(16,20){\line(0,1){40}}
\qbezier(36,32)(44,40)(52,48)
\qbezier(8,48)(16,40)(24,32)
\put(27,25){$p$}
\qbezier(16,19)(16,18)(16,17)
\qbezier(16,16)(16,15)(16,14)
\qbezier(16,13)(16,12)(16,11)
\qbezier(44,19)(44,18)(44,17)
\qbezier(44,16)(44,15)(44,14)
\qbezier(44,13)(44,12)(44,11)
\put(30,10){\oval(28,6)[b]}
\qbezier(0.5,40)(-0.8,40)(-1,42)
\qbezier(-0.8,44)(-0.7,45)(-0.3,46)
\qbezier(0.2,48)(0.5,49)(1.2,50)
\qbezier(2.2,52)(2.8,53)(3.8,54)
\qbezier(6,56.1)(7,57.2)(8,57.6)
\qbezier(10,58.8)(11,59.3)(12,59.8)
\qbezier(14,60.3)(16,60.5)(16,59.5)
\qbezier(44,59)(44,61)(45,60.5)
\qbezier(46,60.2)(47,60)(48,59.6)
\qbezier(50,58.8)(51,58.5)(52,57.6)
\qbezier(54,56.1)(55,55.5)(56.5,53.6)
\qbezier(58,51.5)(59,50)(59.5,48.3)
\qbezier(60.4,46)(60.7,45)(60.8,44)
\qbezier(61,42)(61,40)(59.8,40)
\put(23,-5){(a-6)}
{\linethickness{0.4pt}
\put(30,20){\vector(0,-1){10}}}
\end{picture}
\end{minipage}
\quad
\begin{minipage}{60pt}
\begin{picture}(60,80)
\put(0,40){\line(1,0){14}}
\put(18,40){\line(1,0){24}}
\put(46,40){\line(1,0){14}}
\put(44,20){\line(0,1){40}}
\put(16,20){\line(0,1){40}}
\qbezier(8,32)(16,40)(24,48)
\qbezier(36,48)(44,40)(52,32)
\put(26,25){$q'$}
\qbezier(16,19)(16,18)(16,17)
\qbezier(16,16)(16,15)(16,14)
\qbezier(16,13)(16,12)(16,11)
\qbezier(44,19)(44,18)(44,17)
\qbezier(44,16)(44,15)(44,14)
\qbezier(44,13)(44,12)(44,11)
\put(30,10){\oval(28,6)[b]}
\qbezier(0.5,40)(-0.8,40)(-1,42)
\qbezier(-0.8,44)(-0.7,45)(-0.3,46)
\qbezier(0.2,48)(0.5,49)(1.2,50)
\qbezier(2.2,52)(2.8,53)(3.8,54)
\qbezier(6,56.1)(7,57.2)(8,57.6)
\qbezier(10,58.8)(11,59.3)(12,59.8)
\qbezier(14,60.3)(16,60.5)(16,59.5)
\qbezier(44,59)(44,61)(45,60.5)
\qbezier(46,60.2)(47,60)(48,59.6)
\qbezier(50,58.8)(51,58.5)(52,57.6)
\qbezier(54,56.1)(55,55.5)(56.5,53.6)
\qbezier(58,51.5)(59,50)(59.5,48.3)
\qbezier(60.4,46)(60.7,45)(60.8,44)
\qbezier(61,42)(61,40)(59.8,40)
\put(23,-5){(b-6)}
{\linethickness{0.4pt}
\put(30,20){\vector(0,-1){10}}}
\put(5,48){$p'$}
\put(45,48){$r'$}
\end{picture}
\end{minipage}\\
\begin{minipage}{60pt}
\begin{picture}(60,100)
\put(0,40){\line(1,0){60}}
\put(16,20){\line(0,1){18}}
\put(44,20){\line(0,1){18}}
\put(16,42){\line(0,1){18}}
\put(44,42){\line(0,1){18}}
\qbezier(36,32)(44,40)(52,48)
\qbezier(8,48)(16,40)(24,32)
\put(29,50){$q$}
\put(49,30){$p$}
\qbezier(16,19)(16,18)(16,17)
\qbezier(16,16)(16,15)(16,14)
\qbezier(16,13)(16,12)(16,11)
\qbezier(44,19)(44,18)(44,17)
\qbezier(44,16)(44,15)(44,14)
\qbezier(44,13)(44,12)(44,11)
\put(30,10){\oval(28,6)[b]}
\qbezier(16,60.8)(16,61.3)(16,62)
\qbezier(16,62)(16,62.5)(16.5,63.2)
\qbezier(17,64)(18,65)(19,65)
\qbezier(20.3,65)(21.3,65)(22.3,65)
\qbezier(23.8,65)(24.8,65)(25.8,65)
\qbezier(27.6,65)(28.6,65)(29.6,65)
\qbezier(31,65)(32,65)(33,65)
\qbezier(34.3,65)(35.3,65)(36.3,65)
\qbezier(38,65)(39,65)(40,65)
\qbezier(41.5,64.8)(42,64.8)(43,64)
\qbezier(43.6,63)(43.8,63.2)(44,61)
\put(23,-5){(a-7)}
\qbezier(0,40)(-0.5,40)(-1,43)
\qbezier(-1,45)(-1,46)(-1,47)
\qbezier(-1,49)(-1,50)(-1,51)
\qbezier(-1,53)(-1,54)(-1,55)
\qbezier(-0.7,57)(0,59)(0.8,61)
\qbezier(1.8,63)(2.3,64)(3.3,65)
\qbezier(5,67)(6,68)(7.5,69)
\qbezier(9.5,70.2)(10.5,70.8)(11.5,71.1)
\qbezier(13.5,71.8)(14.5,72)(15.5,72.2)
\qbezier(17.5,72.5)(18.5,72.5)(19.5,72.5)
\qbezier(21.5,72.5)(22.5,72.5)(23.5,72.5)
\qbezier(25.5,72.5)(26.5,72.5)(27.5,72.5)
\qbezier(29.5,72.5)(30.5,72.5)(31.5,72.5)
\qbezier(33.5,72.5)(34.5,72.5)(35.5,72.5)
\qbezier(37.5,72.5)(38.5,72.5)(39.5,72.5)
\qbezier(41.5,72.5)(42.5,72.5)(43.5,72.5)
\qbezier(45.5,72)(46.5,71.8)(47.5,71.4)
\qbezier(48.5,71)(49.5,70.6)(50.5,70.1)
\qbezier(52.5,68.8)(53.5,68.3)(54.5,67.2)
\qbezier(56.5,65.2)(57.5,64)(58.5,62.3)
\qbezier(59.4,60)(60,59)(60.2,58)
\qbezier(60.6,56)(60.9,55)(60.8,54)
\qbezier(61,52)(61,51)(61,50)
\qbezier(61,48)(61,46.5)(61,45)
\qbezier(61,43)(61,40)(60,40)
{\linethickness{0.4pt}
\put(30,20){\vector(0,-1){10}}
}
\end{picture}
\end{minipage}
\quad
\begin{minipage}{60pt}
\begin{picture}(60,100)
\put(0,40){\line(1,0){60}}
\put(16,20){\line(0,1){18}}
\put(44,20){\line(0,1){18}}
\put(16,42){\line(0,1){18}}
\put(44,42){\line(0,1){18}}
\qbezier(8,32)(16,40)(24,48)
\qbezier(36,48)(44,40)(52,32)
\put(27,25){$q'$}
\qbezier(16,19)(16,18)(16,17)
\qbezier(16,16)(16,15)(16,14)
\qbezier(16,13)(16,12)(16,11)
\qbezier(44,19)(44,18)(44,17)
\qbezier(44,16)(44,15)(44,14)
\qbezier(44,13)(44,12)(44,11)
\put(30,10){\oval(28,6)[b]}
\qbezier(0,40)(-0.5,40)(-1,43)
\qbezier(-1,45)(-1,46)(-1,47)
\qbezier(-1,49)(-1,50)(-1,51)
\qbezier(-1,53)(-1,54)(-1,55)
\qbezier(-0.7,57)(0,59)(0.8,61)
\qbezier(1.8,63)(2.3,64)(3.3,65)
\qbezier(5,67)(6,68)(7.5,69)
\qbezier(9.5,70.2)(10.5,70.8)(11.5,71.1)
\qbezier(13.5,71.8)(14.5,72)(15.5,72.2)
\qbezier(17.5,72.5)(18.5,72.5)(19.5,72.5)
\qbezier(21.5,72.5)(22.5,72.5)(23.5,72.5)
\qbezier(25.5,72.5)(26.5,72.5)(27.5,72.5)
\qbezier(29.5,72.5)(30.5,72.5)(31.5,72.5)
\qbezier(33.5,72.5)(34.5,72.5)(35.5,72.5)
\qbezier(37.5,72.5)(38.5,72.5)(39.5,72.5)
\qbezier(41.5,72.5)(42.5,72.5)(43.5,72.5)
\qbezier(45.5,72)(46.5,71.8)(47.5,71.4)
\qbezier(48.5,71)(49.5,70.6)(50.5,70.1)
\qbezier(52.5,68.8)(53.5,68.3)(54.5,67.2)
\qbezier(56.5,65.2)(57.5,64)(58.5,62.3)
\qbezier(59.4,60)(60,59)(60.2,58)
\qbezier(60.6,56)(60.9,55)(60.8,54)
\qbezier(61,52)(61,51)(61,50)
\qbezier(61,48)(61,46.5)(61,45)
\qbezier(61,43)(61,40)(61,40)
\put(23,-5){(b-7)}
\qbezier(16,60.8)(16,61.3)(16,62)
\qbezier(16,62)(16,62.5)(16.5,63.2)
\qbezier(17,64)(18,65)(19,65)
\qbezier(20.3,65)(21.3,65)(22.3,65)
\qbezier(23.8,65)(24.8,65)(25.8,65)
\qbezier(27.6,65)(28.6,65)(29.6,65)
\qbezier(31,65)(32,65)(33,65)
\qbezier(34.3,65)(35.3,65)(36.3,65)
\qbezier(38,65)(39,65)(40,65)
\qbezier(41.5,64.8)(42,64.8)(43,64)
\qbezier(43.6,63)(44,63.2)(44,61)
\put(49,50){$p'$}
{\linethickness{0.4pt}
\put(30,20){\vector(0,-1){10}}
}
\end{picture}
\end{minipage}
\quad
\begin{minipage}{60pt}
\begin{picture}(60,100)
\put(0,40){\line(1,0){60}}
\put(16,20){\line(0,1){18}}
\put(44,20){\line(0,1){18}}
\put(16,42){\line(0,1){18}}
\put(44,42){\line(0,1){18}}
\qbezier(36,32)(44,40)(52,48)
\qbezier(8,48)(16,40)(24,32)
\put(27,25){$p$}
\qbezier(16,19)(16,18)(16,17)
\qbezier(16,16)(16,15)(16,14)
\qbezier(16,13)(16,12)(16,11)
\qbezier(44,19)(44,18)(44,17)
\qbezier(44,16)(44,15)(44,14)
\qbezier(44,13)(44,12)(44,11)
\put(30,10){\oval(28,6)[b]}
\qbezier(0.5,40)(-0.8,40)(-1,42)
\qbezier(-0.8,44)(-0.7,45)(-0.3,46)
\qbezier(0.2,48)(0.5,49)(1.2,50)
\qbezier(2.2,52)(2.8,53)(3.8,54)
\qbezier(6,56.1)(7,57.2)(8,57.6)
\qbezier(10,58.8)(11,59.3)(12,59.8)
\qbezier(14,60.3)(16,60.5)(16,59.5)
\qbezier(44,59)(44,61)(45,60.5)
\qbezier(46,60.5)(47,60)(48,59.6)
\qbezier(50,58.8)(51,58.5)(52,57.6)
\qbezier(54,56.1)(55,55.5)(56.5,53.6)
\qbezier(58,51.5)(59,50)(59.5,48.3)
\qbezier(60.4,46)(60.7,45)(60.8,44)
\qbezier(61,42)(61,40)(59.8,40)
\put(23,-5){(a-8)}
{\linethickness{0.4pt}
\put(30,20){\vector(0,-1){10}}
}
\end{picture}
\end{minipage}
\quad
\begin{minipage}{60pt}
\begin{picture}(60,100)
\put(0,40){\line(1,0){60}}
\put(16,20){\line(0,1){18}}
\put(44,20){\line(0,1){18}}
\put(16,42){\line(0,1){18}}
\put(44,42){\line(0,1){18}}
\qbezier(8,32)(16,40)(24,48)
\qbezier(36,48)(44,40)(52,32)
\put(26,25){$q'$}
\qbezier(16,19)(16,18)(16,17)
\qbezier(16,16)(16,15)(16,14)
\qbezier(16,13)(16,12)(16,11)
\qbezier(44,19)(44,18)(44,17)
\qbezier(44,16)(44,15)(44,14)
\qbezier(44,13)(44,12)(44,11)
\put(30,10){\oval(28,6)[b]}
\qbezier(0.5,40)(-0.8,40)(-1,42)
\qbezier(-0.8,44)(-0.7,45)(-0.3,46)
\qbezier(0.2,48)(0.5,49)(1.2,50)
\qbezier(2.2,52)(2.8,53)(3.8,54)
\qbezier(6,56.1)(7,57.2)(8,57.6)
\qbezier(10,58.8)(11,59.3)(12,59.8)
\qbezier(14,60.3)(16,60.5)(16,59.5)
\qbezier(44,59)(44,61)(45,60.5)
\qbezier(46,60.2)(47,60)(48,59.6)
\qbezier(50,58.8)(51,58.5)(52,57.6)
\qbezier(54,56.1)(55,55.5)(56.5,53.6)
\qbezier(58,51.5)(59,50)(59.5,48.3)
\qbezier(60.4,46)(60.7,45)(60.8,44)
\qbezier(61,42)(61,40)(59.8,40)
\put(23,-5){(b-8)}
{\linethickness{0.4pt}
\put(30,20){\vector(0,-1){10}}}
\put(5,48){$r'$}
\put(47,48){$p'$}
\end{picture}
\end{minipage}
\caption{Enhanced states related to contracted strands.  The arrows indicate the orientation of the contracted strands.  In these figures, $p$, $q$, $r$, $p'$, $q'$, and $r'$ $=$ $x$ or $1$.  (a-5) or (a-7): A part of Type $1_{j}$ $p \otimes q$.  (b-5) ((b-7)): The part of Type $2_{j}$ corresponding to (a-5) ((a-7)).  (a-6) or (a-8): A part of Type $1_{j}$ generating only one circles.  (b-6) ((b-8)): The part of Type $2_{j}$ corresponding to (a-6) ((a-8)).}\label{type12}
\end{figure}
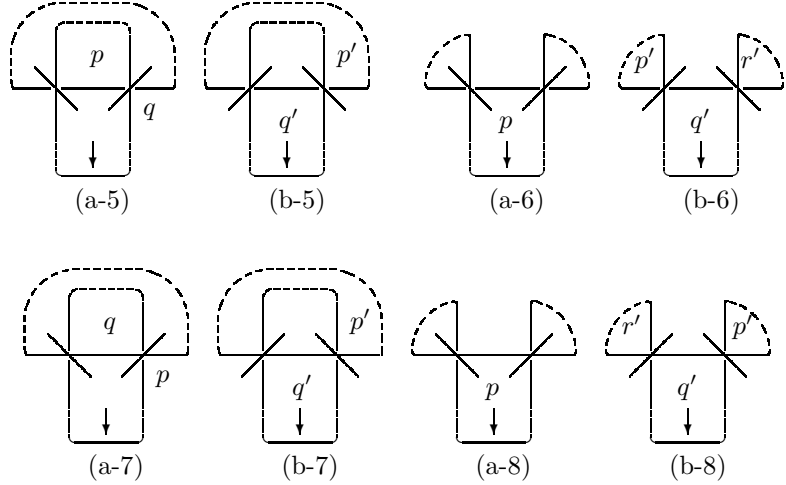

Now, we provide definitions for Types $1_j$ and $2_j$.  
\begin{definition}\label{def_type1_and_2}
Let us fix the tuple of nonnegative integers $\textbf{n}$ corresponding to the colored Jones polynomial $J_{\textbf{n}}$.  Consider the enhanced states of Type $1_j$ in the link diagram $D^{\textbf{n} - 2 \textbf{k}}$.  Using Tables \ref{pre_type1_table1}--\ref{pre_type1_table8}, we locally replace exactly one arbitrary (a-1) or (a-2) panel in Fig. \ref{type-figu1} with (b-1) or (b-2) placed on the same line in each table.  The enhanced state given by this replacement is called Type $2_j$.  Similarly, using Tables \ref{pre_type1_table9}--\ref{pre_type1_table12}, we can locally replace an (a-3) or (a-4) panel in Fig. \ref{type-figu2} with a (b-3) or (b-4) panel to give the Type $2_j$ enhanced state.  In this case, (b-3) and (b-4) in the eighth and ninth lines of Table \ref{pre_type1_table9} are defined by the third--sixth lines of the same table.  Using Fig. \ref{type12}, locally replacing (a-$i$) with (b-$i$) at one place for every $i$ $\in$ $\{5, 6, 7, 8\}$ gives us the Type $2_j$ the enhanced state.  
\end{definition}

\section{Operators on abelian groups of graphs of cables of a link diagram}
\subsection{Case of knots and the coefficient $\mathbb{Z}_2$}\label{coeff_z_2}
To understand the concept of construction of a Khovanov complex of the colored Jones polynomial, we first consider the case of knots and $\mathbb{Z}_2$.  From (\ref{colored_jones}), the colored Jones polynomial $J_{n}$ of a knot diagram $D$ is written as
\begin{equation}
J_{n}(D) = \sum_{k = 0}^{\lfloor n/2 \rfloor} (-1)^{k} \begin{pmatrix} n-k \\ k \end{pmatrix} J(D^{n - 2k}).  
\end{equation}
The binomial coefficient $\begin{pmatrix} n-k \\ k \end{pmatrix}$ is the number of ways in which $k$ pairs of neighbors can be selected from $n$ dots placed on a vertical line, where each dot appears in at most one pair.  Fig. \ref{cable_ex_knot} shows an example for $n$ $=$ $4$.   In Fig. \ref{cable_ex_knot}, (a) corresponds to the binomial $\begin{pmatrix} 4 - 0 \\ 0 \end{pmatrix}$ $=$ $1$; (b), (c), and (d) correspond to $\begin{pmatrix} 4-1 \\ 1 \end{pmatrix}$ $=$ $3$; and (e) corresponds to $\begin{pmatrix} 4-2 \\ 2 \end{pmatrix}$ $=$ $1$.  
\begin{figure}[h!]
\begin{center}
\qquad\qquad\qquad\quad
\includegraphics[width=8cm, bb=0 0 253 73]{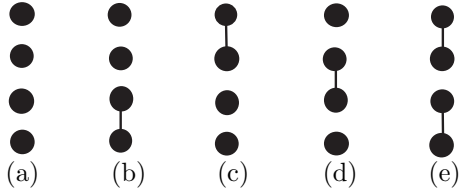}
\caption{Graphs correspond to binomials in the case $n$ $=$ $4$.  (a): $k$ $=$ $0$; (b), (c), (d): $k$ $=$ $1$; and (e): $k$ $=$ $2$.  }\label{cable_ex_knot}
\begin{picture}(0,0)
\put(-80,37){(a)}
\put(-40,37){(b)}
\put(0,37){(c)}
\put(40,37){(d)}
\put(80,37){(e)}
\end{picture}
\end{center}
\end{figure}
Thus, for a knot $K$ oriented as in Fig. \ref{ori_cable} and its diagram $D$, 
\begin{equation}\label{ex_cable_4}
J_4(K) = J(D^{4}) - 3 J(D^{2}) + 1.  
\end{equation}
Here, note that $w(D^{4})$ $=$ $w(D^{2})$ $=$ $0$.  

The abelian group over $\mathbb{Z}_2$ generated by graphs associated with the cables $D^{n - 2k}$ of a diagram $D$ of an oriented knot $K$ is denoted by $\Gamma_{n}(D; \mathbb{Z}_2)$.  The subgroup of $\Gamma_n(D; \mathbb{Z}_2)$ generated by each graph with $k$ edges is denoted by $\Gamma_{n}^{k}(D; \mathbb{Z}_2)$, and its edged graph with $k$ edges is called a $k$-{\itshape{pairing}}, or simply a {\it{pairing}}.  Further, to avoid our confusion, we can refer to an $m$-pairing, where $m$ $=$ $k+l$, as $(k+l)$-pairing.  Since the definition of a $k$-pairing depends only on the integer $n$ and the number of components of $D^{n - 2 k}$, $k$-pairings are knot invariants.  

For the $k$-pairings $\textbf{s}$ $\in$ $\Gamma_{n}^{k}(D; \mathbb{Z}_2)$ and $\textbf{t}$ $\in$ $\Gamma^{k+1}_{n}(D; \mathbb{Z}_2)$, we define the homomorphism $d'_2 : \Gamma_{n}^{k}(D; \mathbb{Z}_2)$ $\to$ $\Gamma_{n}^{k+1}(D; \mathbb{Z}_2)$ as
\begin{equation}\label{def_dif_cable}
d'_2 (\textbf{s}) = \sum_{\textbf{t}~:~{\text{an~edge~added~to}}~\textbf{s} }\textbf{t}.  
\end{equation}
If $\textbf{s}$ is as shown in Fig. \ref{cable_ex_knot} (a), $d'_2({\textrm{(a)}})$ $=$ (b) $+$ (c) $+$ (d).  Similarly, $d'_{2}({\textrm{(b)}})$ $=$ (e), $d'_{2}({\textrm{(c)}})$ $=$ (e), $d'_{2}({\textrm{(d)}})$ $=$ $0$, and $d'_{2}({\textrm{(b)}} + {\textrm{(c)}} + {\textrm{(d)}})$ $=$ (e) $+$ (e) $+$ $0$ $=$ $0$, as in Fig. \ref{dif_cable}.   The map $d'_{2}$ was introduced by Khovanov \cite{coloredkhovanov} and satisfies the following.  
\begin{proposition}\label{graph_knot_pro}
The map $d'_{2}$ satisfies ${d'_{2}}^{2}$ $=$ $0$.  
\end{proposition}
\begin{proof}
Let us consider the $k$-pairing $\textbf{s}$, two $(k+1)$-pairings $\textbf{t}$, $\textbf{t'}$, and the $(k+2)$-pairing $\textbf{u}$ satisfying the following condition ($\ast$): if we remove one edge from $\textbf{u}$ (e.g., (e) of Fig. \ref{cable_ex_knot}), we have $\textbf{t}$ and $\textbf{t'}$ (e.g., (b) and (c) of Fig. \ref{cable_ex_knot}); if we remove the two edges corresponding to those added to $\textbf{s}$ to give $\textbf{t}$ and $\textbf{t'}$, we have $\textbf{s}$ (e.g., (a) of Fig. \ref{cable_ex_knot}).  When we take an arbitrary $\textbf{s}$, there exists a pairing $\textbf{u}$ satisfying ($\ast$).  When we have such a pair as the $k$-pairing $\textbf{s}$ and $(k+2)$-pairing $\textbf{u}$, there exist exactly two $(k+1)$-pairings, i.e., $\textbf{t}$ and $\textbf{t'}$, satisfying ($\ast$).  Thus, we have
\begin{equation}
\begin{split}
d'_{2}(\textbf{t}) &= \textbf{u}_1 + \textbf{u}_2 + \dots + \textbf{u}_l~{\text{s.t.}}~\textbf{u}_i \neq \textbf{u}_j~{\text{for arbitrary}}~i, j~(1 \le i < j \le l), \\
d'_{2}(\textbf{t'}) &= \textbf{u'}_1 + \textbf{u'}_2 + \dots + \textbf{u'}_m~{\text{s.t.}}~\textbf{u'}_i \neq \textbf{u'}_j~{\text{for arbitrary}}~i, j~(1 \le i < j \le m).\\
\end{split}
\end{equation}
By the condition ($\ast$), there exist $\textbf{u}_i$ and $\textbf{u'}_j$ such that $\textbf{u}_i$ $=$ $\textbf{u'}_j$ $=$ $\textbf{u}$.  Then, for an arbitrary $k$-pairing $\textbf{s}$, we have ${d'_2}^{2}(\textbf{s})$ $=$ $\sum_{\textbf{u}} 2 \textbf{u}$ $=$ $0$ (e.g., Fig. \ref{dif_cable}).  
\end{proof}
\begin{figure}
\begin{center}
\qquad\qquad\quad \includegraphics[width=4cm, bb=0 0 162.33 252.33]{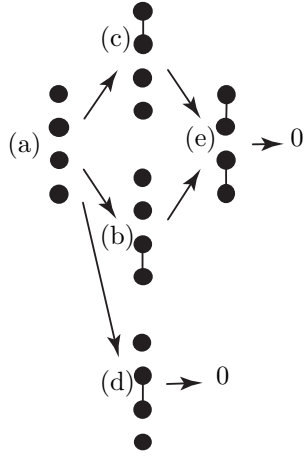}
\end{center}
\begin{picture}(0,0)
\put(160,175){(c)}
\put(125,135){(a)}
\put(160,100){(b)}
\put(160,45){(d)}
\put(192,137){(e)}
\put(204,47){$0$}
\put(232,137){$0$}
\end{picture}
\caption{Example of the map $d'_{2}$ on $k$-pairings of the cable $D^{4}$ of knot diagram $D$ for $k$ $=$ $0$, $1$, and $2$.  The $0$-map is usually omitted as the arrow from (d) or (e).}\label{dif_cable}
\end{figure}

\subsection{Extension to the case of coefficient $\mathbb{Z}$}\label{coeff_z}
The additional consideration of the sign of a $k$-pairing implies the $\mathbb{Z}$ case.  To do this, we consider replacing (\ref{def_dif_cable}) with another formula.  Recall the abelian group $\Gamma_n (D; \mathbb{Z}_2)$ of Sec. \ref{coeff_z_2} for an oriented knot diagram $D$.  We consider the abelian group $\Gamma^{k}_n (D)$ over the coefficient $\mathbb{Z}$ generating every $k$-pairing and set $\Gamma_n (D)$ $=$ $\oplus_{k = 0}^{\infty} \Gamma_{n}^{k}(D)$.  We define an operator $\Gamma_{n}^{k}(D)$ $\to$ $\Gamma_{n}^{k+1}(D)$ by the map between the $k$-pairing $\textbf{s}$ and the $(k+1)$-pairing $\textbf{t}$: 
\begin{equation}\label{def_dif_cable_a}
d'(\textbf{s}) = \sum_{\textbf{t}~:~\text{an edge added to}~\textbf{s}} (\textbf{s} : \textbf{t})~\textbf{t}
\end{equation}
where the incidence number $(\textbf{s} : \textbf{t})$ $=$ $(-1)^{t}$, and $t$ is the number of edges in $\textbf{t}$ on the above from the ``new'' edge in $\textbf{t}$ which is not in $\textbf{s}$ (Fig. \ref{dif_cable_rivited}). 
\begin{figure}[h!]
\begin{center}
\begin{picture}(0,0)
\put(30,44){$+$}
\put(30,87){$+$}
\put(23,127){$+$}
\put(63,44){$-$}
\put(67,134){$+$}
\end{picture}
\quad \includegraphics[width=3cm, bb=0 0 115.67 232]{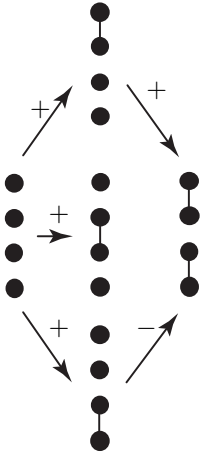}
\caption{Example of operator $d'$ on $\Gamma_{4}(D)$.  Sign $+$ ($-$) corresponds to $1$ ($-1$).  When we write the map $d'$ as in this figure, the arrows corresponding to $0$-maps are often omitted.  }\label{dif_cable_rivited}
\end{center}
\end{figure}
\begin{proposition}
\[d'^{2} = 0.  \]
\end{proposition}
\begin{proof}
For an arbitrary $k$-pairing $\textbf{s}$, it is sufficient to prove $d'(d'(\textbf{s}))$ $=$ $0$.  Let us consider the $(k+2)$-pairing $\textbf{u}$ that has exactly two edges more than $\textbf{s}$.  If we delete only one edge of the graph $\textbf{u}$, we have only two $(k+1)$-pairings $\textbf{t}$ and $\textbf{t'}$, each of which has exactly one edge more than $\textbf{s}$.  This is the same discussion as for the proof of Proposition \ref{graph_knot_pro}.  By the definition of the incidence number, $(\textbf{s} : \textbf{t}) (\textbf{t} : \textbf{u})$ $=$ $- (\textbf{s} : \textbf{t'}) (\textbf{t'} : \textbf{u})$ and so
\begin{equation}
\begin{split}
d'(d'(\textbf{s})) &= \sum_{\textbf{u}} \{ (\textbf{s} : \textbf{t}) (\textbf{t} : \textbf{u}) + (\textbf{s} : \textbf{t'}) (\textbf{t'} : \textbf{u})\} \textbf{u} \\
&= 0.
\end{split}
\end{equation}
\end{proof}

\subsection{Extension to the case of links}\label{extend_link}
We can now consider a natural extension of the discussions in Secs. \ref{coeff_z_2} and \ref{coeff_z} to the case of (unframed) links.  For fixed a tuple of nonnegative integers $\textbf{n}$ $=$ $(n_1, n_2, \dots, n_l)$, the colored Jones polynomial $J_{\textbf{n}}$ of a link diagram $D$ is written as
\begin{equation}
J_{\textbf{n}}(L) = \sum_{\textbf{k} = 0}^{\lfloor \textbf{n}/2 \rfloor} (-1)^{|{\textbf{k}}|} \begin{pmatrix} \textbf{n} - \textbf{k} \\ \textbf{k} \end{pmatrix} J(D^{\textbf{n} - 2 \textbf{k}})
\end{equation}
where $\textbf{k}$ $=$ $(k_1, k_2, \dots, k_l)$ for an arbitrary $k_i$ $\in$ $\mathbb{Z}_{\ge 0}$, $|{\textbf{k}}|$ $=$ $\sum_i k_i$, $\left( \begin{smallmatrix} \textbf{n} - \textbf{k} \\ \textbf{k} \end{smallmatrix} \right)$ $=$ $\prod_{i = 1}^{l} \left( \begin{smallmatrix} n_i - k_i \\ k_i \end{smallmatrix} \right)$, the sum $\sum_{\textbf{k} = \textbf{0}}^{\lfloor {\textbf{n}}/2 \rfloor}$ is the summation over all $0$ $\le$ $k_i$ $\le$ $\lfloor n_i / 2 \rfloor$ for all $i$, and $J(D)$ is the Jones polynomial of a link diagram $D$ of $L$ with $J(D^{0})$ $=$ $1$.  

\begin{figure}
\begin{center}
\begin{picture}(0,0)
\put(5,53){I}
\put(65,62){I\!I}
\put(180,58){I}
\put(232,61){I\!I}
\put(190,32){$1$}
\put(173,43){$2$}
\put(247,35){$3$}
\put(245,58){$1$}
\put(243,100){$2$}
\qbezier(261,90)(256,95)(251,100)
\put(310,-5){I}
\put(322,-5){I\!I}
\put(303,22){\small$1$}
\put(303,47){\small$2$}
\put(316,12){\small$1$}
\put(316,36){\small$2$}
\put(316,59){\small$3$}
\end{picture}
\includegraphics[width=12cm, bb=0 0 505.33 155.33]{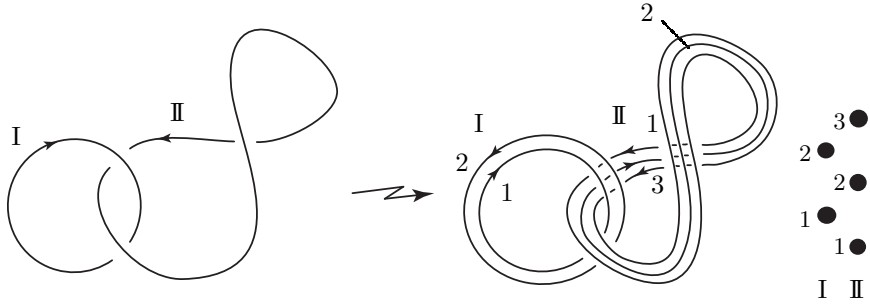}
\caption{The $(2, 3)$-cable $D^{(2, 3)}$ of an oriented link diagram $D$.  The numbers I and I\!I indicate those of components.  The numbers $1$, $2$, and $3$ indicate those of the strands in a component.  }\label{cable_ex_ori}
\end{center}
\end{figure}
\begin{figure}
\begin{center}
\begin{picture}(0,0)
\put(10,45){I}
\put(65,55){I\!I}
\put(165,55){I}
\put(217,63){I\!I}
\put(162,26){$1$}
\put(148,40){$2$}
\put(202,58){$3$}
\put(209,40){$1$}
\put(233,61){$2$}
\put(255,-5){I}
\put(267,-5){I\!I}
\put(248,20){\small$1$}
\put(248,42){\small$2$}
\put(261,12){\small$1$}
\put(261,29){\small$2$}
\put(261,49){\small$3$}
\qbezier(231,58)(226,53)(221,48)
\end{picture}
\quad \includegraphics[width=10cm, bb=0 0 400 113.5]{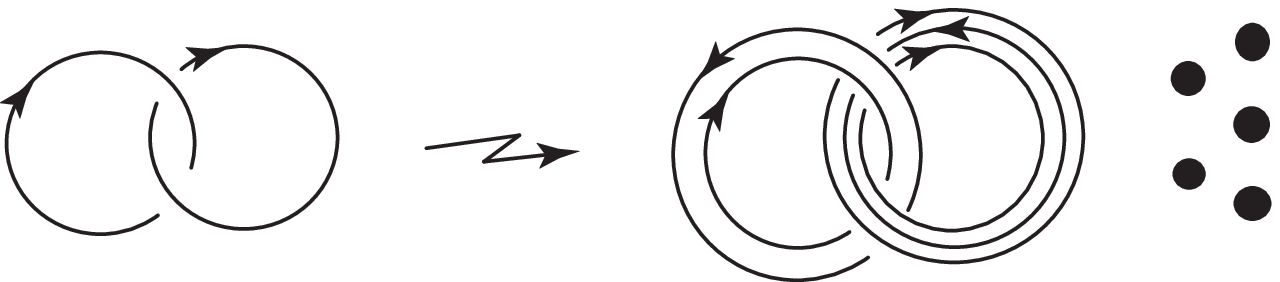}
\end{center}
\caption{Another example of the $(2, 3)$-cable of an oriented link and its associate graph $\Gamma_{(2, 3)}$.  Note that the graph is the same as that in Fig. \ref{cable_ex_ori}.}\label{cable_ex_2}
\end{figure}

Then, for each link diagram $D^{\textbf{n} - 2 \textbf{k}}$ (e.g., Figs. \ref{cable_ex_ori} and \ref{cable_ex_2}), we can consider the Khovanov homology groups of the diagram, and for the set of all diagrams of $D^{\textbf{n} - 2 \textbf{k}}$, we consider $\Gamma_{\textbf{n}}$ as follows.  First, we give the definition in the case of the coefficient $\mathbb{Z}_2$ in the same way as done in Sec. \ref{coeff_z_2}.  

Recall the definition of the graph $k$-pairing in Sec. \ref{coeff_z_2}.  The $k$-pairing has dots in only one vertical line, since a knot consists of one component. Extending the definition of a $k$-pairing to an $l$-component link, we consider dots in $l$ vertical lines.  The rule of adding edges is the same.  It is possible to add one edge connecting two neighborhood dots, but any single dots can connect with at most one other dot.  Let $\textbf{n}$ $=$ $(n_1, n_2, \dots, n_l)$ and $\textbf{k}$ $=$ $(k_1, k_2, \dots, k_l)$ for nonnegative integers $n_i$, $k_i$.  Then, for a given $\textbf{n}$, the $\textbf{n} - 2{\textbf{k}}$ cable $D^{\textbf{n} - 2 \textbf{k}}$ of the $l$-component link diagram $D$ determines the set of graphs having $l$ vertical lines, with $n_i$ dots and $k_i$ edges along the $i$-th vertical line.  The graphs of the set depending on $\textbf{n}$ and $\textbf{k}$ are denoted by $|{\textbf{k}}|$-{\itshape{pairings}}, simply {\itshape{pairings}} where $|{\textbf{k}}|$ $=$ $\sum_{i=1}^{l} k_i$.  The expression "$(k+l)$-pairings"  is permitted also for nonnegative integers $k$ and $l$.  

The abelian group over $\mathbb{Z}_2$ generated by $\textbf{k}$-pairings associated with the ${\textbf{n} - 2\textbf{k}}$ cable of a diagram $D$ of an oriented link $L$ is denoted by $\Gamma_{\textbf{n}}(D; \mathbb{Z}_2)$.  The subgroup of $\Gamma_{\textbf{n}}(D; \mathbb{Z}_2)$ generated by each graph with $k$ edges is denoted by $\Gamma_{\textbf{n}}^{k}(D; \mathbb{Z}_2)$.  Considering every $\textbf{k}$ $=$ $( k_1,$ $k_2, \dots, k_l )$ such that $|{\textbf{k}}|$ $=$ $k$, we have $\Gamma_{\textbf{n}}(D; \mathbb{Z}_2)$ $=$ $\oplus_{|{\textbf{k}}| = k} \Gamma_{\textbf{n}}^{k}(D; \mathbb{Z}_2)$.  As the definition of a $|{\textbf{k}}|$-pairing depends only on the integer tuple $\textbf{k}$ and the number of components of $D^{\textbf{n} - 2 \textbf{k}}$, $|{\textbf{k}}|$-pairings are link invariants.  For the pairings $\textbf{s}$ $\in$ $\Gamma_{\textbf{n}}^{k}(D; \mathbb{Z}_2)$ and $\textbf{t}$ $\in$ $\Gamma_{\textbf{n}}^{k + 1}(D; \mathbb{Z}_2)$, we define the homomorphism $d'_{2} : \Gamma_{\textbf{n}}^{k}(D; \mathbb{Z}_2)$ $\to$ $\Gamma_{\textbf{n}}^{k + 1}(D; \mathbb{Z}_2)$ by
\begin{equation}
d'_{2}(\textbf{s}) = \sum_{\textbf{t}~:~\text{an edge added to}~\textbf{s} } \textbf{t}.  
\end{equation}
\begin{ex}
Let us consider the graph in Fig. \ref{graph_link}, where $\textbf{s}$ $=$ $(a)$.  Then, $d'_{2}((a))$ $=$ $(b)$ $+$ $(c)$ $+$ $(d)$ and ${d'}_{2}^{2}((a))$ $=$ $d'_{2}((b) + (c) + (d))$ $=$ $(e)$ $+$ $(f)$ $+$ $(f)$ $+$ $(e)$ $=$ $0$ ({\rm{mod}} $2$).  
\end{ex}
\begin{figure}[h!]
\begin{center}
\includegraphics[width=5cm, bb=0 0 230.69 200]{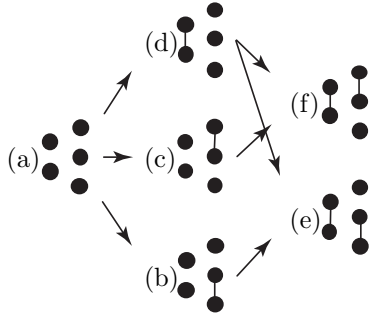}
\end{center}
\begin{picture}(0,0)
\put(88,77){(a)}
\put(140,32){(b)}
\put(140,77){(c)}
\put(140,121){(d)}
\put(195,54){(e)}
\put(195,98){(f)}
\end{picture}
\caption{(a): $0$-pairing.  (b), (c), (d): $1$-pairings.  (e), (f): $2$-pairings.  The arrows indicate maps: $d'_{2}({\rm{(a))}}$ $=$ (b) $+$ (c) $+$ (d) and  $d'_{2}({\rm{(d)}})$ $=$ (e) $+$ (f).  $d'_{2}({\rm{(c)}})$ $=$ (f).  $d'_{2}({\rm{(b)}})$ $=$ (e).  $0$-maps are omitted, e.g., $d'_{2}({\rm{(e)}})$ $=$ $0$.}\label{graph_link}
\end{figure}
In general, we have the following.  
\begin{proposition}\label{dif2}
\begin{equation}\label{graph_diff}
{d'_{2}}^{2} = 0.  
\end{equation}
\end{proposition}
\begin{proof}
The proof is similar to that of Proposition \ref{graph_knot_pro}.  Let us consider an arbitrary pairing $\textbf{s}$ with $k$ edges.  When we arbitrarily add exactly two edges to $\textbf{s}$, we have a pairing with $(k+2)$ edges denoted by $\textbf{u}$.  Below, we consider such a pair $(\textbf{s}, \textbf{u})$.  There exist exactly two pairings with $(k+1)$ edges between $\textbf{s}$ and $\textbf{u}$, meaning that when one of the two edges of $\textbf{u}$ is added to $\textbf{s}$, we get a pairing with $(k+1)$ edges, and there are exactly two possibilities for adding an edge.  These two pairings are denoted by $\textbf{t}$ and $\textbf{t'}$.  For any $|{\textbf{k}}|$-pairing $\textbf{s}$ and $(|{\textbf{k}}| + 2)$-pairing $\textbf{u}$ such that $\textbf{u}$ is formed by adding two edges to $\textbf{s}$, we have that $d'_{2}(d'_{2}(\textbf{s}))$ contains terms $\textbf{u}$ $=$ $\textbf{u}_i$ such that $d'_{2}(\textbf{t})$ $=$ $\sum_{m} \textbf{u}_m$ ($\textbf{u}_m$ $\neq$ $\textbf{u}$ if $m$ $\neq$ $i$) and $\textbf{u}$ $=$ $\textbf{u'}_j$ such that $d'_{2}(\textbf{t'})$ $=$ $\sum_{m} \textbf{u'}_m$ ($\textbf{u'}_m$ $\neq$ $\textbf{u}$ if $m$ $\neq$ $j$).  Therefore, the other terms, i.e., $\textbf{u}_m$ ($m$ $\neq$ $i$), $\textbf{u'}_m$ ($m$ $\neq$ $j$), do not reach $\textbf{u}$ by $d'_{2}$.  Then, the image $d'_{2}(d'_{2}(\textbf{s}))$ contains $2 \textbf{u}$ that is $0$ for each pair $(\textbf{s}, \textbf{u}$).  If we consider another pair $(\textbf{s}, \textbf{u'})$ consisting of the same $|{\textbf{k}}|$-pairing $\textbf{s}$ and another $(|\textbf{k}| + 2)$-pairing $\textbf{u'}$ formed by adding exactly two edges to $\textbf{s}$, the image $d'_{2}(d'_{2}(\textbf{s}))$ contains $2 \textbf{u'}$ by the same discussion.  Then, 
\[d'_{2}(d_{2}(\textbf{s})) = \sum_{\textbf{u}}~2 \textbf{u} = 0.\]
Linearly extending the above formula completes the proof.  
\end{proof}

Let us now consider extending the above discussion to the case of the coefficient $\mathbb{Z}$.  For a given $\textbf{n}$ $=$ $(n_1, n_2, \dots, n_l)$ and $\textbf{k}$ $=$ $(k_1, k_2, \dots, k_l)$, we consider the ${\textbf{n} - 2 \textbf{k}}$ cable of an $l$-component link diagram $D$ of an oriented link $L$ (Figs. \ref{cable_ex_ori} and \ref{cable_ex_2}) in a manner similar to Sec. \ref{extend_link}.  The group $\Gamma_{\textbf{n}}(D)$ over the coefficient $\mathbb{Z}$ is defined as the abelian group generated by $|\textbf{k}|$-pairings.  We denote by $\Gamma_{\textbf{n}}^{k}(D)$ the subgroup of $\Gamma_{\textbf{n}}(D)$ generating $|\textbf{k}|$-pairings such that $|{\textbf{k}}|$ $=$ $\sum_{i=1}^{l} k_i$ $=$ $k$, and then, $\Gamma_{\textbf{n}}(L)$ $=$ $\oplus_{k=0}^{\infty}\Gamma_{\textbf{n}}^{k}(L)$.

Next, we define an operator $\Gamma_{\textbf{n}}^{k}(D)$ $\to$ $\Gamma_{\textbf{n}}^{k+1}(D)$ by the homomorphism between $\Gamma_{\textbf{n}}^{k}(D)$ $\ni$ $\textbf{s}$ $\mapsto$ $\textbf{t}$ $\in$ $\Gamma_{\textbf{n}}^{k+1}(D)$: 
\begin{equation}\label{z_pairing_dif}
d'(\textbf{s}) = \sum_{{\textbf{t}}~:~{\text{an~edge~added~to}}~{\textbf{s}}} ({\textbf{s}}: {\textbf{t}})~{\textbf{t}}
\end{equation}
where the incidence number $(\textbf{s}, \textbf{t})$ $=$ $(-1)^{t}$ and $t$ is the number of edges in $\textbf{t}$ on the right or above from the ``new'' edges in $\textbf{t}$ which are not in $\textbf{s}$ (Fig. \ref{cable_graph}).

\begin{figure}
\begin{center}
\begin{picture}(0,0)
\put(55,37){$+$}
\put(55,54){$+$}
\put(55,72){$+$}
\put(120,20){$+$}
\put(115,32){$+$}
\put(120,48){$-$}
\put(121,87){$-$}
\put(10,48){(a)}
\put(77,85){(d)}
\put(77,47){(c)}
\put(77,8){(b)}
\put(142,75){(f)}
\put(142,20){(e)}
\end{picture}
\qquad \includegraphics[width=6cm, bb=0 0 286.92 200]{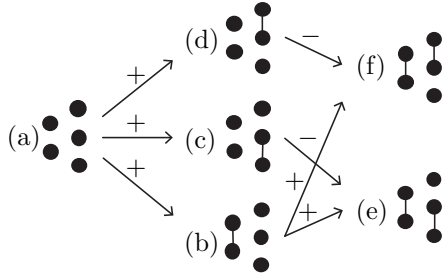}
\end{center}
\caption{Examples of the $k$-pairings of $\Gamma_{(2, 3)}^{k}(D)$ corresponding to the cable $D^{(2, 3)}$ of a link diagram $D$.  (a): $0$-pairing, (b), (c), (d): $1$-pairings, (e), (f): $2$-pairings.  In this case, the set $I_{k}$ of pairings such that $|{\textbf{k}}|$ $=$ $k_1$ $+$ $k_2$ $=$ $k$ is as follows: $I_0$ $=$ $\{$(a)$\}$, $I_1$ $=$ $\{$(b), (c), (d)$\}$.  $I_2$ $=$ $\{$(e), (f)$\}$.  }\label{cable_graph}
\end{figure}

\begin{ex}
We calculate $d'$ for the pairings shown in Fig. \ref{cable_graph}.  Then, $d'((a))$ $=$ $(b)$ $+$ $(c)$ $+$ $(d)$, and so, $d'(d'(a))$ $=$ $(e)$ $+$ $(f)$ $-$ $(e)$ $-$ $(f)$ $=$ $0$.  
\end{ex}

\begin{proposition}
\begin{equation}\label{dif3}
{d'}^{2} = 0.  
\end{equation}
\end{proposition}
\begin{proof}
Let us consider an arbitrary pairing $\textbf{u}$ of $\Gamma_{\textbf{n}}^{k+2}(L)$ (e.g., Fig. \ref{cable_graph} (f)).  If we select exactly two edges of $\textbf{u}$ to be deleted, the unique pairing $\textbf{s}$ (e.g., Fig. \ref{cable_graph} (a)) of $\Gamma_{\textbf{n}}^{k}(L)$ is determined.  For the pair $(\textbf{s}, \textbf{u})$, there exist exactly two pairings $\textbf{t}$ and $\textbf{t'}$ (e.g., Figs. \ref{cable_graph} (c) and (d), respectively), because we obtain $\textbf{u}$ when one of the two deleted edges is added to $\textbf{t}$ or $\textbf{t'}$.  Thus, we have that $d'(\textbf{s})$, which is the linear sum of pairings, contains the two terms $(\textbf{s} : \textbf{t}) \textbf{t}$ $+$ $(\textbf{s} : \textbf{t'}) \textbf{t'}$.  By the definition of the incidence number $(\textbf{s} : \textbf{t})$, $(\textbf{s} : \textbf{t}) (\textbf{t} : \textbf{u})$ $=$ $-$ $(\textbf{s} : \textbf{t'}) (\textbf{t'} : \textbf{u})$.  Then, $d'(d'(\textbf{s}))$ contains $d'((\textbf{s} : \textbf{t}) \textbf{t} + (\textbf{s} : \textbf{t'}) \textbf{t'})$, and we notice that if ${d'}^{2}(\textbf{s})$ contains $\textbf{u}$, $\textbf{u}$ must come from either $\textbf{t}$ or $\textbf{t'}$.  Hence, for an arbitrary $\textbf{u}$ contained $d'(d'(\textbf{s}))$ has the coefficient $(\textbf{s} : \textbf{t}) (\textbf{t} : \textbf{u})$ $+$ $(\textbf{s} : \textbf{t'}) (\textbf{t'} : \textbf{u})$ and we have
\[{d'}^{2}(\textbf{s}) = \sum_{\textbf{u}} \{(\textbf{s} : \textbf{t})(\textbf{t} : \textbf{u}) + (\textbf{s} : \textbf{t'})(\textbf{t'} : \textbf{u})\} \textbf{u} = 0.\] 
\end{proof}

\section{Operators of Khovanov complex for cable diagrams}\label{def_diff_cable}
\subsection{Homomorphisms $\rho$, $f$.}
As a preliminary, we recall or define three homomorphisms on the Khovanov complex $\{C^{i, j}(D), d\}$ of a oriented link diagrams.  
\begin{definition}
Let $S_{\rho}$ be the enhanced states defined as (a-$i$) + (b-$i$ : $p'$ $=$ $m(p, q)$ and $q'$ $=$ $1$) for $i$ $=$ $5$, $7$, and (a-$i$) + (b-$i$ : $\sum p' \otimes r'$ $=$ $\Delta(p)$ and $q'$ $=$ $1$) for $i$ $=$ $6, 8$.  The Khovanov complex generated by such $\{S\}$ is denoted by $C^{i, j}(S_{\rho})$.  This is actually a subcomplex, since $d(C^{i, j}(S_{\rho}))$ $\subset$ $C^{i, j}(S_{\rho})$.  There exists an isomorphism between $C^{i, j}(S_{\rho})$ and $C^{i, j}(D_{\infty 0})$, where $D_{\infty 0}$ is the link diagram given by neglecting two markers in (a-5)--(a-6) of Fig. \ref{type12}, i.e., smoothing as $D_{\infty}$ and $D_{0}$ in Fig. \ref{kauffman_b}.  In the $\mathbb{Z}_2$-case, we can denote these by $C^{i, j}({S_{\rho}; \mathbb{Z}_2})$ and $C^{i, j}(D_{\infty 0}; \mathbb{Z}_2)$.  

The homomorphism $\rho$ : $C^{i, j}(D)$ $\to$ $C^{i, j}(S_{\rho})$ is defined by 
\begin{equation}\label{rho2Eq}
\begin{split}
{\text{For}}~i=5~{\text{or}}~&7, \\
{\text{(a-$i$)}} = p \otimes q &\mapsto p \otimes q (={\text{(a-$i$)}}) + m(p \otimes q) \otimes 1 (= {\text{(b-$i$)}}),\\
{\text{(b-$i$)}} = p' \otimes x &\mapsto - \Delta(p') (={\text{(a-$i$)}}) - m(\Delta(p')) \otimes 1 (= {\text{(b-$i$)}}),\\
{\text{For}}~i=6~{\text{or}}~&8,\\
{\text{(a-$i$)}} = p &\mapsto p (= {\text{(a-$i$)}}) + \Delta(p) \otimes 1 (= {\text{(b-$i$)}}),\\
{\text{(b-$i$)}} = p' \otimes x \otimes r' &\mapsto - m(p' \otimes r') (= {\text{(a-$i$)}}) - \Delta(m(p' \otimes r')) \otimes 1 (= {\text{(b-$i$)}}),\\
{\text{otherwise}} &\mapsto 0.\\
\end{split}
\end{equation}
In the $\mathbb{Z}$-case, to fix the signs of every term of (\ref{rho2Eq}), the last negative marker of each case is always appeared in figures Fig. \ref{type12} of (a-$i$) and (b-$i$) for $i$ $=$ $5$, $6$, $7$, and $8$.  
\end{definition}
Here, the homomorphism $\rho$ is the same map as that used in the retraction for proving the invariance of the second Reidemeister move \cite[Page 132, Formula (2.2)]{ito3}.  The homomorphism $\rho$ satisfies the following property.  
\begin{proposition}\label{prop_rho2}
\begin{equation}\label{chain_rho}
\rho \circ d = d \circ \rho
\end{equation}
In particular, 
\begin{equation}\label{map_rho_d} 
\rho \circ d(S) = 0
\end{equation}
where $S$ is an enhanced state that has a part appearing in the left or right of Fig. \ref{pre_type12}.  
\end{proposition}
\begin{proof}
These formulae can be proved by a straightforward calculation as follows.  

The local diagram $D$ that we focus on in a link diagram and its enhanced states $S_{-+}(p,q)$, $S_{+-}(p,q)$, $S_{+-,1}(p,q)$, $S_{++}(p,q)$, and $S_{--}(p,q)$ are defined by Fig. \ref{secondStates} where the symbols used follow Jacobsson \cite[Sec. 3.3.3]{jacobsson}.  Indices $\pm$ (resp. $1$) represent the signs of markers (resp. label $1$ for the center of the circle).  
\begin{figure}[htbp]
\includegraphics[width=12cm]{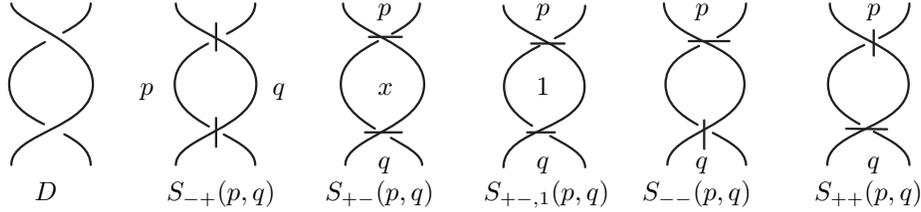}

\begin{picture}(0,0)
\put(10,0){$D$}
\put(60,0){$S_{-+}(p,q)$}
\put(50,40){$p$}
\put(100,40){$q$}
\put(120,0){$S_{+-}(p,q)$}
\put(140,70){$p$}
\put(140,40){$x$}
\put(140,12){$q$}
\put(180,0){$S_{+-,1}(p,q)$}
\put(200,70){$p$}
\put(200,40){$1$}
\put(200,12){$q$}
\put(240,0){$S_{--}(p,q)$}
\put(260,70){$p$}
\put(260,12){$q$}
\put(305,0){$S_{++}(p,q)$}
\put(325,70){$p$}
\put(325,12){$q$}
\end{picture}
\caption{Enhanced states generating $C^{i, j}(D)$.  Each of $p$ and $q$ is $x$ or $1$.  Every enhanced state appearing in this figure has a common ordering of negative markers followed by a negative marker appearing in this figure.}\label{secondStates}
\end{figure}

We also prepare the following notation as in Fig. \ref{notationState} by following Jacobsson \cite[Page 1216, Fig. 3]{jacobsson}.  The calculus of Fig. \ref{differential2} shows that one circle connects with another or one circle splits into two circles.  We describe the calculus using an abstract symbol: we use $(p:q)$ and $(q:p)$ for $p$ and $q$, where each of $p$ and $q$ is an $x$ or a $1$.  For example, the case $\Delta(1)$ $=$ $1 \otimes x$ $+$ $x \otimes 1$ in Fig. \ref{differential2} corresponds to $p$ $=$ $q$ $=$ $1$ and $(p:q, q:p)$ $=$ $(1, x)$ $+$ $(x, 1)$.  
\begin{figure}[htbp]
\begin{center}
\begin{picture}(0,0)
\put(17,35){$p$}
\put(50,35){$q$}
\put(158,55){$p:q$}
\put(158,10){$q:p$}
\end{picture}
\includegraphics[width=7cm]{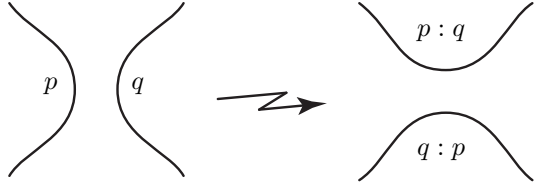}
\end{center}
\caption{Abstract symbols $p:q$ and $q:p$.}\label{notationState}
\end{figure}

The definition (\ref{rho2Eq}) of $\rho$ is presented as
\begin{equation}\label{def_rho2}
\begin{split}
S_{-+}(p, q) &\mapsto S_{-+}(p, q) + S_{+-, 1}(p:q, q:p), \\
S_{+-}(p, q) &\mapsto - S_{-+}(p:q, q:p) -  S_{+-,1}((p:q):(q:p), (q:p):(p:q)),\\
{\rm{othewise}} &\mapsto 0.
\end{split}
\end{equation}

Let $S_{*}^{t}(p, q) \in C^{i, j}(D)$ be an enhanced state such that $t$ is the last negative marker changed and $S_{*}^{t}(p, q) \in C^{i-1, j}(D)$ is available by producing  the new negative marker $t$ from an enhanced state $S_{*}(p, q)$.  The sum $\sum_t S_{*}^{t}(p, q)$ denotes the large sum where the index $t$ runs over the new negative markers.  

\begin{enumerate}
\item $d \rho (S_{-+}(p, q))$ $=$ $d(S_{-+}(p, q) + S_{+-,1}(p:q, q:p))$ $=$ $\sum_t (S_{-+}^{t}(p, q) + S_{+-,1}^{t}(p:q, q:p))$ $=$ $\rho (\sum_t S_{-+}^{t}(p, q))$ $=$ $\rho d (S_{-+}(p, q))$,
\item $d \rho (S_{+-}(p, q))$ $=$ $- d(S_{-+}(p:q, q:p) + S_{+-,1}((p:q):(q:p), (q:p):(p:q)))$ $=$ $-\sum_t (S_{-+}^{t}(p:q, q:p) + S_{+-,1}^{t}((p:q):(q:p), (q:p):(p:q)))$ $=$ $-\rho(\sum_t S_{-+}^{t}(p:q,q:p))$ $=$ $\rho(\sum_t S_{+-}^{t}(p,q))$ $=$ $\rho d (S_{+-}(p, q))$,
\item $d \rho(S_{++}(p, q))$ $=$ $0$ $=$ $\rho(S_{-+}(p:q, q:p) + S_{+-}(p, q))$ $=$ $\rho d(S_{++}(p, q))$, 
\item $d \rho(S_{+-,1}(p,q))$ $=$ $0$ $=$ $\rho d(S_{+-,1}(p,q))$.  
\end{enumerate}
This proves (\ref{chain_rho}).  In particular, the latter two formulae of $S_{++}(p, q)$ and $S_{--}(p, q)$ imply (\ref{map_rho_d}).  
\end{proof}
Here, we comment that $C^{i, j}(S_{\rho})$ is generated by $\{ S_{-+}(p, q)$ $+$ $S_{+-, 1}(p:q, q:p) \}$ using the convention in the proof of Prop. \ref{prop_rho2}.

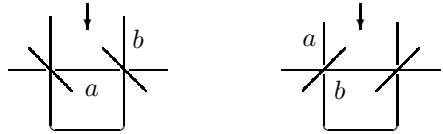
\begin{figure}[htbp]
\begin{center}
\begin{minipage}{60pt}
\begin{picture}(60,80)
\put(0,40){\line(1,0){14}}
\put(18,40){\line(1,0){24}}
\put(46,40){\line(1,0){14}}
\put(44,20){\line(0,1){40}}
\put(16,20){\line(0,1){40}}
\put(30,20){\oval(28,6)[b]}
{\linethickness{0.4pt}
\put(30,65){\vector(0,-1){10}}}
\qbezier(36,48)(44,40)(52,32)
\qbezier(8,48)(16,40)(24,32)
\put(29,30){$a$}
\put(47,48){$b$}
\end{picture}
\end{minipage}
\qquad\qquad
\begin{minipage}{60pt}
\begin{picture}(60,80)
\put(0,40){\line(1,0){60}}
\put(16,42){\line(0,1){16}}
\put(44,42){\line(0,1){16}}
\put(44,20){\line(0,1){18}}
\put(16,20){\line(0,1){18}}
\put(30,20){\oval(28,6)[b]}
{\linethickness{0.4pt}
\put(30,65){\vector(0,-1){10}}
}
\qbezier(36,32)(44,40)(52,48)
\qbezier(8,32)(16,40)(24,48)
\put(20,29){$b$}
\put(8,49){$a$}
\end{picture}
\end{minipage}
\end{center}
\caption{Enhanced states $S$, in order from the left, $S$ such that $d(S)$ may possibly contain (a-5), (a-6), (b-5), or (b-6) and $S$ such that $d(S)$ may possibly contain (a-7), (b-7), (a-8), or (b-8).}\label{pre_type12}
\end{figure}
From the next section, we will the homomorphism such as $\oplus_{\textbf{s} \in I_k} \rho$ on $C^{k, i, j}(D)$ for the set $I_k$ consisting of $k$-pairings.  We can denote the homomorphism $\oplus_m \rho$ by $\rho$ if there would be no confusion.

The homomorphism $f$ : $C^{i, j}(D)$ $\to$ $C^{i, j}(D)$ is defined by
\begin{equation}
\begin{split}
{\text{(b-$i$)}} &\mapsto {\text{corresponding Type $1_j$ to (b-$i$) as Type $2_j$}}~{\text{for each}}~i \in \{1, 2, 3, 4\},\\
S &\mapsto S~{\text{otherwise}}.  
\end{split}
\end{equation}

\subsection{Case of knots and the coefficient $\mathbb{Z}_2$}\label{cable_z_2}
In this section, we define the coboundary operator between tri-graded complexes.  We consider the simplest case: knots and the coefficient $\mathbb{Z}_2$ fixing an integer $n$ of the colored Jones polynomial $J_{n}$.  Let us recall the map $d'_{2}$ from a $k$-pairing $\textbf{s}$ to the summation of $(k+1)$-pairings $\sum \textbf{t}$ of $\Gamma_{n}^{k+1}(D; \mathbb{Z}_2)$.  

For an enhanced state $S$ and $k$-pairing $\textbf{s}$, we consider the tensor product $S \otimes \textbf{s}$ $\in$ $C^{i, j}(D^{n - 2k}; \mathbb{Z}_2) \otimes \Gamma_{n}^{k}(D; \mathbb{Z}_2)$.  In this section, an enhanced state of Type $1_j$ or Type $2_j$ is denoted by $\widetilde{S}$.  Set $C_{n}^{k, i, j}(D; \mathbb{Z}_2)$ $=$ $C^{i, j}(D^{n - 2k}; \mathbb{Z}_2) \otimes \Gamma_{n}^{k}(D; \mathbb{Z}_2)$.  

For $\widetilde{S}$ of Type $1_j$, if we delete all contracted circles of $\widetilde{S}$ of $C^{i, j}(D^{n - 2k}; \mathbb{Z}_2)$, we have an enhanced state $S$ of $C^{i, j}(D^{n - 2(k+1)}; \mathbb{Z}_2)$ (cf. Fig. \ref{cable-fig3}).  This deletion of contracted circles implies a homomorphism defined by $\widetilde{S}$ $\mapsto$ $S$ for $\widetilde{S}$: Type $1_j$ and $\widetilde{S}$ $\mapsto$ $0$ otherwise, which is denoted by $d'_1$.  We define the operator 
\[d'^{k, i, j}_{2} : C_{n}^{k, i, j}(D; \mathbb{Z}_2) \to C_{n}^{k+1, i, j}(D; \mathbb{Z}_2)\] by
\begin{equation}\label{z_2_diff_cable}
\begin{split}
\widetilde{S} \otimes \textbf{s} &\mapsto \rho d'_{1} f (\widetilde{S}) \otimes d'_{2}(\textbf{s})\quad{\text{if}}~\widetilde{S}{\text{ : Type}}~1_j~{\text{or (b-1)--(b-4) of Type}}~2_j, \\
\widetilde{T} \otimes \textbf{s} &\mapsto \rho d'_1 \rho(\widetilde{T}) \otimes d'_{2}(\textbf{s})\quad{\text{if}}~\widetilde{T}{\text{ : (b-5)--(b-8) of Type}}~2_j.\\
{\text{otherwise}} &\mapsto 0  
\end{split}
\end{equation}
where $d'_{2}(\textbf{s})$ is the map defined by (\ref{def_dif_cable}).  Here, we used the maps $C^{i, j}(S_{\rho}; \mathbb{Z}_2)$ $\simeq$ $C^{i, j}(D_{\infty 0}; \mathbb{Z}_2)$ $\hookrightarrow$ $C^{k+1, i, j}(D; \mathbb{Z}_2)$.  
\begin{proposition}\label{prop_diff}
$d'^{k+1, i, j}_{2} \circ d'^{k, i, j}_{2} = 0.  $
\end{proposition}
\begin{proof}
Let us show that $d'^{k+1, i, j}_2 (d'^{k, i, j}_{2} (S \otimes {\textbf{s}}))$ $=$ $0$.  If $S$ is neither Type $1_j$ nor Type $2_j$, $d'^{k+1, i, j}_2 (d'^{k , i, j}_2 (S \otimes {\textbf{s}}))$ $=$ $d'^{k+1, i, j}_2 (0)$ $=$ $0$.  If $S$ is Type $1_j$ or Type $2_j$, we have
\begin{equation}\label{eq1}
{d'_2}^{k+1, i, j} (d'^{k, i, j}_2 (S)) = {d'_2}^{k+1, i, j} (g_1 (S) \otimes d'_{2}(\textbf{s}))
\end{equation}
where $g_1$ $=$ $\rho d_1' f$ or $\rho d_1' \rho$.  By the definition of $g_1$, $g_1 (S)$ $=$ $m T$, where $m$ is a nonnegative integer and $T$ is an enhanced state of $C^{i, j}(D_{\infty 0}; \mathbb{Z}_2)$ $\simeq$ $C^{i, j}(S_{\rho}; \mathbb{Z}_2)$.  If $T$ is neither Type $1_j$ nor Type $2_j$ or $m$ $=$ $0$, the right-hand side of (\ref{eq1}) is $0$.  If $m$ $\neq$ $0$ and $T$ is either Type $1_j$ or Type $2_j$, the right-hand side of (\ref{eq1}) is $g_2 g_1 (S) \otimes {d'_{2}}^{2}(\textbf{s})$, where $g_2$ is $\rho d_1' f$ or $\rho d_1' \rho$.  By Proposition \ref{graph_knot_pro}, we have ${d'_{2}}^{2}$ $=$ $0$, and then, 
\[{d'_2}^{k+1, i, j} (d'_1 g_1 (S) \otimes d'_2(\textbf{s})) = g_2 g_1 (S)  \otimes {d'_{2}}^{2}(\textbf{s}) = 0.  \]
\end{proof}
\begin{figure}
\begin{picture}(0,0)
\put(-5,160){$S_{1}$}
\put(0,285){$\widetilde{S_1}$}
\put(28,285){\footnotesize$1$}
\put(48,275){\footnotesize$x$}
\put(73.5,262){\footnotesize$x$}
\put(85,285){\footnotesize$1$}
\put(190,160){$S_{2}$}
\put(218,280){\footnotesize$1$}
\put(238,272){\footnotesize$1$}
\put(255,259){\footnotesize$x$}
\put(310,283.5){\footnotesize$x$}
\put(140,270){$\otimes$}
\put(335,270){$\otimes$}
\put(190,285){$\widetilde{S_2}$}
\put(33,161){\footnotesize$x$}
\put(63.5,183){\footnotesize$1$}
\put(263,161){\footnotesize$1$}
\put(278,161){\footnotesize$x$}
\put(33,47){\footnotesize$x$}
\put(63,65){\footnotesize$1$}
\put(263,55){\footnotesize$1$}
\put(280,53){\footnotesize$x$}
\put(140,160){$\otimes$}
\put(335,160){$\otimes$}
\put(-5,50){$S_{1}$}
\put(140,50){$\otimes$}
\put(190,50){$S_{2}$}
\put(335,50){$\otimes$}
\end{picture}
\includegraphics[width=13cm,bb=0 0 948 1104.91]{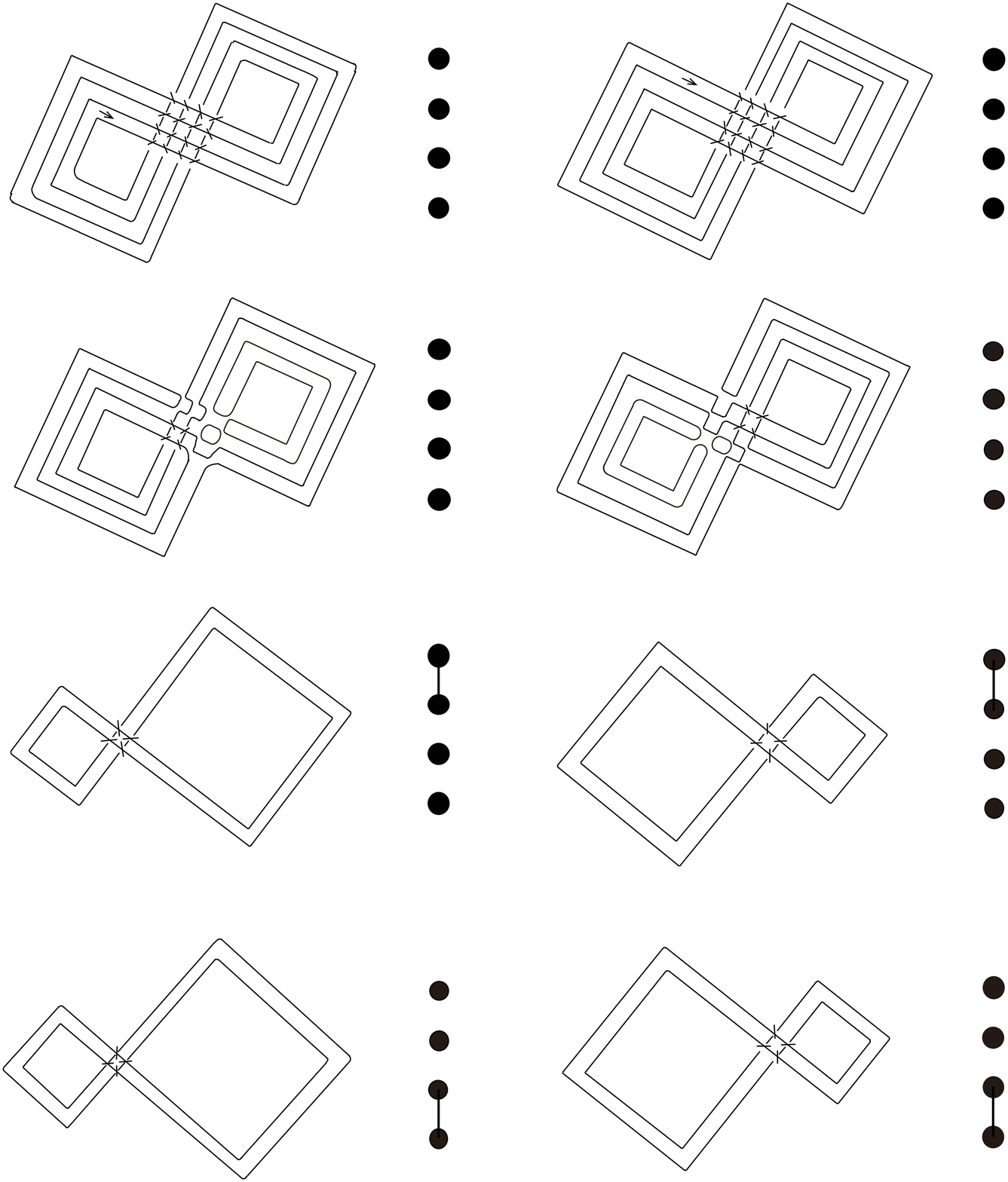}
\caption{Examples of Type $1$ are on the first line.  On the second line, $\widetilde{S_1} \otimes \textbf{s}$ and $\widetilde{S_2} \otimes \textbf{s}$ are enhanced states of Type $1_j$ in $C^{0, 0, 0}(D)$.  On the third line, $S_1 \otimes \textbf{t}$ and $S_2 \otimes \textbf{t}$ in $C^{1, 0, 0}(D)$.  On the fourth line, $S_1 \otimes \textbf{t'}$ and $S_2 \otimes \textbf{t'}$ in $C^{1, 0, 0}(D)$.}\label{shaded}
\end{figure}

\subsection{Case of links and the coefficient $\mathbb{Z}$}\label{cable_z}
We now extend the discussion of Sec. \ref{cable_z_2} to the case of links and the coefficient $\mathbb{Z}$ fixing a tuple of nonnegative integers $\textbf{n}$ of the colored Jones polynomial $J_{\textbf{n}}$.  To extend our argument to $\mathbb{Z}$, we must fix the order of negative markers of Type $1_j$, as we hope to define such a map as Formula (\ref{z_2_diff_cable}).  Markers are placed on a cable of a link diagram, as in Fig. \ref{strand}, where they depend on the directions of contracted strands.  
\begin{figure}[h!]
\begin{center}
\begin{minipage}{60pt}
\begin{picture}(0,80)
\put(20,0){(a)}
{\qbezier(16,54)(20,50)(24,46)}
\qbezier(36,46)(40,50)(44,54)
\put(0,50){\line(1,0){60}}
\put(20,20){\line(0,1){25}}
\put(40,20){\line(0,1){25}}
\put(20,55){\line(0,1){25}}
\put(40,55){\line(0,1){25}}
\put(30,75){\vector(0,-1){15}}
\end{picture}
\end{minipage}
\quad
\begin{minipage}{60pt}
\begin{picture}(0,80)
\qbezier(16,54)(20,50)(24,46)
\qbezier(36,46)(40,50)(44,54)
\put(23,0){(b)}
\put(0,50){\line(1,0){15}}
\put(25,50){\line(1,0){10}}
\put(45,50){\line(1,0){15}}
\put(20,20){\line(0,1){60}}
\put(40,20){\line(0,1){60}}
\put(30,75){\vector(0,-1){15}}
\end{picture}
\end{minipage}
\quad
\begin{minipage}{60pt}
\begin{picture}(0,80)
\qbezier(17,64)(21,60)(25,56)
\qbezier(17,36)(21,40)(25,44)
\qbezier(36,56)(40,60)(44,64)
\qbezier(36,44)(40,40)(44,36)
\put(20,0){(c)}
\put(0,40){\line(1,0){60}}
\put(0,60){\line(1,0){60}}
\put(20,20){\line(0,1){15}}
\put(40,20){\line(0,1){15}}
\put(20,45){\line(0,1){10}}
\put(40,45){\line(0,1){10}}
\put(20,65){\line(0,1){15}}
\put(40,65){\line(0,1){15}}
\put(30,80){\vector(0,-1){10}}
\end{picture}
\end{minipage}
\quad
\begin{minipage}{60pt}
\begin{picture}(0,80)
\qbezier(17,64)(21,60)(25,56)
\qbezier(17,36)(21,40)(25,44)
\qbezier(36,56)(40,60)(44,64)
\qbezier(36,44)(40,40)(44,36)
\put(23,0){(d)}
\put(20,20){\line(0,1){60}}
\put(40,20){\line(0,1){60}}
\put(0,40){\line(1,0){15}}
\put(0,60){\line(1,0){15}}
\put(25,40){\line(1,0){10}}
\put(25,60){\line(1,0){10}}
\put(45,40){\line(1,0){15}}
\put(45,60){\line(1,0){15}}
\put(30,80){\vector(0,-1){10}}
\end{picture}
\end{minipage}
\end{center}
\caption{(a), (b): Two crossings generated by two contracted strands and one non-contracted strand.  (c), (d): Four crossings generated by contracted strands only.}\label{strand}
\end{figure}
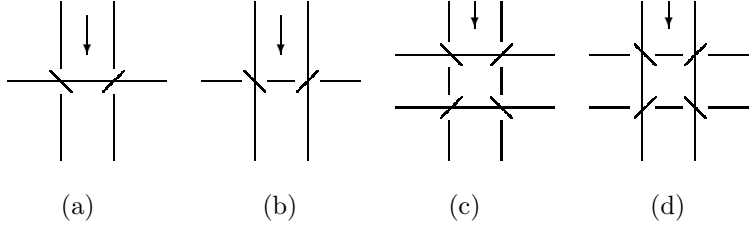  
\begin{definition}(The order of negative markers of Type $1_j$)\label{alternative_marker}
Embed a cable of a link diagram in $\mathbb{R}^{2}$ such that there is only one maximum point along the vertical axis (deform the diagram if necessary).  Let this maximum point be the base point and let the direction of the contracted strands correspond to the orientation of the strand of the lower dot.  Let $y$ be the word corresponding to alternate negative markers along the direction of the contracted strands, starting from the base point.  Let $x$ be an arbitrary word consisting of the other negative markers.  We permit only the word $xy$ to show the order and orientation of negative markers for {\it{the orientation of negative markers of Type $1_j$}}.  
\end{definition}
Note that either the right or left crossing has a negative marker, as in Figs. \ref{strand} (a)--(d), when we proceed along contracted strands and encounter another strand.  
\begin{remark}
By the above definition, the homology groups $H^{k, i, j}((D))$ of Def. \ref{tri_homology}, which are equivalent to homology groups $H^{k}(H^{i}(C_{\textbf{n}}^{*, *, j}(D)))$ of Theorem \ref{bicomplex_theorem}, do not depend on the choice of the base point for the following reason.  The fixing orientation depends on the fixing signs of enhanced states for Type $1_j$; then, the difference is $-d'$ or $d'$ on every $C_{\textbf{n}}^{k, i, j}(\cdot)$, and this sign does not depend on $(k, i, j)$ because the number of contracted strands is always even.  
\end{remark}
Here, we define {\it{$\oplus$-terms}} and {\it{$\ominus$-terms}}.  
\begin{definition}\label{plus_term}
Let us consider an enhanced state $S$ of Type $1_j$ and Type $2_j$.  For the $k$th panel corresponding to either (a-$i$) or (b-$j$) ($1 \le i, j \le 4$), we consider its sign $\epsilon_{k}$ $=$ $1$ ($-1$) if the $\oplus$ ($\ominus$) is marked in Tables \ref{pre_type1_table1}--\ref{pre_type1_table16}.  The enhanced state $S$ is called the {\it{$\oplus$-term}} ({\it{$\ominus$-term}}) if the product $\prod_{k} \epsilon_k$ $=$ $1$ ($=$ $-1$), where the product is taken over all parts (a-$i$) or (b-$j$) ($1 \le i, j \le 4$) of the contracted strands.  
\end{definition}

Assume that $D$ is an oriented $l$-component link diagram.  Let $k_i$ be the number of $i$-th edges in the vertical line of the pairing $\textbf{s}$, $\textbf{k}$ $=$ $(k_1, k_2, \dots, k_l)$, $|{\textbf{k}}|$ $=$ $\sum_{i=1}^{l} k_i$, and let $I_k$ be the set of $k$-pairings, where ``$k$-pairings'' is defined as pairings with $|\textbf{k}|$ $=$ k, see Sec. \ref{extend_link}.  Set $C^{k, i, j}_{\textbf{n}}(D)$ $=$ $\oplus_{\textbf{s} \in I_k} C^{i, j}(D^{\textbf{n} - 2 \textbf{k}}) \otimes \Gamma_{\textbf{n}}^{k}(D)$ and recall the map $d' : \Gamma_{\textbf{n}}^{k}(D)$ $\to$ $\Gamma_{\textbf{n}}^{k+1}(D)$ defined by (\ref{z_pairing_dif}) in Sec. \ref{extend_link}.  Let $S$ be an enhanced state formed by deleting contracted circles from $\widetilde{S}$.  This deletion implies the homomorphism defined by $\widetilde{S}$ $\mapsto$ $S$ for $\widetilde{S}$ : Type $1_j$ and $\widetilde{S}$ $\mapsto$ $0$ otherwise, which is denoted by $d'_1$.  We define the operator
\[d'^{k, i, j} : C^{k, i, j}_{\textbf{n}}(D) \to C^{k+1, i, j}_{\textbf{n}}(D)\]
as
\begin{equation}
\begin{split}
\widetilde{S} \otimes \textbf{s} &\mapsto \epsilon \rho {{d'}_1} f(\widetilde{S}) \otimes d'(\textbf{s})\quad {\text{if}}~\widetilde{S} : \oplus{\text{-term or }}\ominus{\text{-term}}, \\
\widetilde{T} \otimes \textbf{s} &\mapsto \rho {{d'}_1} \rho(\widetilde{T}) \otimes d'(\textbf{s})\quad {\text{if}}~\widetilde{T} : {\text{(b-5)--(b-8) of Type }}2_j~{\text{and}}, \\
{\text{otherwise}} &\mapsto 0  
\end{split}
\end{equation} 
where $\epsilon$ $=$ $1$ ($-1$) if $\widetilde{S}$ is $\oplus$-term ($\ominus$-term).  Here, we used the following maps $C^{i, j}(S_{\rho})$ $\simeq$ $C^{i, j}(D_{\infty 0})$ $\hookrightarrow$ $C^{|{\textbf{k}}|+1, i, j}$.  
\begin{proposition}
$d'^{k+1, i, j} \circ d'^{k, i, j} = 0.$
\end{proposition}
\begin{proof}
The discussion is similar to the proof of Prop. \ref{prop_diff}.  Let us show that ${d'}^{k+1, i, j}({d'}^{k, i, j}(S \otimes \textbf{s}))$ $=$ $0$.  If $S$ is neither Type $1_j$ nor Type $2_j$, ${d'}^{k+1, i, j}({d'}^{k, i, j}$ $(S \otimes \textbf{s}))$ $=$ ${d'}^{k+1, i, j}(0)$ $=$ $0$.  If $S$ is Type $1_j$ or Type $2_j$, 
\begin{equation}\label{proof_eq2}
d'^{k+1, i, j} (d'^{k, i, j}(S \otimes \textbf{s})) = d'^{k+1, i, j}(\eta g_1 (S) \otimes d'(\textbf{s}))  
\end{equation}
where $g_1$ $=$ $\rho d_1' f$ or $\rho d_1' \rho$ and $\eta$ $=$ $-1$ or $1$.  The $g_1 (S)$ is represented as $mT$, where $m$ is a nonnegative integer and $T$ is an enhanced state of $C^{i, j}(D_{\infty 0})$ $\simeq$ $C^{i, j}(S_{\rho})$.  If $T$ is neither Type $1_j$ nor Type $2_j$, the right-hand side of (\ref{proof_eq2}) is $0$.  If $m$ $\neq$ $0$ and $T$ is either Type $1_j$ or Type $2_j$, the right-hand side of (\ref{proof_eq2}) is $\zeta \eta g_2 g_1(S) \otimes d'^{2}(\textbf{s})$, where $g_2$ $=$ $\rho d_1' f$ or $\rho d_1' \rho$ and $\zeta$ $=$ $-1$ or $1$.  By Proposition \ref{dif3}, $d'^{2}$ $=$ $0$.  Then, 
\[d'^{k+1, i, j}(\eta g_1 (S) \otimes d'(\textbf{s})) = \zeta \eta g_2 g_1(S) \otimes d'^{2}(\textbf{s}) = 0. \]
\end{proof}

\section{Proof of Theorem \ref{bicomplex_theorem}}\label{proof_section}
Since it is easy for readers to reduce the case of coefficient $\mathbb{Z}_2$, we next provide the proof of Theorem \ref{bicomplex_theorem} in the case of the coefficient $\mathbb{Z}$.  

Assume that $D$ is an $l$-component link diagram.  We denote the $\mathbb{Z}$-module generated by a single element, i.e., pairing $\textbf{s}$, as $\langle \textbf{s} \rangle$.  By definition, there exists an integer $k$ such that $\langle \textbf{s} \rangle$ becomes a subgroup of $\Gamma_{\textbf{n}}^{k}(D)$.  Let us recall the differential $d : C^{i, j}(D^{\textbf{n} - 2 \textbf{k}})$ $\to$ $C^{i+1, j}(D^{\textbf{n} - 2 \textbf{k}})$ of the Khovanov homology defined in Sec. \ref{preliminary}.  For an arbitrary pairing $\textbf{s}$, we consider the homomorphism $C^{i, j}(D^{\textbf{n} - 2 \textbf{k}}) \otimes \langle \textbf{s} \rangle$ $\to$ $C^{i+1, j}(D^{\textbf{n} - 2 \textbf{k}}) \otimes \langle \textbf{s} \rangle$ defined by $S \otimes \textbf{s}$ $\mapsto$ $d(S) \otimes \textbf{s}$, which is denoted by $d^{i, j}_{\textbf{s}}$.  By this definition, $d^{i+1, j}_{\textbf{s}} \circ d^{i, j}_{\textbf{s}}$ $=$ $0$ since $d_{\textbf{s}}^{i+1, j}(d_{\textbf{s}}^{i, j} (S \otimes \textbf{s}))$ $=$ $d^{2}(S) \otimes \textbf{s}$ $=$ $0$.  

For the tuple of nonnegative integers $\textbf{k}$ $=$ $(k_1, k_2, \dots, k_l)$, let $I_{k}$ be the set of $|{\textbf{k}}|$-pairings, where each $k_i$ is the number of edges in the $i$-th vertical line of its pairing $\textbf{s}$.  The homomorphism $d''^{k, i, j}$ is defined by setting $d''^{k, i, j} :=$ $(-1)^{k} \oplus_{{\textbf{s}} \in I_{k}} d^{i, j}_{\textbf{s}}$ on $\oplus_{\textbf{s} \in I_k} C^{i, j}(D^{\textbf{n} - 2 \textbf{k}}) \otimes \textbf{s}$.  
\begin{lemma}
$d''^{k, i, j}$ is a coboundary operator on $C_{\textbf{n}}^{k, i, j}(D)$.  
\end{lemma}
\begin{proof}
By definition, $d''^{k, i, j}$ is a homomorphism from $\oplus_{\textbf{s} \in I_k} C^{i, j}(D^{\textbf{n} - 2 \textbf{k}}) \otimes \langle \textbf{s} \rangle$ $\to$ $\oplus_{\textbf{s} \in I_k} C^{i+1, j}(D^{\textbf{n} - 2 \textbf{k}}) \otimes \langle \textbf{s} \rangle$.  Then, from the definitions of $\Gamma_{\textbf{n}}^{k}(D)$ and $C_{\textbf{n}}^{k, i, j}(D)$, we have the isomorphism and the equality: 
\[\oplus_{\textbf{s} \in I_k} C^{i, j}(D^{\textbf{n} - 2 \textbf{k}}) \otimes \langle \textbf{s} \rangle \simeq C^{i, j}(D^{\textbf{n} - 2 \textbf{k}}) \otimes \Gamma_{\textbf{n}}^{k}(D) \stackrel{\rm{\text{def}}}{=} C_{\textbf{n}}^{k, i, j}(D).\]
The homomorphism $(-1)^{k} \oplus_{\textbf{s} \in I_k} d_{\textbf{s}^{i, j}}$ is then homomorphism $C_{\textbf{n}}^{k, i, j}(D)$ $\to$ 

$C_{\textbf{n}}^{k, i+1, j}(D)$.  Noting that $d''^{k, i+1, j} \circ d''^{k, i, j}$ $=$ $\oplus_{\textbf{s} \in I_k} d_{\textbf{s}}^{i+1, j} \circ d_{\textbf{s}}^{i, j}$ $=$ $0$, $d''^{k, i, j}$ is a coboundary operator from $C_{\textbf{n}}^{k, i, j}(D)$ to $C_{\textbf{n}}^{k, i+1, j}$.  
\end{proof}

We now have two coboundary operators $d'^{k, i, j}$ (see Sec. \ref{def_diff_cable}) and $d''^{k, i, j}$ on $C_{\textbf{n}}^{k, i, j}(D)$.  For these differentials $d''^{k, i, j}$ and $d'^{k, i, j}$, $d''^{k+1, i, j} \circ d'^{k, i, j}$ $+$ $d'^{k, i+1, j} \circ d''^{k, i, j}$ $=$ $0$ (Prop. \ref{bicomplex_commute}).  This shows the existence of the desired bicomplex.  
\begin{proposition}\label{bicomplex_commute}
$d''^{k+1, i, j} \circ d'^{k, i, j}$ $+$ $d'^{k, i+1, j} \circ d''^{k, i, j}$ $=$ $0$.  
\end{proposition}
\begin{proof}
Let $S$ be an enhanced state and $\textbf{s}$ be a pairing.  The set of generators of $C^{k, i, j}(D)$ is taken as $\{S \otimes \textbf{s}\}$.  

In only this proof, we use the notations $d'(S)$ and $d''(S)$.  By the definitions of $d'^{k, i, j}$ and $d''^{k, i, j}$, $d'^{k, i, j}(S \otimes \textbf{s})$ can be written as $d'(S) \otimes d'(\textbf{s})$, and $d''(S \otimes \textbf{s})$ can be written as $d''(S) \otimes d'(\textbf{s})$, where $d'(\textbf{s})$ is as defined in (\ref{z_pairing_dif}) in Sec. \ref{extend_link}.  Using this notation, we notice that it is sufficient to show that $d'' d' S + d' d'' S$ $=$ $0$ for an enhanced state $S$, which is a generator of $C^{i, j}(D^{\textbf{n} - 2 \textbf{k}})$ with $|{\textbf{k}}|$ $=$ $k$.  This is because $(d''^{k+1, i, j} \circ d'^{k, i, j} + d'^{k, i+1, j} \circ d''^{k, i, j})(S \otimes \textbf{s})$ $=$ $(d'' \circ d' + d' \circ d'')(S) \otimes d'(\textbf{s})$.  

We first look at an enhanced state $S$ to prove the relation $d'' d' S$ $+$ $d' d'' S$ $=$ $0$ by considering the following three cases: 

Case 1: $S$ is neither Type $1_j$ nor $2_j$.  In this case, the definition of $d'$ implies $d'' d' S$ $=$ $0$.  By looking at the definition of Type $1_j$ and Type $2_j$, we can check that $d'' S$ is the sum of enhanced states of the form $d'' S$ $=$ $\sum_{m} \pm S_{0, m} + \sum_n \pm (S_{1, n} + S_{2, n})$ $+$ $\sum_{n'} (S_{1, n'} + S_{2, n'})$, where for each $n$ (for each $n'$) and $i \in \{5, 6, 7, 8\}$, $S_{1, n}$ ($S_{1, n'}$) is an enhanced state of $\oplus$-term ((a-$i$)) and $S_{2, n}$ ($S_{2, n'}$) is the enhanced state of $\ominus$-term ((b-$i$)), and for each $m$, $S_{0, m}$ is the enhanced state that is neither Type $1_j$ nor Type $2_j$.  Again using the definition of $d'$ and setting $T_{1, n}$ $=$ ${d'}(S_{1, n})$, we obtain $d'(\sum_n (S_{1, n} + S_{2, n}))$ $=$ $\sum_n$ $(T_{1, n} - T_{1, n})$ $=$ $0$, as seen in Lemmas \ref{lemma_type1_lem1}--\ref{lemma_type1_lem3}.  On the other hand, for $S_{1, n'}$, there exists an enhanced state $U$ (see Fig. \ref{pre_type12}) such that $d''(U)$ $=$ $S_{1, n'}$ $+$ $S_{2, n'}$ $+$ $\sum_{m'} \pm S_{0, m'}$.  Using the definition of $\rho$ and (\ref{map_rho_d}), we get $\rho(S_{1, n'} + S_{2, n'})$ $=$ $\rho(d''(U))$ $=$ $0$.  Then, using the definition of $d'$, we get $d'(\sum_{n'} (S_{1, n'} + S_{2, n'}))$ $=$ $\pm \rho\left(\sum_{n'} {d'_{1}}(S_{1, n'}) + {d'_{1}}(\rho(S_{2, n'})) \right)$, where $d'_1$ is defined as in Sec. \ref{cable_z}.  We can see that $\rho d'_{1} \rho(S_{1, n'})$ $=$ $\rho d'_{1}(S_{1, n'})$, and so, $\rho \left({d'_{1}}(S_{1, n'}) + {d'_{1}}(\rho(S_{2, n'})) \right)$ $=$ $\rho {d'_{1}}(\rho(S_{1, n'} + S_{2, n'}))$ $=$ $\rho {d'_{1}}(\rho(S_{1, n'} + S_{2, n'} + \sum_{m'} S_{0, m'}))$ $=$ $\rho {d'_{1}}(\rho(d''(U)))$ $=$ $\rho {d'_{1}}(0)$ $=$ $0$.  By definition of $d'$, $d'(S_{0, m})$ $=$ $0$; thus, from the above discussion, $d' d'' S$ $=$ $d'(\sum_{m} \pm S_{0, m}$ $+$ $\sum_{n} \pm (S_{1, n}$ $+$ $S_{2, n})$ $+$ $\sum_{n'} \pm (S_{1, n'}$ $+$ $S_{2, n'}))$ $=$ $0$.  

Case 2: $S$ is Type $1_j$.  In this case, we can write $d'' S$ as the sum $d'' S$ $=$ $d_1'' S$ $+$ $d_2'' S$ ($\sharp$), where $d_1'' S$ consists of all terms in $d'' S$ that are obtained from $S$ by changing the resolution of a crossing between two non-contracted strands, and $d''_{2} S$ consists of all other terms.  Since $S$ is Type $1_j$, none of the terms in $d_2'' S$ can be Type $1_j$ or $2_j$, and hence, $d' d_2'' S$ $=$ $0$.  On the other hand, because $d'$ is given by removing contracted circles (and multiplying by a sign), we have $d' d_1'' S$ $=$ $- d'' d' S$ (since $d''$ preserves the sign $\prod_k \epsilon_k$ of Type $1_j$ defined in Def. \ref{plus_term} and the orientation of negative markers of Type $1_j$ was fixed (Def. \ref{alternative_marker})).  Hence, $d' d'' S$ $=$ $- d'' d' S$.  

Case 3: $S$ is (b-1)--(b-4) of Type $2_j$.  In this case, the definition of $d'$ implies that $d' S$ $=$ $\epsilon \rho d'_1 f S$, where $\epsilon$ is $1$ ($-1$) if $S$ is $\oplus$-term ($\ominus$-term).  First, we consider $S$ to be $\oplus$-term of Type $2_j$.  In this case, $f S$ is $\ominus$-term (see Tables \ref{pre_type1_table1}--\ref{pre_type1_table12}).  Before we show $d'' d' S + d' d'' S$ $=$ $0$, we define some symbols.  We can write $d'' S$ as the sum $d'' S$ $=$ $d_2'' S$ $+$ $d_0'' S$ ($\star$), where $d_2'' S$ consists of all terms in $d'' S$ that are obtained from $S$ by changing the resolution of a crossing between two non-contracted strands, and $d''_0 S$ consists of all other terms.  Since $S$ is (b-1)--(b-4) of Type $2_j$, $d''_2 S$ is also (b-1)--(b-4) of Type $2_j$ and none of the terms in $d_0'' S$ can be Type $1_j$ or $2_j$, and hence, $d' d_0'' S$ $=$ $0$ ($\star\star$).  We now start to show the desired result.  
\begin{equation}
\begin{split}
d'' d' S &= d'' \rho d'_1 f S \quad {\text{(definition of $d'$)}}\\ 
&= - d'' d' (f S) \quad {\text{($fS$ : $\ominus$-term and Type $1_j$)}}\\
&= d' d'' (f S) \quad {\text{(Case 2)}}\\
&= d' (d''_1 (f S)) \quad {\text{(Case 2)}}\\ 
&= - \rho d'_1 f (d''_1 f S) \quad {\text{($d''_1 f S$ : $\ominus$-term and Type $1_j$)}}\\
&= - \rho d'_1 (d''_1 f S) \quad {\text{(definition of $f$)}}\\
&= - \rho d'_1 (d'' f S) \quad {\text{(($\sharp$) and definition of $d'_{1}$)}}\\
&= - \rho d'_1 (f d'' S) \quad {\text{(Lemma \ref{preserve_lemma2})}}\\
&= - \rho d'_1 f (d''_2 S + d''_0 S) \quad {\text{(using ($\star$))}}\\
&= - \rho d'_1 (f d''_2 S + d''_0 S)  \quad {\text{(using the property $f$)}} \\
&= - d' d'' S \quad {\text{($d'' S$ : $\oplus$-terms which is survived by $d'$) and ($\star\star$)}}.\\
\end{split}
\end{equation}
Next, we consider the case in which $S$ is $\ominus$-term.  In this case, $f S$ is $\oplus$-term.  
\begin{equation}
\begin{split}
d'' d' S &=  - d'' \rho d'_1 f S \quad {\text{(definition of $d'$)}} \\
&= - d'' d' (f S) \quad {\text{($f S$ : $\oplus$-term and Type $1_j$)}}\\
&= d' d'' (f S) \quad {\text{(Case 2)}}\\
&= d' d''_1 (f S) \quad {\text{(Case 2)}}\\
&= \rho d'_1 f (d''_1 f S) \quad {\text{($d''_1 f S$ : $\oplus$-term and Type $1_j$)}}\\
&= \rho d'_1 (d''_1 f S) \quad {\text{(definition of $f$)}}\\
&= \rho d'_1 (d'' f S) \quad {\text{($\sharp$) and definition of $d'_1$}}\\
&= \rho d'_1 (f d'' S) \quad {\text{(Lemma \ref{preserve_lemma2})}}\\
&= \rho d'_1 f (d''_2 S + d''_0 S) \quad {\text{(using ($\star$))}}\\
&= \rho d'_1 (f d''_2 S + d''_0 S) \quad {\text{(using the property $f$)}}\\
&=  - d' d'' S \quad {\text{($d'' S$ : $\ominus$-terms will be survived by $d'$ and ($\star\star$))}}.
\end{split}
\end{equation}

Case 4: $S$ is (b-5)--(b-8) of Type $2_j$.  Before we begin the proof of $d' d'' S + d'' d' S$ $=$ $0$, we define some additional symbols.  By the property of $\rho$, we can write $\rho S$ as the sum $\rho S$ $=$ $\rho_1 S$ $+$ $\rho_2 S$ ($\diamond$), where $\rho_1 S$ consists of all terms in $\rho S$ that Type $1_j$ in $\oplus_{i, j} C^{i, j}(S_{\rho})$ and $\rho_2 S$ consists of the other terms.  Then, by definitions of $d'_1$ and $\rho$, we have $d'_1 (\rho_2 S)$ $=$ $0$ ($\diamond\diamond$) and $d_1' (d'' \rho_2 S))$ $=$ $0$ ($\diamond\diamond\diamond$).  
\begin{equation}
\begin{split}
d'' d' S &= d'' \rho d_1' \rho S \quad {(\text{definition of $d'$})}\\
&= d'' \rho d'_1 (\rho_1 S + \rho_2 S) \quad {\text{($\diamond$)}}\\
&= d'' \rho d'_1 \rho_1 S \quad {\text{($\diamond\diamond$)}}\\
&= \rho d'' d'_1 \rho_1 S \quad{\text{(\ref{chain_rho})}}\\
&= - \rho d'_1 d'' \rho_1 S \quad {\text{(\ref{eq-final})}}\\
&= - \rho d'_1 d'' (\rho_1 S + \rho_2 S)  \quad {\text{($\diamond\diamond\diamond$)}}\\
&= - \rho d'_1 d'' (\rho S) \quad {\text{($\diamond$)}}\\
&= - \rho d'_1 \rho (d'' S) \quad {\text{(\ref{chain_rho})}}\\
&= - d' d'' S \quad {\text{(only the terms of Type $2_j$ are survived by $\rho$)}}
\end{split}
\end{equation}
where
\begin{equation}\label{eq-final}
\begin{split}
&d'' d' (\rho_1 S) = - d'_1 d'' (\rho_1 S) \quad {\text{($\rho$: Type $1_j$ and Case 2)}}\\
\Longleftrightarrow &d'' \epsilon \rho d'_1 f (\rho_1 S) = - \epsilon \rho d'_1 f (d'' \rho_1 S) \quad {\text{(definition of $d'$)}}\\
\Longleftrightarrow &d'' \rho d'_1 (\rho_1 S) = - \rho d'_1 d'' (\rho_1 S) \quad {\text{(definition of $f$)}}\\
\Longleftrightarrow &\rho d'' d'_1 (\rho_1 S) = - \rho d'_1 d'' (\rho_1 S).  
\end{split}
\end{equation}
The following lemma (Lemma \ref{preserve_lemma2}) completes the proof of Theorem \ref{bicomplex_theorem}.  
\end{proof}

\begin{lemma}\label{preserve_lemma2}
The formula $d'_1 d'' f S$ $=$ $d'_1 f d'' S$ for $S$ : {\rm{(b-1)--(b-4)}} : Type $2_j$.  
\end{lemma}
\begin{proof}
Let $S$ be an enhanced state of (b-1)--(b-4) of Type $2_j$.  Then, $d'' S$ $=$ $\sum_{m} \pm S_{0, m}$ $+$ $\sum_{n} \pm S_{2, n}$ where for each $n$, $S_{2, n}$ is an enhanced state of Type $2_j$, and for each $m$, $S_{0, m}$ is an enhanced state that is neither Type $1_j$ nor Type $2_j$.  Then, $d'_1 f d'' S$ $=$ $\sum_{n} \pm (d'_1 f S_{2, n})$.  On the other hand, $d'' f S$ $=$ $\sum_{0, m'} S_{0, m'}$ $+$ $\sum_{n'} \pm (fS)_{1, n'}$ where for each $n'$ above, $(fS)_{1, n'}$ is an enhanced state of Type $1_j$, and for each $m'$, $S_{0, m'}$ is an enhanced state that is neither Type $1_j$ nor Type $2_j$.  Then, $d'_1 d'' f S$ $=$ $\sum_{n'} \pm d'_1 (fS)_{1, n'}$.  The terms $S_{0, m}$ and $S_{0, m'}$ were changed at one marker of a crossing between one contracted strand and another strand.  $S_{2, n}$ or $(fS)_{1, n'}$ was then produced from Type $2_j$ by changing one marker of a crossing between both two non-contracted strands.  The situation is illustrated in Fig. \ref{type1_type2}, and the correspondences are confirmed from Tables \ref{pre_type1_table1}--\ref{pre_type1_table16}.  Then, by changing the order of $\{n'\}$, if necessary, we can assume $n'$ $=$ $n$, and under this condition, we have $(fS)_{1, n}$ $=$ $f(S_{2, n})$; hence, $d'_1 d'' f S$ $=$ $d'_1 f d'' S$.  
\end{proof}
\begin{figure}
\begin{picture}(0,0)
\put(102,-100){$fS$}
\put(222,-100){$S$}
\end{picture}
\begin{center}
\qquad\quad \includegraphics[width=8cm,bb=0 0 300 102]{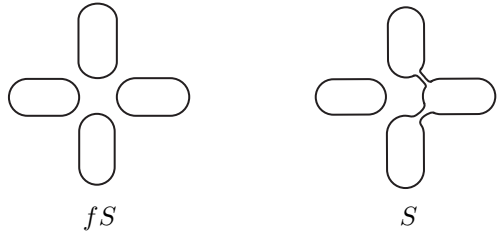}
\end{center}
\caption{$fS$ : Type $1_j$ and $S$ : Type $2_j$ in the case of (a-1) and (b-1) and the case of Fig. \ref{table_type} (a), where the two exteriors of these neighborhoods are the same.  The pictures in this figure are preserved even after changing the smoothing between non-contracted strands.  }\label{type1_type2}
\end{figure}

\begin{remark}
If we try to show the invariance of $H^{k}(H^{i}(C_{\bf{n}}^{*, *, j}(D), d''), d')$ under Reidemeister moves, we desire to have the map $d'* : $$H^{i}(C_{\bf{n}}^{k, *, j}(D)) \to H^{i}(C_{\bf{n}}^{k+1, *, j}(D))$ which is fixed under these moves.  We often describe each of Reidemeister moves as a composition of an inclusion map and a retraction (for the first move, see \cite[Page 336]{viro}, and for the second and third moves, see \cite[Formulae (2.2), (2.6)]{ito3}), denoted here by $\operatorname{in} \circ \rho_i$ $i$ $=$ $1$, $2$, and $3$, respectively.  Then, at least, we need the commutativity of $d'$ and $\operatorname{in} \circ \rho_i$ up to chain homotopy.  The map $d'$ essentially consists of $\rho$ and $f$.  Since $\rho$ is $\rho_2$, one can show the commutativity $\operatorname{in} \circ \rho_i$ and $\rho$ for each $i$ up to chain homotopy.  For the map $f$, though this map $f$ sends Type $1_j$ to Type $2_j$, which seems to induce some relation between $S_{+-}(p, q)$ and $S_{-+}(p, q)$ of (\ref{def_rho2}) in some cases, more research of the property of $d'$ is required to contribute to demonstrating the Reidemeister invariance.  
\end{remark}

\subsection*{Acknowledgements}
The author was a Research Fellow of the Japan Society for the Promotion of Science (20$\cdot$935).  This work was partially supported by IRTG 1529, a Waseda University Grant for Special Projects (2010A-863), and JSPS KAKENHI Grant Number 20$\cdot$935, 23740062.  The author would like to thank the referee for useful comments on the early version, particularly those relating to Proposition \ref{bicomplex_commute}.  The author also thanks Gregory Mezera and Professor Okihiro Sawada for their fruitful comments on the presentation of this paper.

\clearpage
\appendix
\section{Tables of Lemmas \ref{lemma_type1_lem1}--\ref{lemma_type1_lem3}}
\begin{figure}[h!]
\begin{center}
\qquad \includegraphics[width=10cm,bb=0 0 308.5 89]{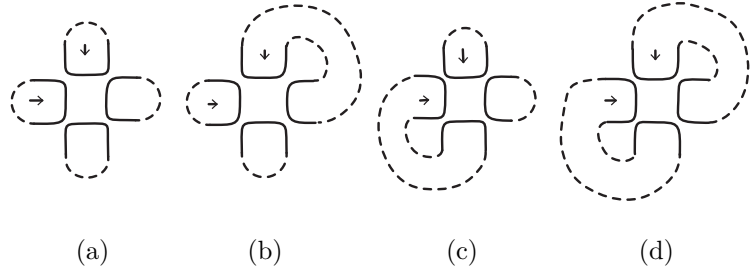}
\end{center}
\begin{picture}(0,0)
\put(65,0){(a)}
\put(130,0){(b)}
\put(205,0){(c)}
\put(278,0){(d)}
\end{picture}
\caption{Types (a-$i$) of Table \ref{pre_type1_table1}--\ref{pre_type1_table16}.  (a) : Tables \ref{pre_type1_table1}, \ref{pre_type1_table5}, \ref{pre_type1_table9}, and, \ref{pre_type1_table13}.  (b) : \ref{pre_type1_table2}, \ref{pre_type1_table6}, \ref{pre_type1_table10}, and \ref{pre_type1_table14}.  (c) : \ref{pre_type1_table3}, \ref{pre_type1_table7}, \ref{pre_type1_table11}, and \ref{pre_type1_table15}.  (d) : \ref{pre_type1_table4}, \ref{pre_type1_table8}, \ref{pre_type1_table12}, and \ref{pre_type1_table16}.  The arrows indicate the direction of contracted strands.}\label{table_type}
\end{figure}
\begin{table}[h!]
\begin{center}
\begin{tabular}{|c|c|c|} \hline
$S$ &$\oplus$ (a-1) : $p \otimes 1 \otimes x \otimes q$  &$\ominus$ (b-1) : $p \otimes q (= r)$ \\  
&$\oplus$ (a-2) : $p \otimes x \otimes 1 \otimes q$&\\ \hline
$a \otimes b \otimes c (= d)$ & (a-1) or (a-2) in $a \otimes b \otimes \Delta(c)$ & $a \otimes m(b \otimes c)$ \\ \hline
$x \otimes 1 \otimes x$ & (a-1) : $x \otimes 1 \otimes x \otimes x$ & (b-1) : $x \otimes x$ \\ \hline
$x \otimes 1 \otimes 1$ & (a-1) : $x \otimes 1 \otimes x \otimes 1$& (b-1) : $x \otimes 1$ \\ \hline
$1 \otimes 1 \otimes x$ & (a-1) : $1 \otimes 1 \otimes x \otimes x$ & (b-1) : $1 \otimes x$ \\ \hline
$1 \otimes 1 \otimes 1$ & (a-1) : $1 \otimes 1 \otimes x \otimes 1$ & (b-1) : $1 \otimes 1$ \\ \hline
$x \otimes x \otimes 1$ &(a-2) : $x \otimes x \otimes 1 \otimes x$ &(b-1) : $x \otimes x$ \\ \hline
$1 \otimes x \otimes 1$ &(a-2) : $1 \otimes x \otimes 1 \otimes x$ &(b-1) : $1 \otimes x$ \\ \hline
$x \otimes x \otimes x$ & none &$0$ \\ \hline
$1 \otimes x \otimes x$ & none & $0$ \\ \hline
\end{tabular}
\end{center}
\caption{(a-1) and (a-2) from Fig. \ref{table_type} (a).  The table shows that $d(S)$ $=$ (a-$i$) $+$ (b-1) $+$ other terms for $i$ $=$ 1 or 2 and for $S$ as in the left figure of Fig. \ref{pre_type1}.}\label{pre_type1_table1}
\end{table}
\begin{table}
\begin{center}
\begin{tabular}{|c|c|c|} \hline
$S$ &$\oplus$ (a-1) : $(p =) 1 \otimes x \otimes q$ &$\ominus$ (b-1) : $(p = r =) q$ \\ 
&$\oplus$ (a-2) : $(p =) x \otimes 1 \otimes q$ & \\ \hline
$(a =) b \otimes c (= d)$ & (a-1) or (a-2) in $a \otimes \Delta(c)$ & $m(a \otimes c)$ \\ \hline
$x \otimes 1$ & (a-2) : $x \otimes 1 \otimes x$ & (b-1) : $x$ \\ \hline
$1 \otimes x$ & (a-1) : $1 \otimes x \otimes x$ & (b-1) : $x$\\ \hline
$1 \otimes 1$ & (a-1) : $1 \otimes x \otimes 1$ & (b-1) : $1$ \\ \hline
$x \otimes x$ & none & $0$ \\ \hline
\end{tabular}
\end{center}
\caption{(a-1) and (a-2) from Fig. \ref{table_type} (b).  The table shows that $d(S)$ $=$ (a-$i$) $+$ (b-1) $+$ other terms for $i$ $=$ 1 or 2 and for $S$ as in the left figure of Fig. \ref{pre_type1}.}\label{pre_type1_table2}
\end{table}
\begin{table}
\begin{center}
\begin{tabular}{|c|c|c|} \hline
$S$&$\oplus$ (a-1) : $p \otimes 1 \otimes x (= q)$&$\ominus$ (b-1) : $p \otimes r \otimes q$\\ 
&$\oplus$ (a-2) : $p \otimes  x \otimes 1 (= q)$ & \\ \hline
$a \otimes b \otimes c \otimes d$&(a-1) or (a-2) in $a \otimes b \otimes m(c \otimes d)$& $a \otimes m(b \otimes c) \otimes d$ \\ \hline
$x \otimes 1 \otimes x \otimes 1$&(a-1) : $x \otimes 1 \otimes x$&(b-1) : $x \otimes x \otimes 1$ \\ \hline
$x \otimes 1 \otimes 1 \otimes x$&(a-1) : $x \otimes 1 \otimes x$&(b-1) : $x \otimes 1 \otimes x$ \\ \hline
$1 \otimes 1 \otimes x \otimes 1$&(a-1) : $1 \otimes 1 \otimes x$&(b-1) : $1 \otimes x \otimes 1$ \\ \hline
$1 \otimes 1 \otimes 1 \otimes x$&(a-1) : $1 \otimes 1 \otimes x$&(b-1) : $1 \otimes 1 \otimes x$ \\ \hline
$x \otimes x \otimes 1 \otimes 1$&(a-2) : $x \otimes x \otimes 1$&(b-1) : $x \otimes x \otimes 1$ \\ \hline
$1 \otimes x \otimes 1 \otimes 1$&(a-2) : $1 \otimes x \otimes 1$&(b-1) : $1 \otimes x \otimes 1$ \\ \hline
$x \otimes x \otimes x \otimes x$&none&$0$ \\ \hline
$x \otimes x \otimes x \otimes 1$&none&$0$ \\ \hline
$x \otimes x \otimes 1 \otimes x$&none&$x \otimes x \otimes x$ \\ \hline
$x \otimes 1 \otimes x \otimes x$&none&$x \otimes x \otimes x$ \\ \hline
$x \otimes 1 \otimes 1 \otimes 1$&none&$x \otimes 1 \otimes 1$ \\ \hline
$1 \otimes x \otimes x \otimes x$&none&$0$ \\ \hline
$1 \otimes x \otimes x \otimes 1$&none&$0$ \\ \hline
$1 \otimes x \otimes 1 \otimes x$&none&$1 \otimes x \otimes x$ \\ \hline
$1 \otimes 1 \otimes x \otimes x$&none&$1 \otimes x \otimes x$ \\ \hline
$1 \otimes 1 \otimes 1 \otimes 1$&none&$1 \otimes 1 \otimes 1$ \\ \hline
\end{tabular}
\end{center}
\caption{(a-1) and (a-2) from Fig. \ref{table_type} (c).  The table shows that $d(S)$ $=$ (a-$i$) $+$ (b-1) $+$ other terms for $i$ $=$ 1 or 2 and for $S$ as in the left figure of Fig. \ref{pre_type1}.}\label{pre_type1_table3}
\end{table}
\begin{table}
\begin{center}
\begin{tabular}{|c|c|c|}\hline
$S$ &$\oplus$ (a-1) : $(p =) 1 \otimes x (= q)$ &$\ominus$ (b-1) : $(r =) p \otimes q$\\ 
&$\oplus$ (a-2) : $(p =) x \otimes 1 (= q)$ &  \\ \hline
$(a =) b \otimes d \otimes c$ & (a-1) or (a-2) in $a \otimes m(d \otimes c)$ & $m(a \otimes c) \otimes d$ \\ \hline
$x \otimes 1 \otimes 1$ &(a-2) : $x \otimes 1$&(b-1) : $x \otimes 1$ \\ \hline
$1 \otimes x \otimes 1$&(a-1) : $1 \otimes x$&(b-1) : $1 \otimes x$ \\ \hline
$1 \otimes 1 \otimes x$&(a-1) : $1 \otimes x$&(b-1) : $x \otimes 1$ \\ \hline
$x \otimes x \otimes x$& none & $0$ \\ \hline
$x \otimes x \otimes 1$&none&$x \otimes x$ \\ \hline
$x \otimes 1 \otimes x$ &none&$0$ \\ \hline
$1 \otimes x \otimes x$&none&$x \otimes x$ \\ \hline
$1 \otimes 1\otimes 1$&none &$1 \otimes 1$ \\ \hline
\end{tabular}
\end{center}
\caption{(a-1) and (a-2) from Fig. \ref{table_type} (d).  The table shows that $d(S)$ $=$ (a-$i$) $+$ (b-1) $+$ other terms for $i$ $=$ 1 or 2 and for $S$ as in the left figure of Fig. \ref{pre_type1}.}\label{pre_type1_table4}
\end{table}
\begin{table}
\begin{center}
\begin{tabular}{|c|c|c|} \hline
$S$ &$\oplus$ (a-1) : $q \otimes p \otimes 1 \otimes x$ &$\ominus$ (b-2) : $p \otimes q$  \\ 
&$\oplus$ (a-2) : $q \otimes p \otimes x \otimes 1$ & \\\hline
$a \otimes b (=d) \otimes c$& (a-1) or (a-2) in $a \otimes \Delta(b) \otimes c$ & $a \otimes m(b \otimes c)$  \\ \hline
$x \otimes 1 \otimes x$ & (a-1) : $x \otimes x \otimes 1 \otimes x$ & (b-2) : $x \otimes x$ \\ \hline
$1 \otimes 1 \otimes x$ & (a-1) : $1 \otimes x \otimes 1 \otimes x$ & (b-2) : $1 \otimes x$  \\ \hline
$x \otimes x \otimes 1$ & (a-2) : $x \otimes x \otimes x \otimes 1$ & (b-2) : $x \otimes x$ \\ \hline
$x \otimes 1 \otimes 1$ & (a-2) : $x \otimes 1 \otimes x \otimes 1$ & (b-2) : $x \otimes 1$ \\ \hline
$1 \otimes x \otimes 1$ & (a-2) : $1 \otimes x \otimes x \otimes 1$ & (b-2) : $1 \otimes x$ \\ \hline
$1 \otimes 1 \otimes 1$ & (a-2) : $1 \otimes 1 \otimes x \otimes 1$ & (b-2) : $1 \otimes 1$\\ \hline
$x \otimes x \otimes x$ & none & $0$ \\ \hline
$1 \otimes x \otimes x$ & none & $0$ \\ \hline
\end{tabular}
\end{center}
\caption{(a-1) and (a-2) from Fig. \ref{table_type} (a).  The table shows that $d(S)$ $=$ (a-$i$) $+$ (b-2) $+$ other terms for $i$ $=$ 1 or 2 and for $S$ as in the right figure of Fig. \ref{pre_type1}.}\label{pre_type1_table5}
\end{table}
\begin{table}
\begin{center}
\begin{tabular}{|c|c|c|} \hline
$S$ &$\oplus$ (a-1) : $q \otimes (p =) 1 \otimes x$ &$\ominus$ (b-2) :  $q \otimes b \otimes p$\\
&$\oplus$ (a-2) : $q \otimes (p =) x \otimes 1$& \\ \hline
$a \otimes b \otimes d \otimes c$ & (a-1) or (a-2) in $a \otimes m(b \otimes d) \otimes c$ & $a \otimes b \otimes m(d \otimes c)$ \\ \hline
$x \otimes 1 \otimes 1 \otimes x$ &(a-1) : $x \otimes 1 \otimes x$ &(b-2) : $x \otimes 1 \otimes x$ \\ \hline
$1 \otimes 1 \otimes 1 \otimes x$ &(a-1) : $1 \otimes 1 \otimes x$&(b-2) : $1 \otimes 1 \otimes x$ \\ \hline
$x \otimes 1 \otimes x \otimes 1$ &(a-2) : $x \otimes x \otimes 1$&(b-2) : $x \otimes 1 \otimes x$ \\ \hline
$1 \otimes x \otimes 1 \otimes 1$ &(a-2) : $1 \otimes x \otimes 1$&(b-2) : $1 \otimes x \otimes 1$ \\ \hline
$1 \otimes 1 \otimes x \otimes 1$ &(a-2) : $1 \otimes x \otimes 1$&(b-2) : $1 \otimes 1 \otimes x$ \\ \hline
$x \otimes x \otimes 1 \otimes 1$&(a-2): $x \otimes x \otimes 1$&$x \otimes x \otimes 1$ \\ \hline
$x \otimes x \otimes x \otimes x$&none&$0$ \\ \hline
$x\otimes x \otimes 1 \otimes x$&none&$x \otimes x \otimes x$ \\ \hline
$x \otimes 1 \otimes x \otimes x$&none&$0$ \\ \hline
$x \otimes x \otimes x \otimes 1$&none&$x \otimes x \otimes x$ \\ \hline
$x \otimes 1 \otimes 1 \otimes 1$&none&$x \otimes 1 \otimes 1$ \\ \hline
$1 \otimes x \otimes x \otimes x$&none&$0$ \\ \hline
$1 \otimes x \otimes 1 \otimes x$&none&$1 \otimes x \otimes x$ \\ \hline
$1 \otimes 1 \otimes x \otimes x$&none&$0$ \\ \hline
$1 \otimes x \otimes x \otimes 1$&none&$1 \otimes x \otimes x$ \\ \hline
$1 \otimes 1 \otimes 1 \otimes 1$&none&$1 \otimes 1 \otimes 1$ \\ \hline
\end{tabular}
\end{center}
\caption{(a-1) and (a-2) from Fig. \ref{table_type} (b).  The table shows that $d(S)$ $=$ (a-$i$) $+$ other terms for $i$ $=$ 1 or 2 and for $S$ as in the right figure of Fig. \ref{pre_type1}.}\label{pre_type1_table6}
\end{table}
\begin{table}
\begin{center}
\begin{tabular}{|c|c|c|} \hline
$S$ &$\oplus$ (a-1) : $x (= q) \otimes 1 \otimes p$ &$\ominus$ (b-2) : $p (=q)$ \\ 
&$\oplus$ (a-2) : $1 (= q) \otimes x \otimes p$& \\
\hline
$a (= c) \otimes b (= d)$ & (a-1) or (a-2) in $a \otimes \Delta(b)$ & $m(a \otimes b)$ \\ \hline
$x \otimes 1$ &(a-1) : $x \otimes 1 \otimes x$ &(b-2) : $x$ \\ \hline
$1 \otimes x$ &(a-2) : $1 \otimes x \otimes x$&(b-2) : $x$ \\ \hline
$1 \otimes 1$ &(a-2) : $1 \otimes x \otimes 1$&(b-2) : $1$ \\ \hline
$x \otimes x$ &none&$0$ \\ \hline
\end{tabular}
\end{center}
\caption{(a-1) and (a-2) from Fig. \ref{table_type} (c).  The table shows that $d(S)$ $=$ (a-$i$) $+$ (b-2) $+$ other terms for $i$ $=$ 1 or 2 and for $S$ as in the right figure of Fig. \ref{pre_type1}.}\label{pre_type1_table7}
\end{table}
\begin{table}
\begin{center}
\begin{tabular}{|c|c|c|} \hline
$S$ &$\oplus$ (a-1) : $q (= 1) \otimes x (= p)$&$\ominus$ (b-2) : $p (= q) \otimes b$ \\ 
&$\oplus$ (a-2) : $q (= x) \otimes 1 (= p)$& \\ \hline
$a (= c) \otimes b \otimes d$ &(a-1) or (a-2) in $a \otimes m(b \otimes d)$ &  $m(a \otimes d) \otimes b$ \\ \hline
$x \otimes 1 \otimes 1$ &(a-2) : $x \otimes 1$ &(b-2) : $x \otimes 1$ \\ \hline
$1 \otimes x \otimes 1$ &(a-1) : $1 \otimes x$ &(b-2) : $1 \otimes x$ \\ \hline
$1 \otimes 1 \otimes x$ &(a-1) : $1 \otimes x$ &(b-2) : $x \otimes 1$ \\ \hline
$x \otimes x \otimes x$&none&$0$ \\ \hline
$x \otimes x \otimes 1$&none&$x \otimes x$ \\ \hline
$x \otimes 1 \otimes x$&none&$0$ \\ \hline
$1 \otimes x \otimes x$&none&$x \otimes x$ \\ \hline
$1 \otimes 1 \otimes 1$&none&$1 \otimes 1$ \\ \hline
\end{tabular}
\end{center}
\caption{(a-1) and (a-2) from Fig. \ref{table_type} (d).  The table shows that $d(S)$ $=$ (a-$i$) $+$ (b-2) $+$ other terms for $i$ $=$ 1 or 2 and for $S$ as in the right figure of Fig. \ref{pre_type1}.  }\label{pre_type1_table8}
\end{table}
\begin{table}
\begin{center}
\begin{tabular}{|c|c|c|c|} \hline
$S$ &$\oplus$ (a-3) : $x \otimes p \otimes q \otimes 1$ &$\ast$ (b-3)&$\ast$ (b-4)  \\
&$\ominus$ (a-4) : $1 \otimes p \otimes q \otimes x$&$p \otimes q$&$(r =) p \otimes q$\\ \hline
$a \otimes b\otimes d$& (a-3) or (a-4)  & $m(a \otimes b) \otimes d$ & $a \otimes m(b \otimes d)$ \\
$(b = c)$&in $a \otimes \Delta(b) \otimes d$&&\\ \hline
$x \otimes x \otimes 1$ &$\oplus$ (a-3) : $x \otimes x \otimes x \otimes 1$ & $0$ &$\ominus$ (b-4) : $x \otimes x$ \\ \hline
$x \otimes 1 \otimes 1$ &$\oplus$ (a-3) : $x \otimes 1 \otimes x \otimes 1$ &$\ominus$ (b-3) : $x \otimes 1$ &\\ 
&$\oplus$ (a-3) : $x \otimes x \otimes 1 \otimes 1$&&$\ominus$ (b-4) : $x \otimes 1$ \\ \hline
$1 \otimes x \otimes x$ &$\ominus$ (a-4) : $1 \otimes x \otimes x \otimes x$ &$\oplus$ (b-3) : $x \otimes x$ &$0$ \\ \hline
$1 \otimes 1 \otimes x$ &$\ominus$ (a-4) : $1 \otimes 1 \otimes x \otimes x$ &$\oplus$ (b-3) : $1 \otimes x$ & \\
&$\ominus$ (a-4) : $1 \otimes x \otimes 1 \otimes x$&&$\oplus$ (b-4) : $1 \otimes x$ \\ \hline
$x \otimes x \otimes x$&none&$0$&$0$ \\ \hline
$x\otimes 1 \otimes x$&none&$\oplus$ (b-3) : $x \otimes x$&$\ominus$ (b-4) : $x \otimes x$ \\ \hline
$1 \otimes x \otimes 1$&none&$\ominus$ (b-3) : $x \otimes 1$&$\oplus$ (b-4) : $1 \otimes x$ \\ \hline
$1 \otimes 1 \otimes 1$&none&$1 \otimes 1$&$1 \otimes 1$ \\ \hline
\end{tabular}
\end{center}
\caption{(a-3) and (a-4) from Fig. \ref{table_type} (a).  The table shows that $d(S)$ contains the same number of $\oplus$-terms as $\ominus$-terms for $S$ as in Fig. \ref{pre_type1b}.  }\label{pre_type1_table9}
\end{table}
\begin{table}
\begin{center}
\begin{tabular}{|c|c|c|c|c|} \hline
$S$&$\oplus$ (a-3) : $x \otimes 1 \otimes q$&$\ast$ (b-4) : $q = p$  \\ 
&$\ominus$ (a-4) : $1 \otimes x \otimes q$& \\ \hline
$a \otimes d (= b = c)$ & (a-1) or (a-2) in $a \otimes \Delta(d)$ & $a \otimes \Delta(d)$ \\ \hline
$x \otimes 1$ &$\oplus$ (a-3) : $x \otimes 1 \otimes x$ &$\ominus$ (b-4) : $x \otimes 1 \otimes x$ \\ \hline
$1 \otimes x$ &$\ominus$ (a-4) : $1 \otimes x \otimes x$ &$\oplus$ (b-4) : $1 \otimes x \otimes x$\\ \hline
$1 \otimes 1$ &$\ominus$ (a-4) : $1 \otimes x \otimes 1$ &$\oplus$ (b-4) : $1 \otimes x \otimes 1$ \\ \hline
$x \otimes x$&none&$x \otimes x \otimes x$ \\ \hline
\end{tabular}
\end{center}
\caption{(a-3) and (a-4) from Fig. \ref{table_type} (b).  The table shows that $d(S)$ contains the same number of $\oplus$-terms as $\ominus$-terms for $S$ as in Fig. \ref{pre_type1b}.}\label{pre_type1_table10}
\end{table}
\begin{table}
\begin{center}
\begin{tabular}{|c|c|c|} \hline
$S$ &$\oplus$ (a-3) : $q (= x) \otimes 1 \otimes p$&$\ast$ (b-3) : $p \otimes r \otimes q$\\
&$\ominus$ (a-4) : $q (= 1) \otimes x \otimes p$ & \\ \hline
$a (= b = c) \otimes d$&(a-3) or (a-4) in $a \otimes \Delta(d)$ & $\Delta(a) \otimes d$ \\ \hline
$x \otimes 1$&$\oplus$ (a-3) : $x \otimes 1 \otimes x$&$\ominus$ (b-3) : $x \otimes 1 \otimes x$ \\ \hline
$1 \otimes x$&$\ominus$ (a-4) : $1 \otimes x \otimes x$&$\oplus$ (b-3) : $1 \otimes x \otimes x$ \\ \hline
$1 \otimes 1$&$\ominus$ (a-4) : $1 \otimes 1 \otimes x$&$\oplus$ (b-3) : $1 \otimes 1 \otimes x$ \\ \hline
$x \otimes x$&none&$x \otimes x \otimes x$ \\ \hline
\end{tabular}
\end{center}
\caption{(a-3) and (a-4) from Fig. \ref{table_type} (c).  The table shows that $d(S)$ contains the same number of $\oplus$-terms as $\ominus$-terms for $S$ as in Fig. \ref{pre_type1b}.}\label{pre_type1_table11}
\end{table}
\begin{table}
\begin{center}
\begin{tabular}{|c|c|c|} \hline
$S$ &$\oplus$ (a-3) : $(p =) 1 \otimes x (= q)$&$\ast$ (b-4) : $(q=) p \otimes r$\\
&$\ominus$ (a-4) : $(p =) x \otimes 1 (= q)$& \\ \hline
$a (= b = c = d)$&(a-3) or (a-4) in $\Delta(a)$&(b-4) : $\Delta(a)$ \\ \hline
$1$&$\oplus$ (a-3) : $1 \otimes x$&$\ominus$ (b-4) : $1 \otimes x$\\ 
&$\ominus$ (a-4) : $x \otimes 1$&$\oplus$ (b-4) : $x \otimes 1$ \\ \hline
$x$&none&$x \otimes x$ \\ \hline
\end{tabular}
\end{center}
\caption{(a-3) and (a-4) from Fig. \ref{table_type} (d).  The table shows that $d(S)$ contains the same number of $\oplus$-terms as $\ominus$-terms for $S$ as in Fig. \ref{pre_type1b}.}\label{pre_type1_table12}
\end{table}
\clearpage
\begin{table}
\begin{center}
\begin{tabular}{|c|c|c|} \hline
$S$ &$\oplus$ (a-3) : $p \otimes 1 \otimes x \otimes q$ &$\ominus$ (a-4) : $p \otimes x \otimes 1 \otimes q$ \\ \hline
$a \otimes b \otimes c$ &(a-3) in $a \otimes \Delta(b) \otimes c$ &(a-4) in $a \otimes \Delta(b) \otimes c$\\ \hline
$x \otimes 1 \otimes x$ &$x \otimes 1 \otimes x \otimes x$&$x \otimes x \otimes 1 \otimes x$ \\ \hline
$x \otimes 1 \otimes 1$ &$x \otimes 1 \otimes x \otimes 1$&$x \otimes x \otimes 1 \otimes 1$ \\ \hline
$1 \otimes 1 \otimes x$ &$1 \otimes 1 \otimes x \otimes x$&$1 \otimes x \otimes 1 \otimes x$ \\ \hline
$1 \otimes 1 \otimes 1$ &$1 \otimes 1 \otimes x \otimes 1$&$1 \otimes x \otimes 1 \otimes 1$  \\ \hline
$x \otimes x \otimes x$ &none& none  \\ \hline
$x \otimes x \otimes 1$ &none& none  \\ \hline
$1 \otimes x \otimes x$ &none& none  \\ \hline
$1 \otimes x \otimes 1$ &none& none  \\ \hline
\end{tabular}
\end{center}
\caption{(a-3) and (a-4) are as in Fig. \ref{table_type} (a).  The table shows that $d(S)$ contains the same number of (a-3) as (a-4) for $S$ as in Fig. \ref{pre_type1a}.}\label{pre_type1_table13}
\end{table}
\begin{table}
\begin{center}
\begin{tabular}{|c|c|c|}\hline
$S$ &$\oplus$ (a-3) : $p (= 1) \otimes x \otimes q$ &$\ominus$ (a-4) : $p (= x) \otimes 1 \otimes q$ \\ \hline
$a (= b) \otimes c$ &(a-3) in $\Delta(a) \otimes c$ &(a-4) in $\Delta(a) \otimes c$ \\ \hline
$1 \otimes x$&$1 \otimes x \otimes x$&$x \otimes 1 \otimes x$ \\ \hline
$1 \otimes 1$&$1 \otimes x \otimes 1$&$x \otimes 1 \otimes 1$ \\ \hline
$x \otimes x$ &none&none \\ \hline
$x \otimes 1$ &none&none \\ \hline
\end{tabular}
\end{center}
\caption{(a-3) and (a-4) are as in Fig. \ref{table_type} (b).  The table shows that $d(S)$ contains the same number of (a-3) as (a-4) for $S$ as in Fig. \ref{pre_type1a}.}\label{pre_type1_table14}
\end{table}
\begin{table}
\begin{center}
\begin{tabular}{|c|c|c|c|c|} \hline
$S$ &$\oplus$ (a-3) : $p \otimes (q =) 1\otimes x$ &$\ominus$ (a-4) : $p \otimes (q =) x \otimes 1$   \\ \hline
$a \otimes b (= c)$ &(a-3) in $a \otimes \Delta(b)$ &(a-4) in $a \otimes \Delta(b)$ \\ \hline
$x \otimes 1$ &  $x \otimes 1 \otimes x$ & $x \otimes x \otimes 1$ \\ \hline
$x \otimes 1$  & $x \otimes 1 \otimes x$ & $x \otimes x \otimes 1$ \\  \hline
$1 \otimes x$ &none&none \\ \hline
$1 \otimes 1$ &none&none \\ \hline
\end{tabular}
\end{center}
\caption{(a-3) and (a-4) are as in Fig. \ref{table_type} (c).  The table shows that $d(S)$ contains the same number of (a-3) as (a-4) for $S$ as in Fig. \ref{pre_type1a}.}\label{pre_type1_table15}
\end{table}
\begin{table}
\begin{center}
\begin{tabular}{|c|c|c|} \hline
$S$ &$\oplus$ (a-3) : $(1=) p \otimes q (= x)$&$\ominus$ (a-4) : $(x=) p \otimes q (=1)$\\ \hline 
$a (= b = c)$ &(a-3) in $\Delta(a)$&(a-4) in $\Delta(a)$ \\ \hline
$1$ &$1 \otimes x$& $x \otimes 1$ \\ \hline
$x$ &none&none \\ \hline
\end{tabular}
\end{center}
\caption{(a-3) and (a-4) are as in Fig. \ref{table_type} (d).  The table shows that $d(S)$ contains the same number of (a-3) as (a-4) for $S$ as in Fig. \ref{pre_type1a}.}\label{pre_type1_table16}
\end{table}

\end{document}